\documentclass{article}
\usepackage{graphicx} 
\usepackage{amsmath} 
\usepackage[colorlinks=true,bookmarks=false,linkcolor=blue,urlcolor=blue,citecolor=blue,breaklinks=true]{hyperref}
\usepackage[T1]{fontenc} 
\usepackage{breakcites} 

\usepackage{doi}
\usepackage{graphicx}
\usepackage{amssymb} 
\usepackage{amsfonts} 
\usepackage{amsthm} 
\usepackage[shortlabels]{enumitem} 
\usepackage{comment} 
\usepackage{bbm}
\usepackage{mathtools}
\usepackage{algorithm}
\usepackage{algpseudocode}
\usepackage[margin=1in]{geometry} 
\usepackage{subcaption}
\usepackage{multirow} 
\usepackage{pifont} 
\usepackage{booktabs} 
\usepackage{mathrsfs}
\usepackage{tikz-cd}
\usetikzlibrary{calc, positioning}
\usepackage[numbers,sort&compress]{natbib}

\usepackage{multicol}

\makeatletter
\renewcommand{\paragraph}{%
  \@startsection{paragraph}{4}%
  {\z@}{1ex \@plus 1ex \@minus .2ex}{-.5em}%
  {\normalfont\normalsize\bfseries}%
}
\makeatother

\newtheorem{theorem}{Theorem}[section]
\newtheorem{lemma}[theorem]{Lemma}
\newtheorem{proposition}[theorem]{Proposition}
\newtheorem{corollary}[theorem]{Corollary}
\newtheorem{example}[theorem]{Example}
\newtheorem{definition}[theorem]{Definition}
\newtheorem{remark}[theorem]{Remark}

\newcommand{\be}{\begin{equation}}
\newcommand{\ee}{\end{equation}}
\newcommand{\bthm}{\begin{theorem}}
\newcommand{\ethm}{\end{theorem}}
\newcommand{\blem}{\begin{lemma}}
\newcommand{\elem}{\end{lemma}}
\newcommand{\bpof}{\begin{proof}}
\newcommand{\epof}{\end{proof}}
\newcommand{\bcor}{\begin{corollary}}
\newcommand{\ecor}{\end{corollary}}
\newcommand{\bprop}{\begin{proposition}}
\newcommand{\eprop}{\end{proposition}}

\newcommand{\mc}[1]{\mathcal{#1}}
\newcommand{\mbb}[1]{\mathbb{#1}}
\newcommand{\mfk}[1]{\mathfrak{#1}}
\newcommand{\mscr}[1]{\mathscr{#1}}
\newcommand{\msf}[1]{\mathsf{#1}}

\newcommand\VC[1]{{\color{blue}Venkat: ``#1''}}

\newcommand{\dist}{\mathrm{dist}}

\newcommand{\RR}{\mathbb{R}}

\newcommand{\NN}{\mathbb{N}}

\newcommand{\vct}[1]{\mathbb{#1}}  

\newcommand{\FS}{\mathrm{FinSet}}

\newcommand{\coFS}{\mathrm{FinSet}^{\mathrm{op}}}

\usepackage{authblk}

\title{Any-Dimensional Learning by Sampling}
\author{Eitan Levin$^\dag$ and Venkat Chandrasekaran$^{\ddag}$ \thanks{Emails: \texttt{eitanl@uchicago.edu},  \texttt{venkatc@caltech.edu}} \vspace{0.1in} \\ $^\dag$ Department of Statistics\\ University of Chicago \\ Chicago, IL 60637 \vspace{0.1in}\\ $^\ddag$ Department of Computing and Mathematical Sciences\\ Department of Electrical Engineering \\ California Institute of Technology \\ Pasadena, CA 91125}
\date{September 8, 2026}

\begin{document}

\maketitle

\begin{abstract}

Many machine learning models are defined for inputs of different sizes, such as point clouds containing different numbers of points, sequences of tokens of different lengths, and graphs on different numbers of nodes.  Such models are trained on finitely many examples of necessarily limited sizes. How well do these models generalize from inputs of small size to larger inputs of size not seen during training?  
Furthermore, evaluating such models on large inputs is often expensive.
How can we sketch large inputs to obtain smaller ones on which the model takes similar values?
At the heart of both questions is the need to compare inputs of different sizes and to approximate large inputs by small ones.
We present a unified approach to address these questions by using random sampling maps to compare inputs of different sizes. 
The sampling maps we consider are generalizations of sampling with replacement, random binning, and species sampling. 
We characterize the application domains in which each type of sampling is appropriate in terms of the symmetries and relations between problem instances of different sizes in the domain.
Our framework yields explicit generalization and sketching rates for function classes continuous with respect to a chosen notion of sampling, encompassing large families of functions defined on sequences, graphs, and tensors of different sizes. Specific examples include moment polynomials on measures, homomorphism densities and numbers of graphs, permutation-invariant transformers, and graph neural networks.

\vspace{0.5cm}


\end{abstract} 
\section{Introduction}\label{sec:intro}


Out-of-sample generalization is a core challenge in machine learning, particularly when there are qualitative differences between the training and test data.  An important source for such differences is the dimensionality of the data, which can vary from training to test time in applications involving point clouds containing different numbers of points, graphs on different numbers of vertices, and sequences of tokens of different lengths. In these and other domains, we are given training data consisting of inputs of bounded size, and we aim to learn a function that can be applied at test time to inputs of arbitrary size.  In particular, such a learned function must exhibit \emph{any-dimensional generalization} so that it performs well on inputs of sizes that are not seen during training.


A related problem is one in which we are given a function that is well-defined on inputs of any size, and we wish to evaluate the function on a high-dimensional input.  
When such a high-dimensional function evaluation is expensive, we seek to suitably subsample, or sketch, the input to obtain a low-dimensional object on which the function can be evaluated efficiently and whose value is close to that of the original input.
We call the problem of producing such low-dimensional sketches \emph{any-dimensional sketching}, as our goal is to produce small approximations of arbitrarily sized inputs.



Our objective in this paper is to show that the above two problems of any-dimensional generalization and sketching are closely related to each other, and to develop a systematic approach to tackle them. The fundamental challenge in addressing both of these questions is that of comparing inputs of different sizes.
%
In other words, in what sense can a small object be an approximation to a large one?  Broadly, we develop a sampling-based approach to represent and compare objects of different sizes.  We associate to an object of arbitrary size a sequence of random variables representing increasingly fine subsamples of the object. 
By defining an appropriate metric over such sequences, we then obtain a convenient method for comparing objects of different sizes.  
Building on this sampling-based approach, we identify classes of any-dimensional functions that can be learned from, or approximated on, low-dimensional inputs, and we quantify how small this dimensionality needs to be for a given accuracy.  We also discuss applications to a number of domains of contemporary interest.  We next outline these contributions in more detail.



\subsection{Our Contributions}

\subsubsection*{Formalizing Any-Dimensional Learning (Section~\ref{sec:compact})}

To frame our discussion concretely, let $(\vct V_n)$ be a sequence of vector spaces, with the index $n$ specifying object size, e.g., vectors of length $n$, adjacency matrices of graphs on $n$ vertices, and so on.  We consider the following two running examples, although our framework also encompasses many others:
\begin{itemize}
    \item Sequences of features or tokens, represented as $n$ vectors in $\RR^d$. Here we set $\vct V_n=\RR^{d\times n}$ for a fixed $d \in \mbb N$ and growing $n$.

    \item Weighted and directed graphs on $n$ vertices, represented by their adjacency matrices. Here we set $\vct V_n = \RR^{n\times n}$ and let $n$ grow.
\end{itemize}
In many applications, it is often the case that the objects of interest belong to some proper subset $\Omega_n \subseteq \vct V_n$; for instance, an unweighted graph is precisely one whose adjacency matrix belongs to $\Omega_n = \{0,1\}^{n \times n}$.  In the sequel, we will see that various structural properties of this subset, and indeed of the disjoint union $\bigsqcup_n \Omega_n \subseteq \bigsqcup_n \vct V_n$ of subsets containing objects of all possible sizes, play a central role in our framework.

In order to formulate the any-dimensional generalization problem precisely, let $\Omega_{\leq n} = \bigsqcup_{i\leq n}\Omega_i$ be the set of objects of size at most $n$.  For two functions $f,g \colon \bigsqcup_n \Omega_n \to \RR$ defined on objects of all sizes, i.e., any-dimensional functions, define $\mathrm{e}_n(f,g)=\sup_{x\in\Omega_{\leq n}}|f(x)-g(x)|$ to be the maximum discrepancy between $f$ and $g$ on objects of size $n$, and $\mathrm{e}_{\infty}(f,g)=\sup_n\mathrm{e}_n(f,g)$ to be the maximum discrepancy on objects of any size. By definition, we have $\lim_n\mathrm{e}_n(f,g)=\mathrm{e}_{\infty}(f,g)$, but it is possible that this limit is infinite. We then investigate the following question:

\vspace{0.25cm}

\noindent \emph{Any-dimensional generalization}: Under what conditions on $\bigsqcup_n \Omega_n$ and on any-dimensional functions $f,g : \bigsqcup_n \Omega_n \to \RR$ do we have $\mathrm{e}_{\infty}(f,g)<\infty$? At what rate does $\mathrm{e}_n(f,g)$ converge to $\mathrm{e}_{\infty}(f,g)$?
Are there hypothesis classes of any-dimensional functions over which the convergence rate is uniform?

\vspace{0.25cm}


If $\mathrm{e}_n(f,g)$ converges to a finite $ \mathrm{e}_{\infty}(f,g)$ as $n \rightarrow \infty$, then we can control the error $\mathrm{e}_\infty(f, g)$ between a function $f$ and its estimate $g$ on objects of arbitrarily large size by the error $\mathrm{e}_n(f, g)$ on objects of size at most $n$.  If we further have explicit and uniform rates of convergence over some hypothesis class, then we can obtain \emph{a priori} bounds on the input size $n$ needed to accurately approximate the target function, and precise bounds on $\mathrm{e}_{\infty}(f,g)$ as a function of $\mathrm{e}_n(f,g)$.  We would thereby reduce the problem of approximating a target function on arbitrary-size inputs to approximating it on inputs of a fixed (sufficiently large) size, which belongs to the realm of traditional generalization theory on which there is a substantial literature.



We have stated any-dimensional generalization in terms of uniform error over inputs of each size. If instead we wish to quantify generalization in terms of the average error with respect to some data distribution, it is essential that we formalize the idea of an \emph{any-dimensional data distribution}. Merely considering distributions $\mu$ on $\bigsqcup_n\Omega_n$ is not sufficient, since any such distribution will be concentrated on bounded-dimensional inputs because $\sum_n\mu(\Omega_n)=1$.   Instead, a more appropriate notion for any-dimensional data distributions is that of a weakly-convergent sequence of distributions $(\mu_n\in\mc P(\Omega_{\leq n}))$ supported on increasingly larger objects, as we show in Proposition~\ref{prop:avg_error}.

%

The second question we consider involves approximating the value of an any-dimensional function on a large input by its value on a small sketch of this input.

\vspace{0.25cm}

\noindent \emph{Any-dimensional sketching}: Under what conditions on $\bigsqcup_n \Omega_n$ and on an any-dimensional function $f : \bigsqcup_n \Omega_n \to \RR$ can we find a (potentially random) map $\msf S_k\colon\bigsqcup_n\Omega_n\to\Omega_k$ satisfying 
\begin{equation*}
    \lim_{k\to\infty}\mathrm{e}_{\infty}(f, f\circ\msf S_k)=0\quad \textrm{almost surely?}
\end{equation*}
At what rate does this limit converge?  Are there hypothesis classes of any-dimensional functions over which the convergence rate is uniform?  

\vspace{0.25cm}

As with any-dimensional generalization, if we have explicit and uniform rates of convergence for the above limit over some hypothesis class, then we can provide \emph{a priori} bounds on the sketch size required for a desired accuracy in evaluating a function on inputs of any size. In particular, such uniform rates would yield sketch sizes that do \emph{not} depend on the ambient dimension of the object being sketched. 

Already at this stage, it is clear that any attempt to address the above two questions would require some way to evaluate similarity between objects of different sizes in $\bigsqcup_n \Omega_n$, and for the any-dimensional functions under consideration to be well-behaved with respect to this similarity notion.  
Indeed, if a model performs well on large inputs of unseen size, then these inputs must in some sense be similar to smaller inputs from the training set.  
Similarly, if a function takes on similar values on a low-dimensional sketch of a high-dimensional object, then the object and its sketch must be similar to each other.
More precisely, we need a distance $d$ between elements of $\bigsqcup_n \Omega_n$ that allows us to quantitatively compare objects of different sizes and to consider any-dimensional functions that are continuous in $d$.  We show in Theorem~\ref{thm:general_equivalence} that precompactness of the metric space $(\bigsqcup_n \Omega_n, d)$ plays a central role in addressing the above questions.

\subsubsection*{A Sampling-Based Approach for Comparing Objects of Different Sizes (Section~\ref{sec:sampling_metric})}

To address the preceding questions on any-dimensional learning and sketching, we describe a sampling-based approach to represent and compare objects of different sizes.
Specifically, we fix a collection of random sampling maps $(\msf S_k\colon\bigsqcup_n \vct V_n \to \vct V_k)_{k\in\NN}$, and we associate to each $x \in \vct V_n$ the sequence of random variables
\begin{equation*}
    x \mapsto \left(\msf S_k(x)\right)_{k\in\NN}.
\end{equation*} 
In words, we form random samples of the object $x$ of increasing sizes, which we view as increasingly-fine sketches of $x$.
Observe that this map is well-defined for inputs of any size (i.e., $x \in \vct V_n$ for any $n$), and its output is a sequence of random variables taking values in the vector spaces $(\vct V_k)$.  We then compare any two objects $x,y\in\bigsqcup_n\vct V_n$ by comparing the distributions of their samples $\msf S_k(x)$ and $\msf S_k(y)$ of each size $k$.
To quantify this comparison, we define the \emph{sampling metric} between $x,y \in \bigsqcup_n \vct V_n$ by
\begin{equation}\label{eq:sampling_metric}
    d_{\mathrm{samp}}(x,y)=\sum_{k\geq 1}2^{-k}W_1(\msf S_k(x), \msf S_k(y)),
\end{equation}
where $W_1$ denotes the Wasserstein-1 distance between two distributions of random variables\footnote{Wasserstein distances are defined between distributions, but to avoid notational clutter, we define these distances between random variables with the understanding that they apply to the laws of these variables.} with respect to some norm on $\vct V_k$, assumed to satisfy $\sup_k\mbb E\|\msf S_k(x)\|<\infty$ for each $x$ so that $d_{\mathrm{samp}}(x,y)<\infty$. The coefficients $2^{-k}$ were chosen for convenience, and any other positive summable sequence would suit.
An appropriately bounded sequence $(x_n)$ converges with respect to the sampling metric~\eqref{eq:sampling_metric} if and only if the sequence of fixed-dimensional samples $(\msf S_k (x_n))_n$ converges weakly for each size $k$, see Proposition~\ref{prop:compact_under_sampling}.  We next present several illustrations of sampling maps.






\begin{example}[Sampling columns]\label{ex:sampling_entries_intro}
    Suppose $\vct V_n=\RR^{d\times n}$ for a fixed $d \in \mbb N$.  Consider the map $\msf S_k$ which samples $k$ columns from $x=[x_1,\ldots,x_n]\in\vct V_n$ uniformly at random with replacement, i.e., sample indices $J_1,\dots,J_k \overset{iid}{\sim}\mathrm{Unif}([n])$ and set $\msf S_k(x)=[x_{J_1},\ldots,x_{J_k}]\in\vct V_k$.  
    We obtain a sampling metric~\eqref{eq:sampling_metric} by using the $\ell_{\infty}$ norm to define the $W_1$ distances.  Given a compact set $\Theta \subseteq \RR^d$, convergence of a sequence $(x_n \in \Theta^n)$ in $d_{\mathrm{samp}}$ is equivalent to weak convergence of the uniformly random columns $(\msf S_1(x_n))$ because the distribution of the size-$k$ sample is $\mathrm{Law}(\msf S_k(x))=\mathrm{Law}(\msf S_1(x))^{\otimes k}$. 
\end{example}

\begin{example}[Sampling vertices]\label{ex:sampling_vertices_intro}
    Consider $\vct V_n=\RR^{n\times n}$, viewed as the space of weighted and directed graphs on $n$ vertices.  Let $\msf S_k$ be the map that samples $k$ vertices with replacement and extracts the corresponding induced subgraph, i.e., for $x \in \vct V_n$, we have $(\msf S_k(x))_{i,j}=x_{J_i,J_j}$ for $J_1,\ldots,J_k\overset{iid}{\sim}\mathrm{Unif}([n])$. Once again, we obtain a sampling metric~\eqref{eq:sampling_metric} with the entrywise $\ell_{\infty}$ norm defining the $W_1$ distances.  Convergence in sampling metric corresponds precisely to the dense graph limits studied in~\cite{LOVASZ2006933,convergent_seqs1,diaconis2007graph, probability_graphons}. 
\end{example}

\begin{example}[Random binning, hashing]\label{ex:binning_intro}
    Suppose $\vct V_n=\RR^{d\times n}$ and consider the sampling map $\msf S_k$ that randomly assigns the $n$ columns of $x \in \vct V_n$ into $k$ bins and sums the columns in each bin, i.e., we sample $J_1,\ldots,J_n \overset{iid}{\sim}\mathrm{Unif}([k])$ and set the $i$th column of $\msf S_k(x)\in\vct V_k$ to $(\msf S_k(x))_i=\sum_{\ell: J_\ell=i}x_{\ell}$ for $i\in[k]$.  This random binning map, also called hashing, has found uses in machine learning as a dimensionality reduction method~\cite{CORMODE200558,shi2009hash} and in information theory as part of source and channel coding protocols~\cite{yassaee2014achievability}.  We will consider the associated sampling metric~\eqref{eq:sampling_metric} defined with the $\ell_2$ norm.
\end{example}

\begin{example}[Species sampling, random partitions]\label{ex:species_intro}
    Suppose $\vct V_n=\RR^n$.  A vector $x\in\Delta^{n}$ defines a random partition of $[k]$ for each $k\in\NN$ as follows. We view $x$ as a distribution on indices in $[n]$ and sample $k$ indices $I_1,\ldots,I_k\in[n]$ iid from this distribution. This sample defines a partition of $[k]$ in which $i,j\in[k]$ belong to the same block if $I_i=I_j$, and we set the entries of $\msf S_k(x) \in\Delta^k$ to be the fraction of indices in $[k]$ belonging to each block of the partition. Explicitly, we randomly order the $\ell\leq k$ distinct indices $\{I_1,\ldots,I_k\}$ sampled in this way as $t_1,\ldots,t_{\ell}$, and set $\msf S_k (x)_j=|\{i: I_i=t_j\}|/k$; the $t_j$'s are sometimes called `species'~\cite[Chap.~7]{rodriguez2013nonparametric}.  
    This sampling map extends to a general $x\in\RR^n$ by homogeneity, so $\msf S_k (x)_j=\|x\|_1\mathrm{sign}(x_{t_j})\msf S_k(|x|/\|x\|_1)_j$ for $x\neq0$ and $\msf S_k(0)=0$.  
    This notion of sampling and the associated random partition models have been studied extensively, see~\cite{pitman1995exchangeable,kingman1,kingman2,rodriguez2013nonparametric} for example. 

\end{example}

The sampling metric and associated notion of convergence have been previously widely used in the graph limits literature~\cite{lovasz2012large} to compare graphs of different sizes and to take their limits, as highlighted in Example~\ref{ex:sampling_vertices_intro}.  We demonstrate in this paper that the underlying idea is much more broadly relevant to comparing objects of different sizes in many other applications.  In particular, we go well beyond the preceding specific illustrations by describing next a combinatorial perspective on sampling based on random maps between finite sets.  This viewpoint yields three broad classes of sampling maps that represent a significant generalization of Examples~\ref{ex:sampling_entries_intro}-\ref{ex:species_intro}.  As we shall see, the sampling perspective is useful both analytically and methodologically, as it provides sketching maps with which to approximate large inputs by small ones and yields explicit rates for any-dimensional generalization and sketching. 

\subsubsection*{General Sampling Maps (Section~\ref{sec:sampling_general})}

We generalize the sampling maps used in Examples~\ref{ex:sampling_entries_intro}-\ref{ex:species_intro} and unify their analysis by using maps between finite sets. 
To motivate our generalization, we note that any map $f\colon[k]\to[n]$ between finite sets defines a (linear) map $\rho(f)\colon\RR^{d\times n}\to\RR^{d\times k}$ extracting the columns $f(1),\ldots,f(k)$ specified by $f$, or more formally $[\rho(f) x]_i=x_{f(i)}$ for $i\in[k]$ where subscripts index columns.
Moreover, if we replace a fixed such map $f$ by a uniformly random map $F_{n,k} : [k] \to [n]$, with each $F_{n,k}(i)$ drawn independently and uniformly from $[n]$, then $\rho(F_{n,k})$ is precisely the random map sampling $k$ columns as in Example~\ref{ex:sampling_entries_intro}.  

Likewise, a map $f\colon[k]\to[n]$ acts on pairs by $(i,j)\mapsto(f(i),f(j))$, and therefore defines a (linear) map $\rho(f)\colon\RR^{n\times n}\to\RR^{k\times k}$ between matrices indexed by these pairs via $[\rho(f)x]_{i,j}=x_{f(i),f(j)}$ for $i,j\in[k]$.  
Once again, applying a uniformly random map $\rho(F_{n,k})x$ 
recovers the vertex sampling map of Example~\ref{ex:sampling_vertices_intro}.

More generally, we consider a sequence of index sets $(\mathcal{I}_n)$ along with actions $\theta(f)\colon\mc I_k\to\mc I_n$ associated to each map $f\colon[k]\to[n]$ between finite sets, and we assume that these actions satisfy certain compatibility conditions (see Definition~\ref{def:FS_index}).  
We then consider the sequence of vector spaces $(\vct V_n = \RR^{\mathcal{I}_n})$ consisting of vectors indexed by $(\mc I_n)$, and associate a (linear) map $\rho(f)\colon\vct V_n\to\vct V_k$ to each map $f\colon[k]\to[n]$ between finite sets given by
\begin{equation}\label{eq:sampling_action}
    [\rho(f)x]_i=x_{\theta(f)(i)} \quad \textrm{for } x\in\vct V_n \textrm{ and } i\in\mc I_k.
\end{equation}
Finally, we define sampling with replacement maps $(\msf R_k)$ on these spaces by
\begin{equation}\label{eq:sampling_general}
    \msf R_k(x) = \rho(F_{n,k})x\quad  \textrm{where } F_{n,k}\colon[k]\to[n] \textrm{ is uniformly random},
\end{equation}
and $x \in \vct V_n$. 
In this manner, we generalize sampling with replacement to more complicated objects such as graph signals (by setting $\mc I_n=[n]^2\sqcup[n]$), hypergraphs and tensors (by setting $\mc I_n=[n]^d$), and polynomials (by setting $\mc I_n=[n]^d/\mfk S_d$ to be the collection of multisets of $[n]$ containing $d$ elements, where $\mfk S_d$ is the group of permutations on $d$ letters permuting indices in a tuple).
Our perspective enables a single unified analysis for general actions $\theta$, and thereby yields rates in a broad array of applications.


Our generalization of random binning also proceeds via an action of maps on index sets. Indeed, any function $f\colon[n]\to[k]$ can be viewed as binning $n$ indices into $k$ bins, and defines a map $\beta(f)\colon\RR^{d\times n}\to\RR^{d\times k}$ by binning columns in this way and summing the columns in each bin $[\beta(f) x]_i=\sum_{j\in f^{-1}(i)}x_j$.
Applying a uniformly random map $\beta(F_{k,n})$ in this way yields the random binning map of Example~\ref{ex:binning_intro}.  
More generally, we consider a sequence of index sets $(\mc I_n)$ with action $\theta$ and associated sequence of vector spaces $(\vct V_n=\RR^{\mc I_n})$ as above. We associate a (linear) map $\beta(f)\colon\vct V_n\to\vct V_k$ to each map $f\colon[n]\to[k]$ between finite sets via
\begin{equation}\label{eq:binning_action}
    [\beta(f)x]_i = \sum_{j\in\theta(f)^{-1}(i)}x_j \quad \textrm{for } x\in\vct V_n \textrm{ and } i\in\mc I_k.
\end{equation}
We remark that the above two maps $\rho$ and $\beta$ from~\eqref{eq:sampling_action} and~\eqref{eq:binning_action} are adjoints of each other, in the sense that $\beta(f)=\rho(f)^\star$ with respect to the usual inner products on $\RR^{\mc I_n}$.
Finally, we define random binning maps $(\msf B_k)$ on these vector spaces by
\begin{equation}\label{eq:binning_general}
    \msf B_k(x) = \beta(F_{k,n})x\quad \textrm{where } F_{k,n}\colon[n]\to[k] \textrm{ is uniformly random},
\end{equation}
and $x \in \vct V_n$.  
By varying the index sets $(\mathcal{I}_n)$, we generalize random binning to more general objects, and once again they can all be analyzed in a unified manner. For example, setting $\mc I_n=[n]^2$ recovers the random quotients of graphs studied in~\cite{levin2025graphs}.  



Lastly, we generalize the species sampling map from Example~\ref{ex:species_intro}. We consider a particular class of index sets $(\mc I_n)$ that can be derived from tuples $[n]^d$ (see Definition~\ref{def:FS_index_degree}), and view vectors $x\in\Delta^{\mc I_n}$ as distributions over the index sets $\mc I_n$. We can then repeat the construction in Example~\ref{ex:species_intro} to form more general species sampling maps $(\msf E_k)$, see Section~\ref{sec:species}.


Given a sequence of index sets $(\mc I_n)$, we have defined three sampling maps on the same underlying vector spaces $\RR^{\mc I_n}$.  When is it appropriate to use a particular sampling map to compare objects of different sizes?  Each of the above sampling maps induces an equivalence between different inputs, meaning there exist $x\neq y$ possibly of different sizes with $d_{\mathrm{samp}}(x,y)=0$. Therefore, a choice of sampling maps is appropriate only if pairs $x,y$ at sampling distance zero can indeed be viewed as equivalent in the context of the given application domain.
We characterize these equivalences in Section~\ref{sec:sampling_general}.

Informally, for sampling with replacement $(\msf R_k)$ as defined above, the distribution of $\msf R_k(x)$ is unchanged if and only if we duplicate or permute the entries of $x$, appropriately defined (see Propositions~\ref{prop:symmetry_coFS} and~\ref{prop:relations_coFS}). For example, if $x\in\RR^{d\times n}$ and we permute the columns of $x$ or duplicate them $m$ times (so $x\mapsto x\otimes\mathbbm{1}_m^\top$), then the distribution of $\msf R_k(x)$ is unchanged for all $k\in\NN$. 
Such permutation and duplication are natural when $x$ represents a point cloud, strategies in a symmetric game converging to a mean-field limit, or dense graphs, see Example~\ref{ex:duplication}.


For random binning $(\msf B_k)$ and species sampling $(\msf E_k)$, the distributions of $(\msf B_k(x))$ and $(\msf E_k(x))$ are unchanged if and only if we zero-pad or permute the entries of $x$, again appropriately defined. 
For example, if $x\in\RR^n$ and we append $m$ zeros to $x$, then the distributions of $\msf B_k(x)$ and $\msf E_k(x)$ are unchanged for all $k$.
Such permutation and zero-padding are natural when $x$ represents a distribution on unlabelled items or a sparse graph, see Example~\ref{ex:zero_padding}.



\subsubsection*{Rates for Sampling with Replacement (Section~\ref{sec:sampling_wrep})}

We derive rates for any-dimensional sketching and generalization using the sampling metric~\eqref{eq:sampling_metric} defined with respect to general sampling with replacement maps $(\msf R_k\colon\bigsqcup_n\vct V_n\to\vct V_k)$.
To state our results, we fix a sequence of compact subsets $(\Omega_n \subseteq \vct V_n)$ such that $\msf R_k (\bigsqcup_n \Omega_n) \subseteq \Omega_k$ almost surely for all $k$, i.e., the subsets $\Omega_n$ are closed under sampling with replacement. For example, the hypercubes $\Omega_n=[-1,1]^n\subseteq\RR^n$ and the collection of unweighted graphs $\Omega_n=\{0,1\}^{n\times n}\subseteq\RR^{n\times n}$ are closed under sampling with replacement.
\begin{theorem}[informal, see Theorem~\ref{thm:W1_bound_sampling}]\label{thm:informal_sampling} 
We have the following rates for sampling with replacement $(\msf R_k\colon\bigsqcup_n\Omega_n\to\Omega_k)$ acting on compact sets $(\Omega_n)$ closed under sampling.
\begin{enumerate}[font=\emph, align=left, labelwidth=!, labelindent=0pt]
    \item[(Any-dimensional sketching)] For any function $f\colon\bigsqcup_n\Omega_n\to\RR$ that is $L$-Lipschitz in $d_{\mathrm{samp}}$ and any $x\in\bigsqcup_n\Omega_n$, with probability at least $1-e^{-2\epsilon^2}$ we have
    \begin{equation*} 
        |f(x)-f(\msf R_n(x))| \leq Lc_1 \exp\Big[-(c_2^{-1}\log n)^{\frac{1}{1+D}}\Big] + L\epsilon\sqrt{\frac{c_3}{n}}.
    \end{equation*}
    
    \item[(Any-dimensional generalization)] For any other $L$-Lipschitz function $\widehat f$, we have that:
    \begin{equation*} 
        \mathrm{e}_{\infty}(f,\widehat f)\leq\mathrm{e}_n(f,\widehat f)+2L c_1\exp\left[-(c_2^{-1}\log n)^{\frac{1}{1+D}}\right].
    \end{equation*}
\end{enumerate}
The constants $c_1,c_2,c_3 >0$ are explicit and depend on the collection $(\Omega_n)$, and the parameter $D$ is a `degree' that quantifies the complexity of the action $\theta$ underlying the sampling maps $(\msf R_k)$.
\end{theorem}


The proof of these bounds relies on a concentration result in which we show that the distance $d_{\mathrm{samp}}(x,\msf R_n(x))$ is $c/n$-subgaussian, for a constant $c > 0$.
For any-dimensional sketching, we are able to approximate the value of a Lipschitz function on an arbitrarily large input by evaluating the function on a random sample of the input, with the size of the sketch depending only on the desired accuracy. 
For any-dimensional generalization, our results show that if $\widehat f\approx f$ on inputs of size $n$, then $\widehat f\approx f$ on inputs of all sizes up to a slack that decays to zero as $n$ increases. In other words, approximating a target function on sufficiently large input sizes guarantees a good approximation on inputs of all sizes.  Moreover, Theorem~\ref{thm:informal_sampling} reduces the analysis of any-dimensional generalization error to generalization error on a finite-dimensional compact set, a classical problem (see~\cite{vapnik1999nature} for example).

The rates in Theorem~\ref{thm:informal_sampling} are clearly quite slow.  On the other hand, these rates are uniform in the sense that they hold for objects of any size.  
Moreover, these slow uniform rates are to be expected in general. For example, when $\Omega_n$ is the collection of $n\times n$ adjacency matrices of simple graphs, a rate of $\exp[-\frac{1}{2}\log\log n]$ was shown in~\cite[Thm~2.9]{convergent_seqs1} for the closely related cut metric as a consequence of Szemer\'edi's regularity lemma (see~\cite[Exer.~10.33]{lovasz2012large} for the connection between the cut and sampling metrics). 
Nevertheless, we can substantially improve the above uniform rates using additional structure in the function $f$.  
For example, we will see in Section~\ref{sec:sampling_wrep} that many any-dimensional functions $f$ depend on their arbitrarily large inputs only via a fixed-dimensional random sample, i.e., $f(x)$ depends on its input only via $\mathrm{Law}(\msf R_k(x))$ for some fixed $k$.


We give two illustrations here of some of these improved rates in specific applications. 
The first pertains to polynomials that are unchanged by suitably defined duplication of the entries of their inputs. Specifically, a function $p\colon\bigsqcup_n\vct V_n\to\RR$ is a polynomial if all the restrictions $p|_{\vct V_n}$ are polynomials of some fixed degree, denoted $\deg(p)$.


\begin{corollary}[informal, see Corollary~\ref{cor:moment_polys}]\label{cor:moment_polys_intro}
    Suppose that $(\Omega_n)$ is a sequence of compact sets closed under sampling, and let $p\colon\bigsqcup_n\Omega_n\to\RR$ be a polynomial that is symmetric and unchanged by duplication. Then for any $x\in\bigsqcup_n\Omega_n$ with probability at least $1-2e^{-2\epsilon^2}$ we have
    \begin{equation*}
        |p(x)-p(\msf R_n(x))|\leq \frac{c_1\deg(p)^2}{n}+\frac{c_2\deg(p)\epsilon}{\sqrt{n}},
    \end{equation*}
    and if $\widehat p$ is another such polynomial of the same degree,
    \begin{equation*}
        \mathrm{e}_{\infty}(p,\widehat p)\leq \mathrm{e}_n(p,\widehat p) + \frac{c_3\deg(p)^2}{n}.
    \end{equation*}
    Above, the constants $c_1,c_2,c_3>0$ only depend on $(\Omega_n)$ and the polynomials $p,\widehat p$.
\end{corollary}
Examples of polynomials satisfying the above conditions are polynomials in moments of measures, graph homomorphism densities, and polynomial graph neural networks; see Corollary~\ref{cor:moment_polys} and the discussion following it.
The proof uses a result of~\cite{levin2025deFin} stating that such polynomials compute moments of fixed-sized samples of their inputs. 
We remark that the above $O(n^{-1})$ any-dimensional generalization rates are a substantial improvement over previous rates in the literature, as the latter were proved for larger function classes and exploit less structure. For example, the framework of~\cite{levin2025transferring} yields the rates $O(n^{-1/d})$ for moment polynomials on $\RR^d$, and $O((\log n)^{-1/4})$ for graph homomorphism densities; see~\cite[Cors.~F.6,~G.4]{levin2025transferring}.

The second consequence we highlight is to sketching and generalization for transformers. 
\begin{corollary}[informal, see Corollary~\ref{cor:transformers}]\label{cor:transformers_intro}
    Suppose $\vct V_n=\RR^{d\times n}$ with $d>2$, and $\Omega_n=\Theta^n$ for compact $\Theta\subseteq\RR^d$. If $T\colon\bigsqcup_n\vct V_n\to\RR^d$ is a permutation-invariant transformer,\footnote{We assume infinite-precision self-attention and mean pooling for the last layer, see Corollary~\ref{cor:transformers}.} 
    then for any $x\in\bigsqcup_n\Omega_n$ with probability at least $1-e^{-2\epsilon^2}$ we have
    \begin{equation}\label{eq:transformer_sampling_intro}
        \|T(x)-T(\msf R_n(x))\|_2\leq \frac{c_1(1+2\epsilon)}{\sqrt{n}}.
    \end{equation}
    For any other such transformer $\widehat T$, we have
    \begin{equation*}
        \mathrm{e}_{\infty}(T,\widehat T)\leq \mathrm{e}_n(T,\widehat T) + \frac{c_2}{\sqrt{n}},
    \end{equation*}
    where $\mathrm{e}_n(T,\widehat T)=\sup_{x\in\Omega_{\leq n}}\|T(x)-\widehat T(x)\|_2$ and $\mathrm{e}_{\infty}(T,\widehat T)=\sup_n\mathrm{e}_n(T,\widehat T)$. 
    Furthermore, if $G(x)=\bar G(\mathrm{Law}(\msf R_1(x)))$ for an $L$-Lipschitz function $\bar G\colon\mc P(\Theta)\to\RR^d$, then
    \begin{equation*}
        \mathrm{e}_{\infty}(T,G)\leq \mathrm{e}_n(T,G) + \frac{c_3}{n^{1/d}}.
    \end{equation*}
    Above, the constants $c_1,c_2,c_3>0$ depend on $L$, the architectures of $T$ and $\widehat T$, and on $\Theta$. 
\end{corollary}
The proof is based on the measure-theoretic in-context mapping of~\cite{furuya2025transformers}; see Section~\ref{sec:proofs_sampling}. The $O(1/\sqrt{n})$ sketching rate for a single-layer transformer was previously proved in~\cite[Prop.~D.13]{bohbot2025token}.
We remark that evaluating $T(x)$ exactly on an input $x\in\vct V_N$ consisting of $N$ tokens requires $O(N^2)$ operations (see~\eqref{eq:attention}), but evaluating it to a desired accuracy $\delta>0$ can be done with high probability in $O(\delta^{-4})$ time (independent of the number of tokens $N$) after sampling $n=O(\delta^{-2})$ columns from $x$ uniformly at random and using~\eqref{eq:transformer_sampling_intro}. 
The first generalization bound above is relevant for model distillation for example, where a small transformer is trained to imitate a much larger one~\cite{miniLM,tinybert}. In another direction, transformers have been shown to approximate functions on measures to arbitrary precision~\cite{furuya2025transformers,lavenant2026continuous}. The second generalization bound quantifies the quality of this approximation on all measures from the quality of approximation on measures obtained by sampling from fixed-size vectors.







\subsubsection*{Rates for Random Binning and Species Sampling (Section~\ref{sec:quotients})}
Next we present our results on any-dimensional sketching and generalization using the sampling metric $d_{\mathrm{samp}}$ defined with respect to general random binning maps $(\msf B_k\colon\bigsqcup_n\vct V_n\to\vct V_k)$.  Once again, we fix a sequence of compact subsets $(\Omega_n \subseteq \vct V_n)$ closed under sampling, this time under both random binning and under species sampling so $\msf B_k (\bigsqcup_n \Omega_n) \subseteq \Omega_k$ almost surely for all $k$ and similarly for $(\msf E_k)$. Examples include the simplex $\Omega_n=\Delta^n$ and entrywise $\ell_1$ balls.

As in Theorem~\ref{thm:informal_sampling}, we would like to use the maps $(\msf B_k)$ for sketching high-dimensional inputs.  The challenge with this approach is that the distance $d_{\mathrm{samp}}(x,\msf B_n(x))$ does not concentrate as well as in the case of sampling with replacement (see Example~\ref{ex:anti_concentration}).  Fortunately, our generalization of the species sampling map $(\msf E_k)$ exhibits the requisite concentration, as $d_{\mathrm{samp}}(x,\msf E_n(x))$ is $c/n$-subgaussian for a constant $c > 0$ depending only on $(\Omega_n)$.  Thus, species sampling furnishes the necessary sketching map.


\begin{theorem}[informal, see Theorem~\ref{thm:W1_rates_FS}]\label{thm:W1_rates_FS_intro}
    We have the following results for random binning $(\msf B_k)$ and species sampling $(\msf E_k)$ acting on compact sets $(\Omega_k)$ closed under both sampling maps. Below, we let $d_{\mathrm{samp}}$ be the sampling metric defined by $(\msf B_k)$.
    \begin{enumerate}[font=\emph, align=left,labelwidth=!, labelindent=0pt]
        \item[(Any-dimensional sketching)] For any function $f\colon\bigsqcup_n\Omega_n\to\RR$ that is $L$-Lipschitz in $d_{\mathrm{samp}}$ and any $x\in\bigsqcup_n\Omega_n$, we have with probability at least $1-e^{-2\epsilon^2}$ that
        \begin{equation*}
            |f(x)-f(\msf E_n(x))| \leq \frac{L(c_1+ c_2\epsilon)}{\sqrt{n}}.
        \end{equation*}

        \item[(Any-dimensional generalization)] If $\widehat f\colon\bigsqcup_n\Omega_n\to\RR$ is also $L$-Lipschitz in $d_{\mathrm{samp}}$, then
        \begin{equation*}
            \mathrm{e}_{\infty}(f,\widehat f)\leq \mathrm{e}_n(f,\widehat f) + \frac{2Lc_1}{\sqrt{n}}.
        \end{equation*}
    \end{enumerate}
    The constants $c_1,c_2 >0$ are explicit and depend only on the collection $(\Omega_n)$.
\end{theorem}

In words, we always get $n^{-1/2}$ rates for random binning, regardless of the complexity of the action $\beta$ in~\eqref{eq:binning_general}. We obtain the same rates also for functions $f$ exhibiting latent low-complexity structure.
The following are some of the function classes satisfying the sketching and generalization rates of Theorem~\ref{thm:W1_rates_FS_intro}.
\begin{enumerate}[align=left, font=\emph,labelwidth=!, labelindent=0pt,wide]
    \item[(Polynomials)] In analogy to Corollary~\ref{cor:moment_polys_intro}, there is a large class of polynomial functions that are continuous with respect to random binning.  Specifically, Theorem~\ref{thm:W1_rates_FS} applies to any polynomial $p\colon\bigsqcup_n\vct V_n\to\RR$ that is symmetric and unchanged by zero-padding, appropriately defined (Corollary~\ref{cor:FS_polynomials}). In fact, we obtain a faster $O(1/n)$ generalization rate for such polynomials in Corollary~\ref{cor:FS_polynomials}. Our results apply in particular to multi-symmetric functions on sequences of vectors and homomorphism numbers for graphs of arbitrary size.

    \item[(Symmetric neural networks)] Many permutation-invariant neural network architectures are continuous with respect to random binning, including DeepSets~\cite{deepsets} (Corollary~\ref{cor:deepsets}), PointNet~\cite{pointnet} (Corollary~\ref{cor:pointnet}), and graph neural networks~\cite{GNNs} (Corollary~\ref{cor:gnns}), with appropriate choices of parameters and nonlinearities.
\end{enumerate}
Finally, we show that the topologies defined by random binning and by species sampling via the corresponding sampling metrics are closely related to each other (Proposition~\ref{prop:relation_to_species}).
See Section~\ref{sec:quotients} for more details.

\subsection{Related Work}
There are several machine learning models that are defined for inputs of different sizes. Examples include neural networks processing sets and point clouds of different sizes~\cite{deepsets,pointnet,bueno2021representation}, graphs of different sizes~\cite{GNNs,maron2018invariant,transferab1}, and transformers processing sequences of tokens of any length~\cite{LIN2022111,transformers_time_series,furuya2025transformers}.  Importantly, the any-dimensional generalization and sketching rates of each of the above models are comparatively less studied, with the following notable exceptions.


The ability of a graph neural network to generalize to graphs of different sizes has been called transferability in the literature, and has been extensively studied by considering appropriate topologies and limits on the space of graphs of all sizes, see~\cite{transferab1,maskey2023transferability,transferb1,graphop_transfer} for example. 
In particular, many results in this line of work exploit the continuity of appropriately-normalized graph neural networks with respect to dense graph limits, which correspond to convergence in the sampling metric~\eqref{eq:sampling_metric} defined using sampling vertices with replacement as in Example~\ref{ex:sampling_vertices_intro}, see~\cite[\S1]{lovasz2012large}.
We generalize and unify many of the techniques in this literature by considering more general notions of sampling maps, which are in turn useful in a broader array of applications.

There is also a literature studying the ability of transformers to generalize to inputs of different lengths, see~\cite{zhou2024transformers,limit_transformers,PE_transformers,yang2026length} for example. Some of the results in this literature are negative, showing that transformers often do \emph{not} generalize well to longer inputs, while others prove length generalization under specific assumptions on the transformer architecture. Our theory contributes to the latter line of work by giving explicit generalization guarantees for permutation-invariant transformers.
The measure-theoretic description of attention we use for this purpose has been previously used to prove universality results for transformers~\cite{furuya2025transformers,furuya2026function,Yun2020Are}, but to our knowledge our result is the first application of these ideas to length generalization.

The authors of~\cite{diaz2025invariant} study the particular case of generalization of symmetric any-dimensional polynomials fitted via least-squares to training inputs of bounded dimensions. They measure generalization using mean-squared error with respect to fixed training and test sets in possibly different dimensions, and their bounds depend explicitly on these sets. In contrast, we simultaneously control the error on inputs of all possible sizes by the error in a fixed training dimension; the latter in turn can be controlled by classic results in generalization theory in a fixed dimension~\cite{vapnik1999nature}.  More broadly, our framework applies seamlessly to general classes of functions continuous with respect to sampling, which go well beyond polynomials.


Transferability of graph neural networks has been extended in~\cite{levin2025transferring} to more general any-dimensional models, where it was defined to be continuity of the model in a certain space containing inputs of all sizes and their limits. 
The any-dimensional generalization and sketching problems that are the subject of our paper are neither formalized nor explicitly tackled there, and sampling maps do not play a role in the construction of their space. Nevertheless, several results in~\cite{levin2025transferring} do pertain to these problems and involve sampling.
%
%
These include rates for any-dimensional generalization for average error with respect to specific data distributions~\cite[Prop.~4.2]{levin2025transferring}, and expected convergence rates for function values on low-dimensional samples of high-dimensional inputs~\cite[Prop.~D.5]{levin2025transferring}. Both assume access to sampling maps satisfying certain desiderata, verified case-by-case. 
%
%
In contrast, sampling maps are fundamental to our framework---we systematically generalize several notions of sampling maps, and we prove the requisite properties pertaining to them in a unified fashion and in broader application domains. In particular, we formalize the notion of any-dimensional data distributions, and we obtain improved generalization rates with respect to them compared to~\cite{levin2025transferring} by exploiting sampling-specific structure in the functions; we also prove any-dimensional generalization with respect to worst-case error.
Furthermore, we obtain convergence rates for function values on low-dimensional samples that hold not only in expectation but also with high probability, which is essential for deriving a practically relevant methodology for sketching.  


Finally, the sketching rates we obtain for sampling with replacement generalize previous rates obtained for measures~\cite{fournier2023convergence} (viewed as limits of point clouds) and graphons~\cite{convergent_seqs1} (viewed as limits of dense simple undirected graphs). These sketching rates were used to test properties of large graphs in the latter literature~\cite{lovasz2012large}.
In particular, rates of approximation for graphons by finite graphs were shown to follow from compactness of the space of graphons in~\cite{lovasz2007szemeredi}, a connection we generalize in Theorem~\ref{thm:general_equivalence}.
Our random binning and species sampling maps generalize those studied in~\cite{levin2025graphs} for graphs with nonnegative edge weights summing to one. In particular, the authors of~\cite{levin2025graphs} show a similar $n^{-1/2}$ rate as in Theorem~\ref{thm:W1_rates_FS_intro} with respect to a different but related metric.

\subsection{Notation}
We denote by $\NN$ the collection of strictly positive integers and by $\NN_0=\NN\cup\{0\}$. For $n\in\NN$, we denote $[n]=\{1,\ldots,n\}$. We denote by $\mfk S_n=\{\pi\colon[n]\to[n] \textrm{ bijective}\}$ the group of permutations on $n$ letters.
If $\mc I$ is a finite set, we denote by $\RR^{\mc I}$ the vector space consisting of vectors indexed by $\mc I$, by $\Delta^{\mc I}$ the unit simplex in $\RR^{\mc I}$ consisting of nonnegative vectors whose entries sum to 1, and by $[-r,r]^{\mc I}$ for $r>0$ the hypercube consisting of vectors with entries of magnitude at most $r$.
If $(\Omega,d)$ is a pseudometric space and $S\subseteq\Omega$, we denote by $\mathrm{dist}(x,S)=\inf_{y\in S}d(x,y)$ the distance of $x\in\Omega$ to $S$. The space $(\Omega,d)$ is totally bounded if all its covering numbers are finite.
We denote the distribution of a random variable $X$ by $\mathrm{Law}(X)$. If $X$ and $Y$ are two random variables on the same space, we write $X\overset{d}{=}Y$ to denote equality of their distributions $\mathrm{Law}(X)=\mathrm{Law}(Y)$. We denote the collection of (Borel) probability distributions on $\Omega$ by $\mc P(\Omega)$. 
If $\mu\in\mc P(\Omega)$, we write $X\sim\mu$ for a random variable $X$ to denote $\mathrm{Law}(X)=\mu$.
If $f\colon\Omega\to\Omega'$ is a continuous map between two topological spaces $\Omega$ and $\Omega'$ and $\mu\in\mc P(\Omega)$, we denote by $f\mu$ the pushforward of $\mu$ by $f$. If $\mu\in\mc P(\Omega)$ and $f\colon\Omega\to\RR$ is a measurable function, we denote by $\mbb E_{\mu}f=\mbb E_{X\sim\mu}f(X)$ the expectation of $f$ with respect to $\mu$.
The Wasserstein-1 distance between two measures $\mu,\nu\in\mc P(\Omega)$ on a metric space $(\Omega,d)$ with finite first moments is defined by
\begin{equation*}
    W_1(\mu,\nu)=\inf_{\substack{\textrm{random } (X,Y)\\ X\sim\mu, Y\sim\nu}} \mbb Ed(X,Y) = \sup_{\substack{f\colon\Omega\to\RR\\ \textrm{1-Lipschitz}}}|\mbb E_{\mu}f-\mbb E_{\nu}f|,
\end{equation*}
where the infimum is taken over couplings of $\mu$ and $\nu$. To simplify our notation, if $X$ and $Y$ are random variables (not necessarily coupled) on the same space $\Omega$, we denote $W_1(X,Y)=W_1(\mathrm{Law}(X),\mathrm{Law}(Y))$.
A sequence of measures $(\mu_n)\subseteq\mc P(\Omega)$ converges weakly to $\mu\in\mc P(\Omega)$ if $\mbb E_{\mu_n}f\to\mbb E_{\mu}f$ for all bounded continuous functions $f\colon\Omega\to\RR$. If $\Omega$ is a compact metric space, this is further equivalent to convergence in $W_1$ distance. A sequence of random variables $(X_n)$ converges weakly if their distributions $(\mathrm{Law}(X_n))$ converge weakly.
If $(\Omega_n)$ are sets, we denote their disjoint union by $\bigsqcup_n\Omega_n=\{(n,x):x\in\Omega_n, n\in\NN\}$. If $\sim$ is an equivalence relation on a set $\Omega$, we denote by $\Omega/\sim$ the quotient space consisting of equivalence classes. Any pseudometric $d$ on $\Omega$ induces an equivalence relation by setting $x\sim y$ if $d(x,y)=0$. If $d$ is a pseudometric on $\bigsqcup_n\Omega_n$, we denote by $\Omega_{\infty}=\bigsqcup_n\Omega_n/\sim$ the quotient space under the equivalence relation induced by $d$, on which $d$ defines a metric, and by $\overline{\Omega}_{\infty}$ the completion of this metric space.

\section{Comparing Objects of Different Sizes}\label{sec:general}
In this section, we begin our study of any-dimensional learning by formally investigating how inputs of different sizes are compared to each other. We start in Section~\ref{sec:compact} by stating a general equivalence between various notions of generalization and sketching across dimensions, and compactness of the metric space (for a general metric) consisting of inputs of all sizes and their limits.  We then focus in Section~\ref{sec:sampling_metric} on the sampling metric.  We describe a large and natural family of compact metric spaces that can be derived with respect to the sampling metric, and we provide a link between (random) limit objects in these metric spaces and an appropriate notion of an any-dimensional data distribution. Finally, we discuss compatibility conditions between index sets that lead to general sampling maps in Section~\ref{sec:sampling_general}.  Missing proofs in this section are given in Section~\ref{sec:proofs_general}.


\subsection{The Role of Compactness}\label{sec:compact}


Given a pseudometric defined on inputs of all sizes, we prove that compactness of this metric space is equivalent to the any-dimensional generalization and sketching properties from Section~\ref{sec:intro}. While the argument is elementary, we include it here for completeness.  Suppose $d$ is a pseudometric on $\bigsqcup_n\Omega_n$, and define the equivalence relation $x\sim y$ if $d(x,y)=0$. We allow for pseudometrics, as opposed to only metrics, because inputs of different sizes can correspond to the same object in the context of a given application domain, see Sections~\ref{sec:intro} and~\ref{sec:sampling_general}.
Note that $d$ is a metric on $\Omega_{\infty}=\bigsqcup_n\Omega_n/\sim$, and we define $\overline{\Omega}_{\infty}$ to be the completion of this metric space.
Denote by $\mc F_1$ the collection of 1-Lipschitz functions on $\bigsqcup_n\Omega_n$, or equivalently, on $\overline{\Omega}_{\infty}$.


\begin{theorem}\label{thm:general_equivalence}
    Assume $(\Omega_n,d)$ is totally bounded for each $n$. Then the following are equivalent.
    \begin{enumerate}[labelwidth=!, labelindent=0pt]
        \item $(\overline{\Omega}_{\infty},d)$ is compact.
        \item We have $\lim_{n\to\infty}\sup_{x\in\Omega_{\infty}}\mathrm{dist}(x,\Omega_{\leq n})=0$.
        \item There exist random maps $\msf S_k\colon\bigsqcup_n\Omega_n\to\Omega_{\leq k}$ for each $k$ such that
        \begin{equation*}
            \lim_{k\to\infty}\sup_{f\in\mc F_1}\mathrm{e}_{\infty}(f, f\circ\msf S_k) = 0,\quad \textrm{almost surely.}
        \end{equation*}
        
        \item There exists a rate $(R_n\geq 0)_{n\in\NN}$ with $R_n\to 0$ such that for any $f,g\in\mc F_1$, we have
        \begin{equation*}
            \mathrm{e}_{\infty}(f,g)\leq \mathrm{e}_n(f,g) + R_n.
        \end{equation*}
        Moreover, we have $\mathrm{e}_n(f,g)<\infty$ for all $n$, so $\mathrm{e}_{\infty}(f,g)<\infty$ as well.
    \end{enumerate}
\end{theorem}
The proof of this theorem proceeds via a straightforward direct argument, which we defer to Section~\ref{sec:proofs_general}.
The above theorem shows that a number of desirable properties are equivalent to compactness of the limit space $\overline{\Omega}_{\infty}$ with the pseudometric $d$.  
Part 2 is an any-dimensional approximation property that says that any object of any size can be approximated by objects of size at most $n$ with uniform rates, which pertains to the any-dimensional sketching question. The derivation of these rates is the key step in our analyses below.
Parts 3 and 4 give a positive answer to the any-dimensional sketching and generalization questions we posed in Section~\ref{sec:intro}, respectively, with uniform rates over the class of 1-Lipschitz functions.
We remark that part 3 can be equivalently stated in terms of deterministic maps $\msf S_k$, but the above formulation more closely aligns with our any-dimensional sketching question and subsequent developments using sampling maps.


As discussed in Section~\ref{sec:intro}, we would also like to quantify generalization error with respect to any-dimensional data distributions.
Specifically, we consider average error with respect to sequences of distributions supported on inputs of growing size that converge weakly to a distribution on limit objects, as formalized by the following proposition. Below, if $\mu\in\mc P(\overline{\Omega}_{\infty})$ and $f,g\colon\overline{\Omega}_{\infty}\to\RR$ we define $\mathrm{e}_{\mu}(f,g)=\mbb E_{\mu}|f-g|$.
\begin{proposition}\label{prop:avg_error}
    Suppose $(\overline{\Omega}_{\infty},d)$ is compact. For any sequence of distributions $(\mu_n\in\mc P(\Omega_{\leq n}))$ and $\mu_{\infty}\in\mc P(\overline{\Omega}_{\infty})$, we have $\mu_n\to\mu_{\infty}$ weakly if and only if there exists a rate $(R_n\geq0)$ with $R_n\to0$ such that for any $f,g\in\mc F_1$, we have $\mathrm{e}_{\mu_{\infty}}(f,g)\leq\mathrm{e}_{\mu_n}(f,g)+R_n$.
\end{proposition}
Once again, the proof is direct and is given in Section~\ref{sec:proofs_general}.
The above results are completely general, and show that any compact pseudometric between objects of different sizes enables generalization across dimensions for functions that are continuous with respect to that metric. We turn next to analyzing the sampling metric~\eqref{eq:sampling_metric} in more detail, and show that it admits rich families of compact sets and a natural correspondence between limit objects and any-dimensional data distributions.

\subsection{The Sampling Metric}\label{sec:sampling_metric}
Fix vector spaces $(\vct V_n)$ and sampling maps $(\msf S_k\colon\bigsqcup_n\vct V_n\to\vct V_k)$, where each $\msf S_k$ is viewed as a random map.\footnote{We can also view $\msf S_k$ as a kernel, i.e., a map $\bigsqcup_n\vct V_n\to\mc P(\vct V_k)$. The two views are equivalent by~\cite[Lemma~3.2(vii)]{kallenberg1997foundations}.} We begin by showing that any sequence of compact sets $(\Omega_n\subseteq\vct V_n)$ closed under these sampling maps yields a compact limiting space.


\begin{proposition}\label{prop:compact_under_sampling}
    Suppose $(\Omega_n\subseteq\vct V_n)$ is a sequence of compact sets in the usual topology such that $\msf S_k\left(\bigsqcup_n\Omega_n\right)\subseteq\Omega_k$ almost surely. Endow each $\vct V_n$ with a norm such that $\sup_n\sup_{x\in\Omega_n}\|x\|<\infty$, and consider the sampling metric~\eqref{eq:sampling_metric} on $\bigsqcup_n\Omega_n$ with $W_1$ distances defined using the above norms. Then $(\overline{\Omega}_{\infty},d_{\mathrm{samp}})$ is compact and each $(\Omega_n,d_{\mathrm{samp}})$ is totally bounded. Moreover, a sequence $(x_i)\subseteq\bigsqcup_n\Omega_n$ converges in $d_{\mathrm{samp}}$ if and only if $(\msf S_k(x_i))_i$ converges weakly for each $k\in\NN$.
\end{proposition}
We emphasize that we are assuming above that each $\Omega_n$ is compact in the usual norm topology on $\vct V_n$, and conclude that each $\Omega_n$ is precompact (or equivalently, totally bounded) in the topology induced by the sampling metric $d_{\mathrm{samp}}$.
This proposition follows from Tychonoff's theorem applied to the product $\prod_n\mc P(\Omega_n)$, and we give the full proof in Section~\ref{sec:proofs_general}.
Note that this result assumes nothing about the sampling maps or how they are related to each other.

We proceed to give several examples of sampling maps and sets satisfying the above hypotheses. In some of these examples the limit space $\overline{\Omega}_{\infty}$ has been characterized before and shown to be compact, while in others an explicit description may be involved and challenging to obtain, but compactness still follows from the above result.

\begin{example}[Sampling columns]\label{ex:sampling_entries}
    Suppose $\vct V_n=\RR^{d\times n}$, and consider the maps $\msf S_k$ sampling $k$ columns from $x$ uniformly at random with replacement as in Example~\ref{ex:sampling_entries_intro}.
    Note that for any compact $\Theta\subseteq \RR^d$, the sequence of product sets $\Omega_n=\Theta^n$ is closed under sampling. As noted in Example~\ref{ex:sampling_entries_intro}, a sequence $(x_i)\subseteq \bigsqcup_n\Theta^n$ converges in $d_{\mathrm{samp}}$ if and only if $(\mathrm{Law}(\msf S_1(x_i)))_i\subseteq\mc P(\Theta)$ converge weakly. 
    In this case, sending $x\mapsto \mathrm{Law}(\msf S_1(x))$ identifies $\Omega_{\infty}$ with the collection of all measures in $\mc P(\Theta)$ that have a finite support and rational weights, as these are precisely the measures of the form $\mathrm{Law}(\msf S_1(x))$ for some $x\in\Theta^n$ and some $n\in\NN$. We can then identify $\overline{\Omega}_{\infty}$ with all of $\mc P(\Theta)$ endowed with the weak topology, which is indeed compact. 
\end{example}

\begin{example}[Sampling vertices]\label{ex:sampling_vertices}
    Suppose $\vct V_n=\RR^{n\times n}$, viewed as the space of weighted and directed graphs, and consider the sampling map $\msf S_k$ drawing $k$ vertices with replacement and extracting the corresponding induced subgraph as in Example~\ref{ex:sampling_vertices_intro}. 
    Then the sequence $(\Omega_n)$ of simple undirected graphs on $n$ vertices is closed under sampling, and limits in $d_{\mathrm{samp}}$ correspond to dense graph limits~\cite{LOVASZ2006933,convergent_seqs1}. Such limits can be represented by a graphon, which is a symmetric measurable function $W\colon[0,1]^2\to[0,1]$ modulo an equivalence relation~\cite{LOVASZ2006933}. In this case, we can identify $\Omega_{\infty}$ with the space of so-called step-graphons, which are (equivalence classes of) step functions associated to finite graphs~\cite[\S3.1]{convergent_seqs1}. We can then identify $\overline{\Omega}_{\infty}$ with the space of all graphons, which was shown to be compact in~\cite{lovasz2007szemeredi}.

    For more general sequences of sets $(\Omega_n)$ closed under sampling, including weighted and directed graphs, the characterization of $\overline{\Omega}_{\infty}$ becomes more involved~\cite{probability_graphons}, but its compactness still follows from Proposition~\ref{prop:compact_under_sampling}.
\end{example}

\begin{example}[Random binning and species sampling]\label{ex:binning}
    If $\vct V_n=\RR^n$, then the sequence of simplices $(\Omega_n=\Delta^n)$ is closed under both random binning as in Example~\ref{ex:binning_intro} and species sampling as in Example~\ref{ex:species_intro}. 
    In this case, we can view $\Omega_{\infty}$ as the space of random exchangeable probability measures on $[0,1]$ with finite support by associating to each $x\in\RR^n$ the random measure $\sum_{i=1}^nx_i\delta_{T_i}$ where $T_1,\ldots,T_n$ are iid uniform in $[0,1]$. Then the closure $\overline{\Omega}_{\infty}$ with respect to both random binning and species sampling corresponds to the space of all random exchangeable measures\footnote{Here we endow the space of random probability measures $\mc P(\mc P([0,1]))$ with the weak topology with respect to the weak topology on $\mc P([0,1])$. See~\cite{levin2025graphs} for more details.} on $[0,1]$ by~\cite[Props.~3.3, 4.1]{levin2025graphs}, and is compact by Proposition~\ref{prop:compact_under_sampling}.

    The sequence of $\ell_1$ balls $(\Omega_n=\mc B_{\ell_1}^{(n)}(r))$ of radius $r$ is also closed under both binning and species sampling, and its limit space can similarly be identified with a compact space of (signed) random measures.
\end{example}

Combining Proposition~\ref{prop:compact_under_sampling} and Theorem~\ref{thm:general_equivalence}, we conclude that if $(\Omega_n)$ is closed under sampling then there are uniform rates at which we can approximate arbitrary-sized objects by fixed-sized ones (as in Theorem~\ref{thm:general_equivalence}(2)), and uniform any-dimensional sketching and generalization rates for Lipschitz functions measured by worst-case error (as in Theorem~\ref{thm:general_equivalence}(3-4)). We derive explicit such rates for several notions of sampling in Sections~\ref{sec:sampling_wrep} and~\ref{sec:quotients} below after introducing the relevant sampling maps.

We further seek uniform rates for generalization measured by average error with respect to any-dimensional data distributions. According to Proposition~\ref{prop:avg_error}, such uniform rates are available for sequences of measures $(\mu_n\in\mc P(\Omega_n))$ with a weak limit $\mu_{\infty}\in\mc P(\overline{\Omega}_{\infty})$. However, there are many sequences of measures converging to the same limit $\mu_{\infty}$, and we seek a `canonical' choice of such a sequence to model data distributions of growing dimensionality. 
Under the sampling metric, there is a natural family of such sequences of distributions. Specifically, each $\mu_{\infty}\in\mc P(\overline{\Omega}_{\infty})$ defines a sequence of distributions of its samples $(\msf S_k(X))$ where $X\sim\mu_{\infty}$ is independent of the sampling maps $\msf S_k$. This sequence of distributions converges weakly back to $\mu_{\infty}$ whenever the sampling maps $(\msf S_k)$ yield increasingly-good approximations of their inputs, in the sense that $\msf S_k(x)\to x$ weakly as $k\to\infty$ for any $x\in\overline{\Omega}_{\infty}$. This is the case for all the sampling maps considered in the paper. 

We now state the above result formally. For the following result, we endow $\mc P(\overline{\Omega}_{\infty})$ with the weak topology, and endow $\prod_n\mc P(\Omega_n)$ with the product of weak topologies. We also observe that the sampling maps $(\msf S_k\colon\bigsqcup_n\Omega_n\to\Omega_k)$ extend to $\overline{\Omega}_{\infty}$. Indeed, if $x\in\overline{\Omega}_{\infty}$ then we can write $x=\lim_ix_i$ for $(x_i)\subseteq\bigsqcup_n\Omega_n$ and define $\msf S_k(x)$ to be the weak limit of $(\msf S_k(x_i))_i$ for each $k\in\NN$. It is easy to check that this is well-defined by definition of convergence in $d_{\mathrm{samp}}$.
\begin{proposition}\label{prop:fd_measures_as_limits}
    In the setting of Proposition~\ref{prop:compact_under_sampling}, assume that
    \begin{equation}\label{eq:convergence_of_samples}
        \lim_{n\to\infty}\mbb Ed_{\mathrm{samp}}(x,\msf S_n(x))=0,\quad \textrm{for each } x\in\overline{\Omega}_{\infty}.
    \end{equation}
    Then for any $\mu\in\mc P(\overline{\Omega}_{\infty})$ the sequence of measures $\mu_n=\mathrm{Law}(\msf S_n(X))$ where $X\sim \mu$ is independent of $\msf S_n$ converges weakly to $\mu$.
    Furthermore, the map $\mc P(\overline{\Omega}_{\infty})\to\prod_n\mc P(\Omega_n)$ sending 
    \begin{equation}\label{eq:fd_sampling_map}
        \mu\mapsto (\mu_n=\mathrm{Law}(\msf S_n(X)))_{n\in\NN},
    \end{equation} 
    is a linear isomorphism from $\mc P(\overline{\Omega}_{\infty})$ onto its image.
\end{proposition}
We give the proof of this result in Section~\ref{sec:proofs_general}.  We characterize the image of the map \eqref{eq:fd_sampling_map} for our generalizations of sampling with replacement and random binning in Theorems~\ref{thm:proj_consistent_laws} and~\ref{thm:eqp_consistent_laws} below.  The above sequences $(\mu_n)$ are precisely the ones we use in the sequel to quantify average generalization error in the sense of Proposition~\ref{prop:avg_error}.
Note that~\eqref{eq:convergence_of_samples} is the only assumption we have made on the sampling maps $(\msf S_k)$.

\begin{example}\label{ex:deFin1}
    In the setting of Example~\ref{ex:sampling_entries} with $\Omega_n=\Theta^n$ for compact $\Theta\subseteq\RR^d$, we have $\overline{\Omega}_{\infty}=\mc P(\Theta)$ so limit objects are probability distributions on $\Theta$; hence $\mc P(\overline{\Omega}_{\infty})$ corresponds to random probability measures on $\Theta$. 
    The map~\eqref{eq:fd_sampling_map} sends a distribution over measures $\msf M\in\mc P(\mc P(\Theta))$ to the mixture of iid distributions $\mathrm{Law}(\msf S_n(X))=\int \mu^{\otimes n}\, d\msf M(\mu)$ where $X\sim\msf M$. We show in Theorem~\ref{thm:W1_bound_sampling} below that~\eqref{eq:convergence_of_samples} holds for sampling with replacement, and hence Proposition~\ref{prop:fd_measures_as_limits} yields a linear isomorphism between random measures and mixtures of iid distributions, recovering a part of de Finetti's theorem. De Finetti's theorem further states that the space of such mixtures is isomorphic to the space of infinite exchangeable arrays, a result we generalize in Theorem~\ref{thm:proj_consistent_laws} to our broader notion of sampling with replacement.
%
%
\end{example}

\begin{remark}[Extremality]
    We remark that the map~\eqref{eq:fd_sampling_map} identifies limit objects in $\overline{\Omega}_{\infty}$ with extremal sequences of measures obtained from sampling. Indeed, elements of $\overline{\Omega}_{\infty}$ are precisely the extreme points of $\mc P(\overline{\Omega}_{\infty})$, which are mapped isomorphically to extreme points of the image of the map~\eqref{eq:fd_sampling_map}. As an illustration, in the setting of Example~\ref{ex:deFin1} we recover the fact that sequences of iid distributions obtained from deterministic measures on $\Theta$ are precisely the extreme points of all sequences of mixtures of iid distributions. We recover further results from the literature pertaining to such extremality in Sections~\ref{sec:sampling_wrep} and~\ref{sec:quotients}.
\end{remark}

Having studied sampling maps in general, we now describe three specific families of sampling maps generalizing the examples in Section~\ref{sec:intro}.

\subsection{General Sampling Maps}\label{sec:sampling_general}
In this section, we formalize our generalizations of sampling with replacement, random binning, and species sampling from Examples~\ref{ex:sampling_entries_intro}-\ref{ex:species_intro} in Section~\ref{sec:intro}, and we explain when each sampling map is appropriate by characterizing their symmetries and relations between dimensions. We begin in Section~\ref{sec:sampling_and_binning} by generalizing sampling with replacement and random binning, which are dual to each other in a precise sense and can be analyzed together. Then in Section~\ref{sec:species} we generalize species sampling.

\subsubsection{Sampling with Replacement and Random Binning}\label{sec:sampling_and_binning}


As explained in Section~\ref{sec:intro}, the key to our general perspective on sampling with replacement and random binning is an action of maps between finite sets on index sets of vectors, which we now make precise.
\begin{definition}[Compatible index sets]\label{def:FS_index}
    A (FinSet-)\emph{compatible sequence of index sets} is a sequence of finite sets $(\mc I_n)$ together with maps $\theta(f_{n,k})\colon\mc I_k\to\mc I_n$ associated to each $f_{n,k}\colon[k]\to[n]$ which satisfies (i) $\theta(\mathrm{id}_{[n]})=\mathrm{id}_{\mc I_n}$; and (ii) $\theta(f_1\circ f_2)=\theta(f_1)\circ\theta(f_2)$ whenever the composition is well-defined.  The collection of maps $\theta = \{\theta(f_{n,k}) : f_{n,k}\colon[k]\to[n]\}$ is called an \emph{action}.\footnote{In the language of category theory, this is a functor from the category $\FS$ of finite sets to itself.}
\end{definition}
As a consequence of Definition~\ref{def:FS_index}, if $n\leq N$ and $\phi\colon[n]\to[N]$ is an injection, then $\theta(\phi)\colon\mc I_n\to\mc I_N$ is an injection as well. Indeed, there is a surjection $\psi\colon[N]\to[n]$ satisfying $\psi\circ\phi=\mathrm{id}_{[n]}$, in which case $\theta(\psi)\circ\theta(\phi)=\mathrm{id}_{\mc I_n}$ so $\theta(\phi)$ is injective. Likewise, if $\psi\colon[N]\to[n]$ is surjective then $\theta(\psi)\colon\mc I_N\to\mc I_n$ is surjective.

Examples of such compatible sequences include $\mc I_n=[d]\times [n]$ with action $\theta(f)(i,j)=(i,f(j))$, tuples $\mc I_n=[n]^d$ with action $\theta(f)(i_1,\ldots,i_d)=(f(i_1),\ldots,f(i_d))$, and multisets $\mc I_n=[n]^d/\mfk S_d$, viewed as the set of orbits under permutations of the $d$ indices in each tuple of $[n]^d$, with the same action. 
More generally, standard operations on sets can be applied to create new compatible sequences from previous ones.
For example, if $\{(\mc I_n),\theta_I\}$ and $\{(\mc J_n),\theta_J\}$ are two compatible sequences, then so are $\{(\mc I_n\sqcup\mc J_n),\theta_I\sqcup\theta_J\}$ and $\{(\mc I_n\times\mc J_n),\theta_I\times\theta_J\}$. 

All of the compatible sequences of index sets in this paper can be derived by applying such standard operations to the basic sequence $(\mc I_n=[n])$.  For compatible sequences obtained in this manner, we quantify the complexity of the sequence using the following notion of a `degree'.  This degree appears explicitly in the sequel in our rates for any-dimensional sketching and generalization.
\begin{definition}[Degree of compatible index sets]\label{def:FS_index_degree}
Consider a compatible sequence $(\mc I_n)$ of the form $\mc I_n=\bigsqcup_{m=1}^M[n]^{d_m}/H_m$, where $H_m\subseteq\mfk S_{d_m}$ is a subgroup of permutations acting by permuting the coordinates in each multi-index in $[n]^{d_m}$.  The \emph{degree} of such a compatible sequence is given by $D=\max_md_m$.\footnote{Our notion of degree is closely related to the degree of a polynomial functor~\cite{Niu_Spivak_2025}.}
\end{definition}
As illustrations, the degree of $\mc I_n=[d]\times [n]=\bigsqcup_{i=1}^d[n]$ is one, while the degrees of $\mc I_n=[n]^2$, $\mc I_n=[n]^2\sqcup[n]$, and $\mc I_n=[n]^2/\mfk S_2$ are all two.

As discussed in Section~\ref{sec:intro}, we associate to each compatible sequence $(\mathcal{I}_n)$ a sequence of vector spaces $(\vct V_n=\RR^{\mc I_n})$, and let maps between finite sets act on these vector spaces via the maps $\rho$ and $\beta$ given by~\eqref{eq:sampling_action} and~\eqref{eq:binning_action}, respectively.  
The maps $\rho$ and $\beta$ in turn yield generalizations of sampling with replacement $(\msf R_k)$ and random binning $(\msf B_k)$ in~\eqref{eq:sampling_general} and~\eqref{eq:binning_general}, respectively.
\begin{example}[Vector spaces from index sets]\label{ex:vecs_from_inds}
    The following are examples of vector spaces obtained from compatible sequences.
    \begin{enumerate}[font=\emph, align=left, wide, labelwidth=!, labelindent=0pt]
        \item[(Sequences)] When $\mc I_n=[d]\times[n]$, we have $\vct V_n=\RR^{d\times n}$ whose elements we view as sequences of $n$ vectors in $\RR^d$. In this case, if $f\colon [m]\to[n]$ then $\rho(f)\colon\vct V_n\to\vct V_m$ extracts the $m$ columns $f(1),\ldots,f(m)$, while $\beta(f)\colon\vct V_m\to\vct V_n$ bins $m$ columns into $n$ bins $f^{-1}(1),\ldots,f^{-1}(n)$. The map $\msf R_k$ samples $k$ columns with replacement as in Example~\ref{ex:sampling_entries_intro}, while $\msf B_k$ randomly bins the columns into $k$ bins and sums all columns in each bin as in Example~\ref{ex:binning_intro}.

        \item[(Graphs)] When $\mc I_n=[n]^2$, we have $\vct V_n=\RR^{n\times n}$ whose elements we view as adjacency matrices of (weighted, directed) graphs on $n$ vertices, and when $\mc I_n=[n]^2/\mfk S_2$, we have $\vct V_n=\{X\in\RR^{n\times n}:X^\top=X\}$ whose elements we view as undirected graphs on $n$ vertices. In either case, if $f\colon[m]\to[n]$ then $\rho(f)$ extracts the induced subgraph on the $m$ vertices $f(1),\ldots,f(m)$ while $\beta(f)$ forms the quotient graph defined by binning $m$ vertices into $n$ bins $f^{-1}(1),\ldots,f^{-1}(n)$. When $\mc I_n=[n]^2$, the adjacency matrix of the quotient graph is given by 
        \begin{equation*}
            (\beta(f)X)_{i,j}=\sum_{\substack{k\in f^{-1}(i)\\ \ell\in f^{-1}(j)}}X_{k,\ell}.
        \end{equation*}
        The map $\msf R_k$ extracts the induced subgraph on $k$ vertices sampled with replacement as in Example~\ref{ex:sampling_vertices_intro}, while $\msf B_k$ forms the quotient graph defined by randomly binning vertices into $k$ bins as in~\cite{levin2025graphs}.

        \item[(Graph signals)] When $\mc I_n=[n]^2\sqcup([n]\times [d])$, we have $\vct V_n=\RR^{n\times n}\oplus \RR^{n\times d}$ whose elements we view as pairs $(G,X)$ of a graph $G$ on $n$ vertices and a graph signal $X$ assigning $d$-dimensional features to each vertex~\cite{graph_signal_processing}. In this case, if $f\colon[m]\to[n]$ then $\rho(f)(G,X)=(\rho(f)G,\rho(f)X)$ extracts the induced subgraph on the $m$ vertices $f(1),\ldots,f(m)$ along with their corresponding features, while $\beta(f)(G,X)=(\beta(f)G,\beta(f)X)$ forms the quotient graph defined by binning $m$ vertices into $n$ bins $f^{-1}(1),\ldots,f^{-1}(n)$ and sums the features of vertices in the same bin. The map $\msf R_k$ samples $k$ vertices with replacement along with their corresponding features, while $\msf B_k$ randomly bins the vertices into $k$ bins, forms the quotient graph, and sums all features of vertices in the same bin.
    \end{enumerate}
    When $\mc I_n=[n]^d$ or $[n]^d/\mfk S_d$, the above examples generalize to directed and undirected hypergraphs.
\end{example}
Note that we have two sampling maps $(\msf R_k)$ and $(\msf B_k)$ defined on the same sequences of vector spaces. We explain in Examples~\ref{ex:duplication} and~\ref{ex:zero_padding} when each of these sampling maps is appropriate.

We remark that the actions $\rho$ and $\beta$ from~\eqref{eq:sampling_action} and~\eqref{eq:binning_action}, respectively, respect compositions of maps between finite sets. Specifically, we have $\rho(f_1\circ f_2)=\rho(f_2)\circ \rho(f_1)$ while $\beta(f_1\circ f_2)=\beta(f_1)\circ\beta(f_2)$ whenever the composition $f_1\circ f_2$ is well-defined, so $\rho$ reverses the order of compositions while $\beta$ preserves it.\footnote{Such composition-reversing and composition-preserving linear actions of maps between finite sets were called $\coFS$- and $\FS$-representations, respectively, in~\cite{levin2025deFin}.} We will exploit this compositionality in what follows.  

\subsubsection*{When is each sampling map appropriate?}
As discussed in Section~\ref{sec:intro}, the sampling maps $(\msf R_k)$ and $(\msf B_k)$ induce equivalence relations between distinct objects, meaning that $d_{\mathrm{samp}}(x,y)=0$ for $x\neq y$ when $d_{\mathrm{samp}}$ is defined using either choice of sampling maps.  Importantly, we will see that distinct objects of the same size and also of different sizes can be equivalent under $d_{\mathrm{samp}}$. We characterize the equivalences induced by sampling with replacement and random binning, and discuss some of the application domains in which these equivalences are natural.

For objects of the same size, equivalences arise from permutation symmetry.
To see this symmetry, observe that if $\pi\colon[n]\to[n]$ is a permutation (i.e., a bijection), then $\pi$ defines a bijection from $\mc I_n$ to itself via $\theta$, and two isomorphisms of $\vct V_n$ to itself via $\rho$ and $\beta$. These two actions of permutations on $\vct V_n$ define representations of the group $\mfk S_n$ on $\vct V_n$; moreover, the actions are dual to each other and the associated representations are isomorphic.  When the context is clear, we denote the action of a permutation $\pi\in\mfk S_n$ on a vector $x\in\vct V_n$ by $\pi x$, which equals either $\rho(\pi)x$ or $\beta(\pi)x$ depending on whether $(\vct V_n)$ is endowed with sampling with replacement or with random binning maps, respectively.
\begin{proposition}[Permutation symmetry]\label{prop:symmetry_coFS}
    Suppose $(\mc I_n)$ is a compatible sequence, let $(\vct V_n=\RR^{\mc I_n})$ be the associated vector spaces endowed with either sampling with replacement or random binning, denoted by $(\msf S_k)$. Fix $x \in \vct V_n$.
    \begin{enumerate}[font=\emph, align=left,labelwidth=!, labelindent=0pt]
        \item[(Exchangeability)] The distribution $\mathrm{Law}(\msf S_k(x))$ is exchangeable, meaning $\pi_k \msf S_k(x)\overset{d}{=}\msf S_k(x)$ for each $\pi_k\in\mfk S_k$.
        \item[(Permutation invariance)] For any $\pi_n\in\mfk S_n$, we have $\msf S_k(\pi_nx) \overset{d}{=}\msf S_k(x)$ for all $k$, hence $d_{\mathrm{samp}}(x,\pi_nx)=0$.
    \end{enumerate}
\end{proposition}
\begin{proof}
    For any $k,\ell\in\NN$, let $F_{\ell,k}\colon[k]\to[\ell]$ be a uniformly random map, and note that $F_{\ell,k}\circ\pi_k\overset{d}{=}\pi_\ell\circ F_{\ell,k}\overset{d}{=}F_{\ell,k}$ since all three maps send each $i\in[k]$ independently to a uniformly random element of $[\ell]$. Applying $\rho$ and $\beta$, we conclude that $\pi_k\circ\msf S_k\overset{d}{=}\msf S_k\circ\pi_n\overset{d}{=}\msf S_k$ for any $\pi_k\in\mfk S_k$ and $\pi_n\in\mfk S_n$, as claimed.
\end{proof}


In addition to permutation symmetry in each fixed dimension, there are also equivalences between objects of different sizes.
We characterize these equivalences using two distinguished maps between finite sets. The first set are the inclusions $\iota_{N,n}\colon[n]\hookrightarrow [N]$ defined as 
\begin{equation*}
    \iota_{N,n}(i) = i \quad \textrm{for } i \in [n],\quad \textrm{whenever } n\leq N.
\end{equation*}
These inclusions in turn define inclusions of $\mc I_n$ into $\mc I_N$ via the action $\theta$. 
Then the map $\rho(\iota_{N,n})\colon\vct V_N\to\vct V_n$ acts as the orthogonal projection with respect to the usual inner product, while $\beta(\iota_{N,n})\colon\vct V_n\to\vct V_N$ is an embedding by zero-padding.  For example, when $\vct V_n = \RR^n$ the projection $\rho(\iota_{N,n})$ extracts the first $n$ entries from a length-$N$ vector, while $\beta(\iota_{N,n})$ zero-pads a vector of length $n$ by $N-n$ zeros. Similarly, when $\vct V_n = \RR^{n \times n}$ the projection $\rho(\iota_{N,n})$ extracts the induced subgraph on the first $n$ vertices from a graph on $N$ vertices, while the embedding $\beta(\iota_{N,n})$ corresponds to appending $N-n$ isolated vertices to a graph on $n$ vertices. See Example~\ref{ex:vecs_from_inds} for these and other examples.

The second distinguished set of maps are the consecutive partitions $\kappa_{n,N}\colon[N]\to[n]$ defined by
\begin{equation}\label{eq:standard_equipartition}
    \kappa_{n,N}(j + (i-1)(N/n)) = i \quad \textrm{for } i\in[n], j\in[N/n],\quad \textrm{whenever } n|N.
\end{equation}
The corresponding linear maps $\rho(\kappa_{n,N})\colon\vct V_n\to\vct V_N$ are embeddings that duplicate the entries of their inputs in a suitable sense based on the action $\theta$.  For example, we have $\rho(\kappa_{n,N})x=x\otimes\mathbbm{1}_{N/n}$ when $\vct V_n=\RR^n$, which acts by duplicating each entry of a vector $x$ of length $n$ the same number $N/n$ of times. Similarly, we have $\rho(\kappa_{n,N})X=X\otimes\mathbbm{1}_{N/n}\mathbbm{1}_{N/n}^\top$ when $\vct V_n=\RR^{n\times n}$, which acts by duplicating each entry of a matrix $X$ of size $n\times n$ into a block of $(N/n)\times(N/n)$ entries.
Meanwhile, the maps $\beta(\kappa_{n,N})\colon\vct V_N\to\vct V_n$ are projections that sum consecutive blocks of coordinates. For example, the vector $\beta(\kappa_{n,N})x\in\RR^n$ is obtained from $x\in\RR^N$ by summing each group of consecutive $N/n$ coordinates. Similarly, the matrix $\beta(\kappa_{n,N})X\in\RR^{n\times n}$ is obtained from $X\in\RR^{N\times N}$ by summing the entries of each consecutive $N/n\times N/n$ block in $X$. See Example~\ref{ex:vecs_from_inds}.

We are now ready to state the relationships across dimensions satisfied by our sampling maps.
\begin{proposition}[Relations between Dimensions]\label{prop:relations_coFS}
    Suppose $(\mc I_n)$ is a compatible sequence, and let $(\vct V_n=\RR^{\mc I_n})$ endowed with sampling with replacement $(\msf R_k)$ and random binning $(\msf B_k)$.
    \begin{enumerate}[font=\emph, align=left]
        \item[(Inclusions)] For any $n\leq N$, we have $\rho(\iota_{N,n})\circ\msf R_N\overset{d}{=}\msf R_n$ and $\msf B_k\circ\beta(\iota_{N,n})\overset{d}{=}\msf B_k$ for all $k\in\NN$. 
        \item[(Consecutive partitions)] For any $n|N$, we have $\msf R_k\circ\rho(\kappa_{n,N}) \overset{d}{=}\msf R_k$ for all $k\in\NN$ and $\beta(\kappa_{n,N})\circ\msf B_N\overset{d}{=}\msf B_n$.
    \end{enumerate}
\end{proposition}
\begin{proof}
    For the first claim, fix $n\leq N$, let $k\in\NN$ be arbitrary and note that $F_{k,N}\circ\iota_{N,n}\overset{d}{=}F_{k,n}$ since both maps send each $i\in[n]$ independently to a uniformly random element of $[k]$. Applying $\rho$ and $\beta$ to this identity yields the first claim. For the second claim, note that $\kappa_{n,N}\circ F_{N,k}\overset{d}{=}F_{n,k}$ since each $i\in[k]$ is sent by $F_{N,k}$ independently to a uniformly random element of $[N]$, so its image is equally likely to lie in each fiber of $\kappa_{n,N}$. Applying $\rho$ and $\beta$ to this identity yields the second claim.
\end{proof}

In particular, note that for any $x\in\vct V_n$, we have $d_{\mathrm{samp}}(x,\rho(\kappa_{n,N})x)=0$ if $d_{\mathrm{samp}}$ is defined using $(\msf R_k)$, while $d_{\mathrm{samp}}(x,\beta(\iota_{N,n})x)=0$ for any $x\in\vct V_n$ if $d_{\mathrm{samp}}$ is defined using $(\msf B_k)$.  Note here that $\rho(\kappa_{n,N})x$ and $\beta(\iota_{N,n})x$ belong to $\vct V_N$, so the sampling distance between objects of different sizes can equal zero.

Observe that the maps $\iota_{N,n}$ are particular types of injections, while the maps $\kappa_{n,N}$ are particular types of equipartitions.\footnote{An equipartition is a map with equal-sized fibers.}  
By combining Propositions~\ref{prop:symmetry_coFS} and~\ref{prop:relations_coFS}, we are able to obtain the following consequences for all injections and equipartitions:
\begin{equation}\label{eq:injection_equivalence}
    \rho(\phi_{N,n})\circ\msf R_N\overset{d}{=}\msf R_n\quad  \textrm{and} \quad \msf B_k\circ\beta(\phi_{N,n})\overset{d}{=}\msf B_k\quad  \textrm{for any injection $\phi_{N,n}\colon[n]\to[N]$},
\end{equation}
and
\begin{equation}\label{eq:eqp_equivalence}
    \msf R_k\circ\rho(\psi_{n,N})\overset{d}{=}\msf R_k\quad \textrm{and}\quad \beta(\psi_{n,N})\circ\msf B_N\overset{d}{=}\msf B_n\quad \textrm{for any equipartition } \psi_{n,N}\colon[N]\to[n],  
\end{equation}
since we can write $\phi_{N,n}=\pi_N\circ \iota_{N,n}$ and $\psi_{n,N}=\kappa_{n,N}\circ\pi_N$ for some permutation $\pi_N\in\mfk S_N$.

As a converse to the preceding equivalences, we now prove that if $d_{\mathrm{samp}}(x,y)=0$ then $x,y$ must differ by permutations and duplication for sampling with replacement, or permutations and zero-padding for random binning.
\begin{proposition}\label{prop:equivalence_converse}
    Suppose $(\mc I_n)$ is a compatible sequence, and let $(\vct V_n=\RR^{\mc I_n})$ endowed with sampling with replacement $(\msf R_k)$ and random binning $(\msf B_k)$. Let $x\in\vct V_n$ and $y\in\vct V_N$ for $n\leq N$.
    \begin{enumerate}[labelwidth=!, labelindent=0pt]
        \item If $\msf R_k(x)\overset{d}{=}\msf R_k(y)$ for all $k\in\NN$, then $\rho(\kappa_{N,L})y=\rho(\psi_{n,L})x$ for $L=\mathrm{lcm}\{n,N\}$ and some equipartition $\psi_{n,L}\colon[L]\to[n]$.
        \item If $\msf B_k(x)\overset{d}{=}\msf B_k(y)$ for all $k\in\NN$, then $y=\beta(\phi_{N,n})x$ for some injection $\phi_{N,n}\colon[n]\to[N]$.
    \end{enumerate}
\end{proposition}
The proof uses concentration results for uniformly random maps between finite sets. In particular, the proof of part 1 entails showing that a bipartite graph constructed from fibers of these random maps contains a perfect matching. We defer the full proof to Section~\ref{sec:proofs_general}. 
We now give some examples of application domains where each of the above equivalences naturally arises.
\begin{example}[Duplication]\label{ex:duplication}
    If $\vct V_n=\RR^{d\times n}$ and $x=(x_1,\ldots,x_n)\in\vct V_n$ is viewed as a point cloud in $\RR^d$, then reordering the columns or duplicating them the same number of times yields an equivalent point cloud (e.g., the corresponding discrete distribution $\frac{1}{n}\sum_{i=1}^n\delta_{x_i}$ is unchanged). In the context of mean-field games, we can view the columns of $x$ as strategies played by $n$ symmetric players, and in the mean-field limit one can show that duplicating the columns of $x$ yields a strategy with the same payoff~\cite[\S2]{cardaliaguet2010notes}. Sampling columns with replacement would therefore be appropriate for either of these types of data.

    If $\vct V_n=\RR^{n\times n}$ and we view $x\in\vct V_n$ as adjacency matrices of graphs, then simultaneously permuting the rows and columns of $x$ corresponds to relabelling the vertices, which does not change the topology of the graph. Furthermore, duplicating each vertex of $x$ corresponds to the so-called blowup of the graph, and these blowups have the same graph homomorphism densities~\cite{convergent_seqs1}. Sampling vertices with replacement would therefore be appropriate when studying such structural properties of graphs.
\end{example}
\begin{example}[Zero-padding]\label{ex:zero_padding}
    If $\vct V_n=\RR^n$ and we view $x\in\vct V_n$ as a (signed) distribution on $n$ items, then reordering the entries of $x$ corresponds to relabelling the items, while zero-padding $x$ amounts to viewing it as a distribution on a larger collection of items assigning zero weight to all but the original $n$ ones. Such relabelling and zero-padding preserves many structural properties of $x$, such as various diversity indices measuring its dispersion~\cite{Zhang02072016}, and hence random binning is appropriate when studying such properties.
    
    If $\vct V_n=\RR^{n\times n}$ and we view $x\in\vct V_n$ as an adjacency matrix of a graph, then permuting and zero-padding $x$ amounts to relabelling and adding isolated vertices to the graph, respectively. These operations again preserve many properties of the graph in question, including homomorphism numbers and various measures of centrality and clustering~\cite{levin2025graphs}, so random binning of vertices is appropriate when studying such properties.
\end{example}

If $f\colon\bigsqcup_n\vct V_n\to\RR$ is continuous with respect to the sampling metric defined by $(\msf R_k)$, then 
\begin{equation}\label{eq:eqp_invariant}
    f(\rho(\psi_{n,N})x)=f(x) \quad \textrm{for all } x\in\vct V_n \textrm{ and equipartitions } \psi_{n,N}\colon[N]\to[n],
\end{equation}
or equivalently, each restriction $f|_{\vct V_n}$ is permutation-invariant and $f$ is unchanged by duplication of entries defined by $\rho(\kappa_{n,N})$. 
Likewise, if $f$ is continuous with respect to the sampling metric defined by $(\msf B_k)$, then
\begin{equation}\label{eq:zero_pad_invariant}
    f(\beta(\phi_{N,n})x)=f(x) \quad \textrm{for all } x\in\vct V_n \textrm{ and injections } \phi_{N,n}\colon[n]\to[N],
\end{equation}
or equivalently, each restriction $f|_{\vct V_n}$ is permutation-invariant and $f$ is unchanged by zero-padding defined by $\beta(\iota_{N,n})$. Thus, functions continuous in the sampling metric defined using either choice of sampling maps respect the same equivalence relations.

The converses of~\eqref{eq:eqp_invariant} and~\eqref{eq:zero_pad_invariant} are false in general, that is, there are functions invariant under equipartitions or injections that are not continuous in the corresponding sampling metric. For example, the norm $f(x)=\|x\|_1$ satisfies~\eqref{eq:zero_pad_invariant} but we show in Corollary~\ref{cor:deepsets} that it is not continuous with respect to random binning. 
However, the converses are true on bounded domains if $f|_{\vct V_n}$ is a polynomial of some degree $d$ for all $n$. In this case, if $f$ is unchanged by equipartitions in the sense of~\eqref{eq:eqp_invariant} then there exists $k\in\NN$ and a fixed-dimensional polynomial $g_k\in\RR[\vct V_k]$ such that $f(x)=\mbb Eg_k(\msf R_k(x))$ for all $x$ by~\cite[Thm.~5.2]{levin2025deFin}. Similarly, if $f$ is unchanged by zero-padding in the sense of~\eqref{eq:zero_pad_invariant} then there exists $k\in\NN$ and $g_k\in\RR[\vct V_k]$ such that $f(x)=\mbb Eg_k(\msf B_k(x))$ for all $x$ by the same theorem. In other words, these polynomials are computing a moment of a fixed-dimensional sample of their input under either sampling with replacement or random binning. We exploit this sampling representation for polynomials in Sections~\ref{sec:sampling_wrep} and \ref{sec:quotients}.

\subsubsection{Species Sampling}\label{sec:species}
We now generalize the species sampling map from Example~\ref{ex:species_intro}.
Fix a compatible sequence $(\mc I_n)$ and consider the corresponding sequence $(\vct V_n=\RR^{\mc I_n})$ of vector spaces.
For simplicity, suppose first that $x\in\Delta^{\mc I_N}$ is a simplex element.  
We view such $x$ as a probability distribution over $\mc I_N$, and consider forming an empirical approximation to it from $n$ iid samples. Specifically, sampling $\alpha_1,\ldots,\alpha_n\in\mc I_N$ iid according to $x$, we set 
\begin{equation}\label{eq:empirical_sampling}
    \mscr E_n(x)=\frac{1}{n}\sum_{i=1}^ne_{\alpha_i}.
\end{equation}
In general, we have $\mscr E_n(x)\in \vct V_N$ and hence $\mscr E_n$ does not reduce the dimensionality of its input. However, it does produce a sparse vector, as $\mscr E_n(x)$ has at most $n$ nonzero entries regardless of the dimensionality of $x$. We can exploit this sparsity by relabeling the elements of the index set to produce a low-dimensional element defining our generalization $\msf E_n(x)$ of species sampling. We next describe this construction formally.

Suppose first that $\mc I_n=[n]^D$ for all $n$ and $x\in\Delta^{\mc I_N}$ as above.
\begin{enumerate}[labelwidth=!, labelindent=0pt]

    \item Sample $\alpha_1,\ldots,\alpha_n$ from the distribution $x$ on $\mc I_N=[N]^D$. Denote $\alpha_i=(j_1^{(i)},\ldots,j_D^{(i)})$ for each $i$.

    \item Consider the subset $\bigcup_{i=1}^n\{j_1^{(i)},\ldots, j_D^{(i)}\}$ of $[N]$, and denote its cardinality by $k\leq nD$. Uniformly randomly enumerate the elements of this subset to obtain $t_1,\ldots,t_k$.

    \item Set
    \begin{equation}\label{eq:species_sampling_general}
        \msf E_n(x)_{\ell_1,\ldots,\ell_D}= \mscr E_n(x)_{t_{\ell_1},\ldots,t_{\ell_D}},
    \end{equation}
    if $\ell_1,\ldots,\ell_D\in[k]$, and $\msf E_n(x)_{\ell_1,\ldots,\ell_D}=0$ otherwise. 
\end{enumerate}
In words, we view the elements of $[N]$ as labeling distinct species, and elements of $\mc I_N = [N]^D$ as relations between these species. The empirical sampling map $\mscr E_n$ samples these relations while `remembering' the species labels, while the species sampling map $\msf E_n$ `forgets' the labels of the species in steps 2 and 3 of the above construction. 
For example, if $\mc I_N=[N]$ then $\msf E_n$ recovers the species sampling map of Example~\ref{ex:species_intro}. Furthermore, if $\mc I_N=[N]^2$ and we view elements of $\vct V_N=\RR^{N\times N}$ as adjacency matrices of (weighted, directed) graphs on $N$ vertices, then $\msf E_n$ can be viewed as sampling $n$ edges proportionally to their edge weights, and has been studied in connection with limits of large graphs in~\cite{levin2025graphs}.

The above construction of $\msf E_n(x)$ generalizes to any compatible sequence $(\mc I_n)$ of finite degree and any $x\in\vct V_N$, not necessarily a simplex element, as follows.
\begin{itemize}
    \item If $\mc I_n=[n]^D/H$ for some subgroup $H\subseteq\mfk S_D$, then each $\alpha_i$ sampled in step 1 above is the orbit of a tuple, denoted $\alpha_i=H(j_1^{(i)},\ldots,j_D^{(i)})$. Because $H$ acts by permuting indices in a tuple, step 2 is well-defined. Moreover, if $(i_1,\ldots,i_D)=h(j_1,\ldots,j_D)$ for some $h\in H$ then $(t_{i_1},\ldots,t_{i_D})=h(t_{j_1},\ldots,t_{j_D})$. Thus, setting $\msf E_n(x)_{H(i_1,\ldots,i_D)}= \mscr E_n(x)_{H(t_{i_1},\ldots,t_{i_D})}$ is well-defined and generalizes~\eqref{eq:species_sampling_general} to this case.

    \item If $\mc I_n=\bigsqcup_{m=1}^M[n]^{d_m}/H_m$ is a general compatible sequence of degree $D=\max_md_m$, then each $\alpha_i$ sampled in step 1 above belongs to $[N]^{d_{m_i}}/H_{m_i}$ for some $m_i$. 
    Then in step 2 we randomly enumerate the set of all indices contained in all of the $\alpha_i$, and in step 3 we generalize~\eqref{eq:species_sampling_general} as above.  More precisely, denoting $\alpha_i=H_{m_i} (j_1^{(i)},\ldots,j_{d_{m_i}}^{(i)})$, in step 2 we randomly enumerate the elements of the subset $\bigcup_{i=1}^n\{j_1^{(i)},\ldots,j_{d_{m_i}}^{(i)}\}$ of $[N]$, which again contains $k\leq nD$ distinct elements $t_1,\ldots,t_k$. In step 3, we then set $\msf E_n(x)_{H_m(i_1,\ldots,i_{d_m})}=\mscr E_n(x)_{H_m(t_{i_1},\ldots,t_{i_{d_m}})}$, generalizing~\eqref{eq:species_sampling_general} to a general compatible sequence of index sets.

    \item We extend both $\mscr E_n$ and $\msf E_n$ beyond simplex vectors by homogeneity. Specifically, if $(\mc I_n)$ is a compatible sequence of finite degree and $x\in\vct V_N$, define $\mscr E_n(x)=\|x\|_1\mathrm{sign}(x)\odot \mscr E_n(|x|/\|x\|_1)$ where $\odot$ denotes entrywise product (and $\mscr E_n(0)=0$ by convention), and define $\msf E_n(x)$ as above.
\end{itemize}
We have thus generalized species sampling to general vectors indexed by elements of a compatible sequence of finite degree to obtain random maps $(\msf E_k\colon \bigsqcup_n\vct V_n\to \vct V_{kD})_k$.
\paragraph{When is species sampling appropriate?}
Species sampling induces the same equivalences between objects of different sizes as random binning, namely, equivalence under permutations and zero-padding. 
\begin{proposition}\label{prop:species_samp_zero_pad}
    Suppose $(\mc I_n)$ is a compatible sequence of finite degree and let $(\vct V_n=\RR^{\mc I_n})$ endowed with the species sampling maps $(\msf E_k)$. We have $\msf E_k\circ \beta(\phi_{N,n})\overset{d}{=}\msf E_k$ for any injection $\phi_{N,n}\colon[n]\to[N]$.
\end{proposition}

\begin{proof}
    Observe that $\beta(\phi_{N,n})\mscr E_k(x)\overset{d}{=}\mscr E_k(\beta(\phi_{N,n})x)$ as can be seen directly from~\eqref{eq:empirical_sampling} using the fact that $\phi_{N,n}$ is injective. The construction of $\msf E_k$ then shows that $\msf E_k(\beta(\phi_{N,n})x)\overset{d}{=}\msf E_k(x)$, as claimed. 
\end{proof}
In particular, if $d_{\mathrm{samp}}$ is defined using species sampling, then $d_{\mathrm{samp}}(x,\beta(\phi_{N,n})x)=0$ for all $x\in\vct V_n$ and all injections $\phi_{N,n}\colon[n]\to[N]$. 
Likewise, if $f\colon\bigsqcup_n\vct V_n\to\RR$ is continuous with respect to the sampling metric defined by $(\msf E_k)$, we have $f(\beta(\phi_{N,n})x)=f(x)$.
Conversely, it follows from Proposition~\ref{prop:relation_to_species} below that if $d_{\mathrm{samp}}(x,y)=0$ for $x\in\vct V_n$ and $y\in\vct V_N$ for $n\leq N$ then $y=\beta(\phi_{N,n})x$ for some injection $\phi_{N,n}$.

Thus, species sampling may be appropriate in the same domains as random binning, such as those in Example~\ref{ex:zero_padding}.
In fact, we show in Proposition~\ref{prop:relation_to_species} of Section~\ref{sec:quotients} that random binning and species sampling define closely related topologies via the corresponding sampling metrics.


\subsection{Missing Proofs from Section~\ref{sec:general}}\label{sec:proofs_general}
We begin by proving Theorem~\ref{thm:general_equivalence} via a direct argument.

\begin{proof}[Proof (Theorem~\ref{thm:general_equivalence}).]
    We start by arguing that properties 1 and 2 are equivalent. Since $\overline{\Omega}_{\infty}$ is a complete metric space by construction, it is compact if and only if it is totally bounded. 
    If property 2 holds, then for any $\epsilon>0$ there is $N(\epsilon)\in\NN$ such that $\sup_{x\in\overline{\Omega}_{\infty}}\mathrm{dist}(x,\Omega_{\leq N(\epsilon)})\leq \epsilon/2$. Since $\Omega_{\leq N(\epsilon)}$ is totally bounded, it has a finite $\epsilon/2$-net that is an $\epsilon$-net for $\overline{\Omega}_{\infty}$, proving that the latter is indeed totally bounded.
    Conversely, if $\overline{\Omega}_{\infty}$ is totally bounded and $\{y_1,\ldots,y_k\}\subseteq \overline{\Omega}_{\infty}$ is a finite $\epsilon/2$-net, let $x_i\in\bigsqcup_n\Omega_n$ satisfy $d(x_i,y_i)\leq\epsilon/2$ and note that $\{x_1,\ldots,x_k\}\subseteq\Omega_{\leq N(\epsilon)}$ for some $N(\epsilon)\in\NN$, hence $\sup_{x\in\overline{\Omega}_{\infty}}\mathrm{dist}(x,\Omega_{\leq N(\epsilon)})\leq \epsilon$. Letting $\epsilon\to0$ proves statement 2.
    Thus, statements 1 and 2 are equivalent.


    We proceed to argue that statements 2 and 3 are equivalent. If statement 2 holds, for any $(\epsilon_k>0)$ with $\epsilon_k\to 0$, define the (deterministic) map $\msf S_k\colon\bigsqcup_n\Omega_n\to\Omega_{\leq k}$ by finding $\msf S_k(x)\in\Omega_{\leq k}$ satisfying $d(x,\msf S_k(x))\leq \mathrm{dist}(x,\Omega_{\leq k})+\epsilon_k$. Then for any $f\in\mc F_1$ we have $\mathrm{e}_{\infty}(f,f\circ\msf S_k)\leq \sup_{x\in\Omega_{\infty}}\mathrm{dist}(x,\Omega_{\leq k})+\epsilon_k\to0$, so statement 3 holds. Conversely, suppose statement 3 holds and fix any $x\in\bigsqcup_n\Omega_n$. Then $f_x=d(x,\cdot)$ is 1-Lipschitz and 
    \begin{equation*}
        \mathrm{dist}(x,\Omega_{\leq k}) \leq d(x,\msf S_k(x)) = |f_x(\msf S_k(x))-f_x(x)| \leq \mathrm{e}_{\infty}(f_x,f_x\circ\msf S_k) \leq \sup_{f\in\mc F_1}\mathrm{e}_{\infty}(f,f\circ\msf S_k).
    \end{equation*}
    Since this holds for any $x\in\bigsqcup_n\Omega_n$, we conclude that statement 2 holds.
    
    Finally, we argue that statement 4 is equivalent to the previous equivalent three statements. If statement 3 holds, denote $R_n=\mbb E\sup_{f\in\mc F_1}\mathrm{e}_{\infty}(f,f\circ\msf S_n)$, and note that $R_n\to0$ because $(\sup_{f\in\mc F_1}\mathrm{e}_{\infty}(f,f\circ\msf S_k))_k$ is a sequence of bounded (by statement 1) random variables converging to zero almost surely, so they also converge to zero in expectation. Then for any $f,g\in\mc F_1$ and any $x\in\Omega_{\infty}$ we have
    \begin{equation*}
        |f(x)-g(x)|\leq |f(x)-\mbb Ef(\msf S_n(x))| + |\mbb Ef(\msf S_n(x))-\mbb Eg(\msf S_n(x))| + |\mbb Eg(\msf S_n(x))-g(x)|\leq \mathrm{e}_n(f,g) + 2R_n,
    \end{equation*}
    hence statement 4 holds. In particular, since $\Omega_n$ is bounded and $f,g\in\mc F_1$ we have $\mathrm{e}_n(f,g)<\infty$, so $\mathrm{e}_{\infty}(f,g)<\infty$ for all $f,g\in\mc F_1$. Conversely, suppose statement 4 holds. For any $n\in\NN$, let $f=\mathrm{dist}(\cdot,\Omega_{\leq n})$ and $g=0$, and note that 
    \begin{equation*}
        \sup_{x\in\Omega_{\infty}}\mathrm{dist}(x,\Omega_{\leq n}) = \mathrm{e}_{\infty}(f,g) \leq \mathrm{e}_n(f,g) + R_n = R_n,
    \end{equation*}
    which converges to zero as $n\to\infty$, hence statement 2 holds.
\end{proof}

Next, we prove Proposition~\ref{prop:avg_error}, likewise via a direct argument.
\begin{proof}[Proof (Proposition~\ref{prop:avg_error}).]
    Since $\overline{\Omega}_{\infty}$ is compact, we have $\mu_n\to\mu_{\infty}$ weakly if and only if $\lim_nW_1(\mu_n,\mu_{\infty})=0$, or equivalently, if and only if $\lim_n\sup_{f\in\mc F_1}|\mbb E_{\mu_{\infty}}f-\mbb E_{\mu_n}f|=0$. Thus, if $\mu_n\to\mu_{\infty}$ weakly then $\mathrm{e}_{\mu_{\infty}}(f,g)\leq \mathrm{e}_{\mu_n}(f,g)+2W_1(\mu_n,\mu_{\infty})$ since if $f,g\in\mc F_1$ then $\frac{1}{2}|f-g|\in\mc F_1$. Conversely, if such rate $R_n$ exists, note that
    \begin{equation*}\begin{aligned}
        W_1(\mu_n,\mu_{\infty}) &=\sup_{f\in\mc F_1}|\mbb E_{\mu_{\infty}}f-\mbb E_{\mu_n}f| = \max\Big\{\sup_{f\in\mc F_1}(\mbb E_{\mu_{\infty}}f-\mbb E_{\mu_n}f),\ \sup_{f\in\mc F_1}(\mbb E_{\mu_{\infty}}(-f)-\mbb E_{\mu_n}(-f))\Big\}\\
        &= \sup_{f\in\mc F_1}(\mbb E_{\mu_{\infty}}f-\mbb E_{\mu_n}f) = \sup_{f\in\mc F_1}(\mbb E_{\mu_{\infty}}(f-\inf_{x\in\Omega_{\infty}}f(x))-\mbb E_{\mu_n}(f-\inf_{x\in\Omega_{\infty}}f(x)))\\
        &= \sup_{\substack{f\in\mc F_1\\ f\geq 0}}(\mbb E_{\mu_{\infty}}f-\mbb E_{\mu_n}f) = \sup_{\substack{f\in\mc F_1\\ f\geq 0}}(\mathrm{e}_{\mu_{\infty}}(f,0)-\mathrm{e}_{\mu_n}(f,0)) \leq R_n\to0,
    \end{aligned}\end{equation*}
    and hence $\mu_n\to\mu_{\infty}$ weakly. Above, in going from the first to the second line we used the fact that $-f\in\mc F_1$ if and only if $f\in\mc F_1$, and the last inequality follows since $\mathrm{e}_{\mu}(f,0)=\mbb E_{\mu}f$ when $f\geq0$.
\end{proof}

We now prove Proposition~\ref{prop:compact_under_sampling} using Tychonoff's theorem.
\begin{proof}[Proof (Proposition~\ref{prop:compact_under_sampling}).]
    Endow each $\mc P(\Omega_n)$ with the weak topology, metrized by the $W_1$-metric with respect to the given norms on $\vct V_n$, and endow $\prod_n\mc P(\Omega_n)$ with the product topology, metrized by $d((\mu_k),(\nu_k))=\sum_{k\geq 1}2^{-k}W_1(\mu_k,\nu_k)$. Note that $\prod_n\mc P(\Omega_n)$ is compact by Tychonoff's theorem, and that the map sending $x\mapsto (\mathrm{Law}(\msf S_k(x)))_k$ is an isometric embedding of $\Omega_{\infty}$ into $\prod_n\mc P(\Omega_n)$. It therefore extends to an isometric embedding of $\overline{\Omega}_{\infty}$ as a closed, and therefore compact, subset of $\prod_n\mc P(\Omega_n)$. In particular, each $\Omega_n$ is totally bounded in $d_{\mathrm{samp}}$ since $\overline{\Omega}_{\infty}$ is totally bounded. The above argument further shows that convergence of $(x_i)$ in $d_{\mathrm{samp}}$ is equivalent to convergence of the sequence of sequences $\Big((\mathrm{Law}(\msf S_k(x_i)))_k\Big)_i$ in the product topology on $\prod_n\mc P(\Omega_n)$, which precisely means that $(\msf S_k(x_i))_i$ converges weakly for each $k$.
\end{proof}

We turn to proving Proposition~\ref{prop:fd_measures_as_limits}, relating limits in sampling metric to sequences of distributions.
\begin{proof}[Proof (Proposition~\ref{prop:fd_measures_as_limits}).]
    For the first claim, let $X\sim \mu$ and let $\msf S_n$ be independent of $X$ for each $n$. Then
    \begin{equation*}
        \limsup_{n\to\infty}W_1(X, \msf S_n(X))\leq \limsup_{n\to\infty}\mbb E_X\mbb E_{\msf S_n}d_{\mathrm{samp}}(X,\msf S_n(X)) \leq \mbb E\limsup_{n\to\infty}\mbb E_{\msf S_n}d_{\mathrm{samp}}(X,\msf S_n(X))=0,
    \end{equation*}
    where we used the fact that $d_{\mathrm{samp}}(X,\msf S_n(X))\leq \mathrm{diam}(\overline{\Omega}_{\infty})<\infty$ almost surely. This proves that $\msf S_n(X)\to X$ weakly in $\overline{\Omega}_{\infty}$, which is the first claim.
    
    For the second claim, it is clear that~\eqref{eq:fd_sampling_map} is linear, i.e., maps mixtures to mixtures, as $\mathrm{Law}(\msf S_n(X))=\int \mathrm{Law}(\msf S_n(x))\, d\nu(x)$ if $X\sim\nu$ is independent of $\msf S_n$.
    To see that it is continuous, note that if $\mu_i\to\mu_{\infty}$ weakly in $\mc P(\overline{\Omega}_{\infty})$, then there is a coupling of $X_i\sim \mu_i$ such that $X_i\to X_{\infty}$ in $d_{\mathrm{samp}}$ almost surely by Skorokhod's representation. We then get $\mathrm{Law}(\msf S_n(X_i))\to\mathrm{Law}(\msf S_n(X_{\infty}))$ for each $n$ where $\msf S_n$ is independent of the $X_i$.
    To see that~\eqref{eq:fd_sampling_map} is injective, suppose $\msf S_n(X)\overset{d}{=}\msf S_n(Y)$ for random $X,Y\in\overline{\Omega}_{\infty}$ and for all $n$ and take the weak limits of both sides as $n\to\infty$ to obtain $X\overset{d}{=}Y$ by the first claim.
    Thus, the map~\eqref{eq:fd_sampling_map} is indeed a linear isomorphism onto its image.
    %
\end{proof}

Finally, we prove Proposition~\ref{prop:equivalence_converse}.
\begin{proof}[Proof (Proposition~\ref{prop:equivalence_converse}).]
    For the first claim, it suffices to prove that if $n=N$ and $\msf R_k(x)\overset{d}{=}\msf R_k(y)$ for all $k$ then there exists $\pi\in\mfk S_n$ such that $y=\rho(\pi)x$. Applying this claim to $\rho(\kappa_{N,L})y, \rho(\kappa_{n,L})x\in\vct V_L$ yields the claim for general $x$ and $y$.
    
    Suppose $\msf R_k(x)\overset{d}{=}\msf R_k(y)$ for all $k$ and $n=N$. Then for each $k\in\NN$ there is a coupling $(F_{n,k},G_{n,k})$ of two uniformly random maps from $[k]$ to $[n]$ satisfying $\rho(F_{n,k})x=\rho(G_{n,k})y$. We argue that for sufficiently-large $k\geq n$, with positive probability we can find an injection $\phi_{k,n}\colon[n]\to[k]$ such that $G_{n,k}\circ\phi_{k,n}=\mathrm{id}_{[n]}$ and $F_{n,k}\circ\phi_{k,n}\in\mfk S_n$ is a permutation, in which case applying $\rho(\phi_{k,n})$ to both sides above yields $\rho(F_{n,k}\circ\phi_{k,n})x=y$ as desired. To this end, we define a bipartite graph with bipartition $[n]\sqcup[n]$ where $i,j\in[n]$ are connected if $F_{n,k}^{-1}(i)\cap G_{n,k}^{-1}(j)\neq\emptyset$. We prove that this graph has a perfect matching $\pi\in\mfk S_n$ with positive probability, in which case picking $\phi_{k,n}(i)\in G_{n,k}^{-1}(i)\cap F_{n,k}^{-1}(\pi(i))$ yields the desired injection. We prove the existence of this perfect matching by showing that Hall's condition holds with positive probability.

    Pick $\epsilon < \frac{1}{2n}$. For any subset $S\subseteq[n]$, we have with probability at least $1-2\exp(-2\epsilon^2k)$ that 
    \begin{equation*}
        \left||G_{n,k}^{-1}(S)| - \frac{|S|k}{n}\right| \leq \epsilon k,
    \end{equation*}
    since $|G_{n,k}^{-1}(S)|=\sum_{i=1}^k\mathbbm{1}[G_{n,k}(i)\in S]\sim\mathrm{Binom}(k,|S|/n)$. Therefore, with probability at least $1-2^{n+2}\exp(-2\epsilon^2k)$ the above holds both for $F_{n,k}$ and for $G_{n,k}$, and simultaneously for all subsets $S\subseteq[n]$. Choosing $k$ sufficiently large ensures that this probability is positive.
    For each $S\subseteq[n]$ define 
    \begin{equation*}
        N(S)=\{i\in[n]: G_{n,k}^{-1}(i)\cap F_{n,k}^{-1}(S)\neq\emptyset\},
    \end{equation*}
    and suppose that $|N(S)|<|S|$ for some such $S$. 
    Observe that $F_{n,k}^{-1}(S)\subseteq G_{n,k}^{-1}(N(S))$ by definition of $N(S)$ and hence 
    \begin{equation*}
        \frac{(|N(S)|+1)k}{n}-\epsilon k\leq\frac{|S|k}{n}-\epsilon k \leq |F_{n,k}^{-1}(S)|\leq |G_{n,k}^{-1}(N(S))|\leq \frac{|N(S)|k}{n}+\epsilon k.
    \end{equation*}
    We conclude that $\epsilon\geq \frac{1}{2n}$, a contradiction. Thus, we have $|S|\leq |N(S)|$ for all $S\subseteq[n]$, so Hall's condition is satisfied and the above bipartite graph has a perfect matching, as desired. This proves the first claim.

    For the second claim, suppose $\msf B_k(x)\overset{d}{=}\msf B_k(y)$ for all $k$, so there is a coupling $(F_{k,n},G_{k,N})$ of uniformly random maps $[n]\to[k]$ and $[N]\to[k]$, respectively, satisfying $\beta(F_{k,n})x=\beta(G_{k,N})y$. Note that for $k\geq n$, the random map $F_{k,n}$ is injective with probability at least $1-\frac{n(n-1)}{2k}$ and similarly for $G_{k,N}$. Therefore, for $k\geq N$ sufficiently large both $F_{k,n}$ and $G_{k,N}$ are injective with positive probability, in which case there are permutations $\pi,\tau\in\mfk S_k$ satisfying $F_{k,n}=\pi\circ\iota_{k,n}$ and $G_{k,N}=\tau\circ\iota_{k,N}$. In this case we get \begin{equation}\label{eq:injection_equivalence_step}
        \beta(\phi_{k,n})x=\beta(\iota_{k,N})y,\quad \textrm{for some injection } \phi_{k,n}\colon[n]\to[k].
    \end{equation}
    Define the map $f\colon[k]\to[N]$ sending $f(i)=i$ for $i\in[N]$, sending $\phi_{k,n}([n])\setminus[N]$ injectively to $[N]\setminus \phi_{k,n}([n])$, and sending $f(i)=1$ for $i\notin [N]\cup \phi_{k,n}([n])$. Note that $f\circ\iota_{k,N}=\mathrm{id}_{[N]}$ and $f\circ\phi_{k,n}$ is injective, so applying $\beta(f)$ to both sides in~\eqref{eq:injection_equivalence_step} yields the second claim.
\end{proof}

\section{Rates for Sampling with Replacement}\label{sec:sampling_wrep}

In this section, we prove rates for any-dimensional sketching and generalization using the sampling metric defined with respect to sampling with replacement~\eqref{eq:sampling_general}.  
The general outline for this section is as follows.  We first derive rates at which arbitrary objects can be approximated by fixed-dimensional ones in Theorem~\ref{thm:W1_bound_sampling}, quantifying part 2 of Theorem~\ref{thm:general_equivalence}.  
Based on Proposition~\ref{prop:fd_measures_as_limits}, we then characterize in Theorem~\ref{thm:proj_consistent_laws} the structure of any-dimensional data distributions obtained by sampling with replacement from limit objects.  
Next we leverage the equivalences in Theorem~\ref{thm:general_equivalence} to derive explicit sketching and generalization rates in Theorem~\ref{thm:sampling_rates_funcs} for functions continuous with respect to sampling with replacement.  
We conclude by deriving corollaries of our general results for functions arising in several applications, including graph densities and signals as well as many permutation-invariant neural network architectures such as transformers.

\subsection{Main Results}
We fix a compatible sequence of index sets $(\mc I_n)$, and consider the associated vector spaces $(\vct V_n=\RR^{\mc I_n})$ and sampling with replacement maps $(\msf R_k\colon\bigsqcup_n\vct V_n\to\vct V_k)$ defined by~\eqref{eq:sampling_general}. 
We also fix a sequence of compact sets $(\Omega_n)$ closed under sampling, such as the sequence of hypercubes $\Omega_n=[-r,r]^{\mc I_n}$. Throughout this section, we consider the sampling metric~\eqref{eq:sampling_metric} defined using $(\msf R_k)$, with $W_1$-distances taken with respect to the $\ell_{\infty}$-norms.

Our first result shows that $\mbb Ed_{\mathrm{samp}}(x,\msf R_n(x))\to0$ as $n\to\infty$ and gives the associated rate at which this expectation goes to zero.
To bound $\mbb Ed_{\mathrm{samp}}(x,\msf R_n(x))$, we must bound for each $k$ the expected distance $\mbb E_{\msf R_n}W_1(\msf R_k(x),\msf R_k\circ\msf R_n(x))$, where we (a) fix a realization of $\msf R_n$; (b) compute the $W_1$ distance between the distributions of $\msf R_k(x)$ and $\msf R_k\circ\msf R_n(x)$, yielding a random number depending on $\msf R_n$; and (c) take the expectation with respect to $\msf R_n$.
\begin{theorem}\label{thm:W1_bound_sampling}
    Suppose a compatible sequence $(\mc I_n)$ has degree $D$, and suppose $(\Omega_n\subseteq[-r,r]^{\mc I_n})$ is closed under sampling. 
    Then for any $x\in\bigsqcup_n\Omega_n$ and any $k$ such that $|\mc I_k|>2$, we have
    \begin{equation*}
        \mbb E_{\msf R_n}W_1(\msf R_k(x), \msf R_k\circ\msf R_n(x))\leq \frac{k(k-1)}{n}r + 228r\left(\frac{k^2}{n}\right)^{1/|\mc I_k|}
    \end{equation*}
    and $W_1(\msf R_k(x), \msf R_k\circ\msf R_n(x))$ is $\frac{4k^2r^2}{n}$-subgaussian with respect to the randomness in $\msf R_n$.
    We also have
    \begin{equation*}
        \mbb Ed_{\mathrm{samp}}(x,\msf R_n(x))\leq Cr\exp_2\Big[-(M^{-1}\log_2 n)^{\frac{1}{1+D}}\Big],
    \end{equation*}
    where $C>0$ is a constant depending on $(\mc I_n)$ and $M$ is as in Definition~\ref{def:FS_index_degree}. Furthermore, the distance $d_{\mathrm{samp}}(x,\msf R_n(x))$ is $\frac{16r^2}{n}$-subgaussian. 
\end{theorem}
Above, we write $\exp_2(a)=2^a$ and $\log_2(b)=\frac{\log b}{\log 2}$.
The proof is based on the observation that a uniformly random map $F_{N,k}\colon[k]\to[N]$ can be identified with $k$ iid uniformly distributed random indices $(F_{N,k}(1),\ldots,F_{N,k}(k))$ in $[N]$, so $\mathrm{Law}(F_{N,k})=\mu^{\otimes k}$ where $\mu$ is the uniform distribution on $[N]$. Meanwhile, if we fix a realization $\widehat F_{N,n}$ of $F_{N,n}$, we have $\mathrm{Law}(\widehat F_{N,n}\circ F_{n,k})=\mu_n^{\otimes k}$ where $\mu_n=\frac{1}{n}\sum_{i=1}^n\delta_{\widehat F_{N,n}(i)}$ is an empirical measure obtained by sampling $n$ iid points from $\mu$. Comparing $\mu_n$ to $\mu$ yields the proof, see Section~\ref{sec:proofs_sampling}.

We give a few examples of the rates implied by Theorem~\ref{thm:W1_bound_sampling}.
\begin{example}[Sampling columns with replacement]
    Suppose $\vct V_n=\RR^{d\times n}$ and $\Omega_n=[-r,r]^{d\times n}$. As discussed in Example~\ref{ex:sampling_entries}, we have an isomorphism $\overline{\Omega}_{\infty}\cong \mc P([-r,r]^d)$ induced by $x\mapsto \mathrm{Law}(\msf R_1(x))$. In particular, in this special case we have $\mathrm{Law}(\msf R_k(x))=\mathrm{Law}(\msf R_1(x))^{\otimes k}$, so $d_{\mathrm{samp}}(x,y)$ and $W_1(\msf R_1(x),\msf R_1(y))$ metrize the same topology. The rate $\mbb E_{\msf R_n}W_1(\msf R_1(x),\msf R_1\circ\msf R_n(x))=O(n^{-1/d})$ for $d>2$ from Theorem~\ref{thm:W1_bound_sampling} is optimal~\cite{fournier2023convergence}. 
\end{example}

\begin{example}[Sampling vertices with replacement]
    Suppose $\vct V_n=\RR^{n\times n}$ and let $\Omega_n\subseteq[-r,r]^{n\times n}$ be closed under sampling. Convergence in $d_{\mathrm{samp}}$ recovers the limits of weighted graphs studied in~\cite{probability_graphons}. We obtain a rate of convergence of $O(\exp_2[-(\log_2 n)^{1/3}])$ in $d_{\mathrm{samp}}$, compared with a rate of $O(\exp_2[-\frac{1}{2}\log\log n])$ implied by the results of~\cite{probability_graphons} (who focus on convergence in a different metric, but the two can be related as in~\cite[Exer.~10.33]{lovasz2012large}). 
    We get the similar rate $O(\exp_2[-(\frac{1}{d+1}\log_2 n)^{1/3}])$ for pairs of graphs and $d$-dimensional graph signals on them, for which $\vct V_n=\RR^{n\times n}\oplus\RR^{n\times d}$.
    Likewise, we obtain a rate of $O\left(\exp_2\left[-(M^{-1}\log n)^{\frac{1}{1+D}}\right]\right)$ for approximation of general rank-$D$ hypergraphs (with $M=1$) and hypergraph $d$-dimensional signals ($M=d+1$).
\end{example}

Theorem~\ref{thm:W1_bound_sampling} has two important consequences. First, it implies that $\mbb Ed_{\mathrm{samp}}(x,\msf R_n(x))\to0$ for any $x\in\overline{\Omega}_{\infty}$, so Proposition~\ref{prop:fd_measures_as_limits} applies and gives us a characterization of any-dimensional data distributions obtained by sampling from random limit objects.
\begin{theorem}[Projection-consistent distributions]\label{thm:proj_consistent_laws}
    Suppose $(\mc I_n)$ has finite degree and let $(\Omega_n\subseteq[-r,r]^{\mc I_n})$ be closed under sampling. Then the following are equivalent for a sequence of probability distributions $(\mu_n\in\mc P(\Omega_n))$: 
    \begin{enumerate}[align=left, font=\emph, labelwidth=!, labelindent=0pt]
        \item[(Projection consistency)] We have $\rho(\phi_{N,n})\mu_N=\mu_n$ whenever $ \phi_{N,n}\colon[n]\to[N]$ is injective.
        \item[(Sampling representation)] There exists a distribution $\mu_{\infty}\in\mc P(\overline{\Omega}_{\infty})$ such that $\mu_n=\mathrm{Law}(\msf R_n(X))$ for all $n$, where $X\sim\mu_{\infty}$ is independent of $\msf R_n$.
    \end{enumerate}
    Moreover, the measure $\mu_{\infty}$ above is unique, and sequences of the form $(\mu_n=\mathrm{Law}(\msf R_n(x)))$ for deterministic $x\in\overline{\Omega}_{\infty}$ are extremal in the set of all projection-consistent sequences.
\end{theorem}
The proof combines Proposition~\ref{prop:fd_measures_as_limits} and Theorem~\ref{thm:W1_bound_sampling}, along with a ``finite'' de Finetti theorem comparing sampling with and without replacement from~\cite{levin2025deFin} generalizing~\cite{diaconis_freedman}, see Section~\ref{sec:proofs_sampling}.

Theorem~\ref{thm:proj_consistent_laws} may be viewed as generalizing one formulation of de Finetti's theorem and several related results concerning infinite exchangeable arrays, as we proceed to illustrate.
\begin{example}[Infinite exchangeable arrays]\label{ex:deFin}
    Suppose $\vct V_n=\RR^{d\times n}$ and $\Omega_n=\Theta^n$ for compact $\Theta\subseteq\RR^d$. In this case, projection-consistent distributions are precisely infinite exchangeable arrays. Indeed, if $(X_1,X_2,\ldots)$ is an infinite exchangeable array with $X_i\in\Theta$ almost surely then $\mu_n=\mathrm{Law}(X_1,\ldots,X_n)$ is projection-consistent, and conversely if a sequence of measures $(\mu_n\in\mc P(\Omega_n))$ is projection-consistent then the sequence $(X_1,\ldots,X_n)\sim\mu_n$ extends to an infinite exchangeable array by Kolmogorov's extension theorem. Since $\overline{\Omega}_{\infty}\cong\mc P(\Theta)$ in this case, Theorem~\ref{thm:proj_consistent_laws} recovers the correspondence between infinite exchangeable arrays and mixtures of iid arrays, along with the extremality of iid arrays, that follows from de Finetti's theorem and its extensions by Hewitt--Savage~\cite{hewitt_savage} and Dynkin~\cite{dynkin1953classes}.
\end{example}

\begin{example}[Infinite exchangeable random graph models]\label{ex:AldousHoover}
    Suppose $\vct V_n=\RR^{n\times n}$ and $\Omega_n\subseteq\{0,1\}^{n\times n}$. Reasoning as in Example~\ref{ex:deFin}, we note that projection-consistent sequences $(\mu_n\in\mc P(\Omega_n))$ correspond to infinite two-dimensional arrays of binary random variables $(X_{i,j})_{i,j\in\NN}$ whose distribution is unchanged by simultaneous permutations of rows and columns. Such arrays can be viewed, in turn, as unweighted random graph models on countably-many exchangeable vertices, see~\cite{diaconis2007graph}. In this setting, Theorem~\ref{thm:proj_consistent_laws} implies that the countable random graph models obtained by sampling vertices with replacement from deterministic graphs and their limits are extremal, recovering~\cite[Cor.~5.4]{diaconis2007graph}. 
    Our Theorem~\ref{thm:proj_consistent_laws} similarly applies to general weighted countable graph models, as well as hypergraph models. 
\end{example}

The second consequence of Theorem~\ref{thm:W1_bound_sampling} is sketching and generalization bounds for functions continuous in sampling metric. While the rates we obtain for general Lipschitz-continuous functions are quite slow, we obtain substantially faster rates by exploiting particular representations of the functions at hand in terms of low-dimensional samples of their inputs.

\begin{theorem}[Sketching and generalization rates]\label{thm:sampling_rates_funcs}
    In the setting of Theorem~\ref{thm:W1_bound_sampling}, consider functions $f,\widehat f\colon\bigsqcup_n\Omega_n\to\RR$.
    \begin{enumerate}[font=\emph, align=left, labelwidth=!, labelindent=0pt]
        \item[(General)] If $f$ is $L$-Lipschitz in $d_{\mathrm{samp}}$, then for any $x\in\bigsqcup_n\Omega_n$ with probability at least $1-2e^{-2\epsilon^2}$ we have
        \begin{equation*}
            |f(x)-f(\msf R_n(x))|\leq LCr\exp_2\Big[-(M^{-1}\log_2 n)^{\frac{1}{1+D}}\Big] + \frac{4rL\epsilon}{\sqrt{n}}.
        \end{equation*}
        If $\widehat f$ is also $L$-Lipschitz in $d_{\mathrm{samp}}$, then 
        \begin{equation}\label{eq:generalization_sampling_lip} 
        \mathrm{e}_{\infty}(f,\widehat f)\leq\mathrm{e}_n(f,\widehat f)+2LCr\exp_2\left(-(M^{-1}\log_2 n)^{\frac{1}{1+D}}\right).
        \end{equation}

        \item[(Fixed-dimensional laws)] Suppose $f(x)=g(\mathrm{Law}(\msf R_k(x)))$ where $g\colon\mc P(\Omega_k)\to\RR$ is $L$-Lipschitz in Wasserstein distance and $|\mc I_k|>2$. Then for any $x\in\bigsqcup_n\Omega_n$ with probability at least $1-2e^{-2\epsilon^2}$ we have
        \begin{equation*}
            |f(x)-f(\msf R_n(x))| \leq Lr\left[ \frac{k(k-1)}{n} + \frac{2k\epsilon}{\sqrt{n}} + 228\left(\frac{k^2}{n}\right)^{1/|\mc I_k|}\right].
        \end{equation*}
        If $\widehat f$ has the same form as $f$ above, then
        \begin{equation}\label{eq:generalization_sampling_Law}
            \mathrm{e}_{\infty}(f,\widehat f)\leq \mathrm{e}_n(f,\widehat f) + 2Lr\left[\frac{k(k-1)}{n} + 228\left(\frac{k^2}{n}\right)^{1/|\mc I_k|}\right].
        \end{equation}

        \item[(Fixed-dimensional moments)] Suppose $f=\sigma(\mbb Eg_k\circ\msf R_k)$ where each coordinate function of $g_k\colon\Omega_k\to\RR^\ell$ is $B$-bounded and $\sigma\colon\RR^\ell\to\RR$ is $L_{\sigma}$-Lipschitz with respect to the $\ell_1$ norm on $\RR^\ell$. Then for any $x\in\bigsqcup_n\Omega_n$ with probability at least $1-2\ell e^{-2\epsilon^2}$ we have
        \begin{equation*}
            |f(x)-f(\msf R_n(x))|\leq L_{\sigma}Bk\ell\left(\frac{k-1}{n}+\frac{2\epsilon}{\sqrt{n}}\right).
        \end{equation*}
        If $\widehat f$ has the same form (possibly with a different $L_{\sigma}$-Lipschitz $\widehat \sigma$ and $B$-bounded $\widehat g_k$), then
        \begin{equation}\label{eq:generalization_sampling_means}
            \mathrm{e}_{\infty}(f,\widehat f)\leq \mathrm{e}_n(f,\widehat f) + 2L_{\sigma}Bk\ell\left(\frac{k-1}{n} + \frac{2}{\sqrt{n}}\right).
        \end{equation}
        When $\ell=1$ and $\sigma=\mathrm{id}$, the $2/\sqrt{n}$ term above can be omitted.
    \end{enumerate}
    The generalization bounds~\eqref{eq:generalization_sampling_lip},~\eqref{eq:generalization_sampling_Law}, and~\eqref{eq:generalization_sampling_means} also hold for average error with respect to projection-consistent distributions.  
\end{theorem}
The proof uses Theorem~\ref{thm:W1_bound_sampling}, and we again defer it to Section~\ref{sec:proofs_sampling}. 
%
We now illustrate these improved rates on several classes of functions, beginning with polynomials.
\begin{corollary}[Polynomials]\label{cor:moment_polys}
    In the setting of Theorem~\ref{thm:W1_bound_sampling}, let $p\colon\bigsqcup_n\Omega_n\to\RR$ be a polynomial of degree $d$ unchanged by equipartitions in the sense of~\eqref{eq:eqp_invariant}. Then there exists $B>0$ such that for any $x\in\bigsqcup_n\Omega_n$ with probability at least $1-2e^{-2\epsilon^2}$, we have
    \begin{equation}\label{eq:sampling_poly_sketch}
        |p(x)-p(\msf R_n(x))|\leq BDd\left[\frac{Dd-1}{n} + \frac{2\epsilon}{\sqrt{n}}\right].
    \end{equation}
    If $\widehat p$ is another such polynomial, then there exists $\widehat B>0$ such that
    \begin{equation*}
        \mathrm{e}_{\infty}(p,\widehat p)\leq \mathrm{e}_n(p,\widehat p) + \frac{(B+\widehat B)Dd(Dd-1)}{n}.
    \end{equation*}
\end{corollary}

\begin{proof}
    It was shown in~\cite[Sec.~5]{levin2025deFin} that there is a polynomial $g_{Dd}\in\RR[\vct V_{Dd}]$ satisfying $p(x)=\mbb Eg_{Dd}(\msf R_{Dd}x)$ for all $x$. Setting $B=\sup_{x\in\Omega_{Dd}}|g_{Dd}(x)|$ and similarly for $\widehat B$, we obtain both conclusions from Theorem~\ref{thm:sampling_rates_funcs}.
\end{proof}
Notably, Corollary~\ref{cor:moment_polys} applies to any polynomial function and any finite-degree compatible sequence of index sets. Here are a few examples encompassed by the above result.
\begin{enumerate}[align=left, font=\emph,labelwidth=!, labelindent=0pt,wide]
    \item[(Moment polynomials)] When $\vct V_n=\RR^{d\times n}$, polynomials unchanged by duplication are precisely moment polynomials 
    $$p(x)=q\Big(\mbb E(\msf R_1(x))^{\alpha_1},\ldots,\mbb E(\msf R_1(x))^{\alpha_\ell}\Big),$$ 
    where $\alpha_i\in \NN_0^d$ and $q$ is a fixed polynomial~\cite[\S6.1]{levin2025deFin}. Here $\mbb E(\msf R_1(x))^{\alpha}=\frac{1}{n}\sum_{j=1}^n\prod_{\ell=1}^d(x_j)_\ell^{\alpha_{\ell}}$ for $\alpha\in\NN_0^d$.
    
    \item[(Graph densities)] When $\vct V_n=\RR^{n\times n}$, polynomials unchanged by duplication are precisely linear combinations of graph homomorphism densities~\cite[\S6.3]{levin2025deFin}. Specifically, if $H\in\NN_0^{k\times k}$ is a multigraph, we define the homomorphism density of $H$ in $G\in\vct V_n$ by
    \begin{equation*}
        t(H;G) = \frac{1}{n^k}\sum_{f\colon[k]\to[n]}\prod_{i,j=1}^kG_{f(i),f(j)}^{H_{i,j}}.
    \end{equation*}
    It is called a homomorphism density because when $H$ and $G$ are undirected simple graphs, the value $t(H;G)$ is the fraction of maps between vertex sets $V(H)\to V(G)$ that are graph homomorphisms.
    Similarly, when $\vct V_n=(\RR^n)^{\otimes d}$ is viewed as the space of adjacency tensors of hypergraphs, polynomials unchanged by duplication correspond to linear combinations of hypergraph densities.

    \item[(Graph signals)] When $\vct V_n=\RR^{n\times n}\oplus \RR^{n\times d}$ consists of pairs of graphs and graph signals on them, then polynomials unchanged by duplication include polynomial graphon neural networks~\cite{transferab1}, which are compositions of polynomials of the form
    \begin{equation*}
        (G,X)\mapsto \left(G,\sigma\left(\sum_{\ell=0}^L\frac{1}{n^{\ell}}G^{\ell}X\Theta_{\ell}\right)\right),\quad \textrm{for polynomial } \sigma \textrm{ and } \Theta_0,\ldots,\Theta_L\in\RR^{d\times d},
    \end{equation*}
    followed by a last pooling layer such as $(G,X)\mapsto \frac{1}{n}X^\top \mathbbm{1}$ that is polynomial and unchanged under duplication.
\end{enumerate}
For all of the above families of polynomials, we obtain any-dimensional sketching rates of $O(n^{-1/2})$ and any-dimensional generalization rates of $O(n^{-1})$. 

We remark that the rates for these polynomials are a substantial improvement over those previously available. For example, the result~\cite[Prop.~4.2]{levin2025transferring} implies $O(n^{-1/d})$ generalization rates for moment polynomials and $O(1/\sqrt{\log n})$ rates for graph densities and signals with respect to projection-consistent any-dimensional distributions. The improvement comes precisely from exploiting the fact that polynomials compute moments of low-dimensional samples, instead of merely their Lipschitz continuity.

Next, we turn to nonpolynomial function classes.
\begin{corollary}[Normalized DeepSets]\label{cor:normalized_deepsets}
    Suppose $\mc I_n=[d]\times [n]$ so $\vct V_n=\RR^{d\times n}$, and let $(\Omega_n\subseteq [-r,r]^{d\times n})$ be a sequence of compact sets closed under sampling. Suppose each coordinate function of $h\colon\RR^d\to\RR^\ell$ is $B$-bounded on $\Omega_1$ and $\sigma\colon\RR^\ell\to\RR$ is $L_{\sigma}$-Lipschitz in $\ell_1$-norm. Consider $f\colon\bigsqcup_n\Omega_n\to\RR$ defined by \begin{equation}\label{eq:normalized_deepsets}
        f(x)=\sigma\left(\frac{1}{n}\sum_{i=1}^nh(x_i)\right)=\sigma(\mbb Eh\circ\msf R_1(x)),\quad \textrm{for } x=(x_1,\ldots,x_n)\in\vct V_n.
    \end{equation}
    Then for any $x\in\bigsqcup_n\Omega_n$ with probability at least $1-2\ell e^{-2\epsilon^2}$ we have
    \begin{equation*}
        |f(x)-f(\msf R_n(x))|\leq \frac{2L_{\sigma}B\ell\epsilon}{\sqrt{n}},
    \end{equation*}
    and for any other $\widehat f$ of the above form, we have
    \begin{equation*}
        \mathrm{e}_{\infty}(f,\widehat f)\leq \mathrm{e}_n(f,\widehat f) + \frac{4L_{\sigma}B\ell}{\sqrt{n}}.
    \end{equation*}
\end{corollary}
\begin{proof}
    This directly follows from Theorem~\ref{thm:sampling_rates_funcs}.
\end{proof}
When $\sigma$ and $h$ are neural networks, the architecture defined by~\eqref{eq:normalized_deepsets} is called normalized DeepSets~\cite{bueno2021representation,levin2025transferring}. In particular, when $\Omega_n=\Theta^n$ for compact $\Theta\subseteq\RR^d$, we can extend $f$ in~\eqref{eq:normalized_deepsets} to measures $\mu\in\mc P(\Theta)$ by $f(\mu)=\sigma(\mbb E_{\mu}h)$. This extension was shown in~\cite{bueno2021representation} to be Lipschitz-continuous in Wasserstein metric on $\mc P(\Theta)$. In turn, the result~\cite[Prop.~4.2]{levin2025transferring} implies $O(n^{-1/d})$ generalization rates for such Lipschitz functions on measures for average error with respect to infinite exchangeable arrays (Example~\ref{ex:deFin}). This rate can similarly be derived from Theorem~\ref{thm:sampling_rates_funcs}, and Corollary~\ref{cor:normalized_deepsets} yields improved $O(n^{-1/2})$ rates for functions of the form~\eqref{eq:normalized_deepsets}.

Finally, we give sketching and generalization rates for transformers, making Corollary~\ref{cor:transformers_intro} precise. 
\begin{corollary}[Permutation-invariant transformers]\label{cor:transformers}
    Suppose $\mc I_n=[d]\times[n]$ with $d>2$, so $\vct V_n=\RR^{d\times n}$, and $\Omega_n=\Theta^n$ for compact $\Theta\subseteq[-r,r]^d$. For matrices $Q,K\in\RR^{h\times d}$ and $V\in\RR^{d\times d}$, define the self-attention mapping $A\colon\vct V_n\to\vct V_n$ by
    \begin{equation}\label{eq:attention}
        A(x_1,\ldots,x_n)_i = x_i + \frac{1}{Z_i}\sum_{j=1}^n\exp(\langle Qx_i, Kx_j\rangle)Vx_j,\quad \textrm{where } Z_i=\sum_{j=1}^n\exp(\langle Qx_i, Kx_j\rangle).
    \end{equation}
    Let $\phi\colon\RR^d\to\RR^d$ be $L_{\phi}$-Lipschitz, and let it act column-wise on $\vct V_n$. Finally, let $P\colon\vct V_n\to\vct V_1$ be the mean pooling map $P(x_1,\ldots,x_n)=\frac{1}{n}\sum_{i=1}^nx_i$. Define a depth-$\ell$ transformer to be the mapping $T\colon\bigsqcup_n\vct V_n\to\vct V_1$ defined by the composition
    \begin{equation*}
        T = P\circ(\phi\circ A)^{\circ \ell}.
    \end{equation*}
    Then for any $x\in\bigsqcup_n\Omega_n$ and $\epsilon>0$ with probability at least $1-2(\ell+1) e^{-2\epsilon^2}$ we have
    \begin{equation}\label{eq:transformer_sampling}
        \|T(x)-T(\msf R_n(x))\|_2\leq \frac{C(1+2\epsilon)}{\sqrt{n}},
    \end{equation}
    where $C>0$ depends on the architecture and on $r$.
    For any other such map $\widehat T$, we then have
    \begin{equation*}
        \mathrm{e}_{\infty}(T,\widehat T)\leq \mathrm{e}_n(T,\widehat T) + \frac{C'}{\sqrt{n}},
    \end{equation*}
    where $\mathrm{e}_n(T,\widehat T)=\sup_{x\in\Omega_{\leq n}}\|T(x)-\widehat T(x)\|_2$ and similarly for $\mathrm{e}_{\infty}$. Here $C'>0$ depends on both architectures and on $r$. 
    Also, for any $G\colon\bigsqcup_n\Omega_n\to\vct V_1$ of the form $G(x)=\bar G(\mathrm{Law}(\msf R_1(x)))$ for Lipschitz $\bar G\colon\mc P(\Theta)\to\vct V_1$, we have
    \begin{equation*}
        \mathrm{e}_{\infty}(T,G)\leq \mathrm{e}_n(T,G) + \frac{C''}{n^{1/d}},
    \end{equation*}
    where $C''>0$ depends on the architecture of $T$, on the Lipschitz constant of $\bar G$, and on $\Theta$.
\end{corollary}

The proof is based on the measure-theoretic in-context mapping of~\cite{furuya2025transformers}, and is deferred to Section~\ref{sec:proofs_sampling}. 
The $O(1/\sqrt{n})$ sketching rate for $\ell=1$ was previously proved in~\cite[Prop.~D.13]{bohbot2025token}.
We remark that evaluating $T(x)$ exactly for $x\in\vct V_N$ takes $O(N^2)$ operations due to~\eqref{eq:attention}, but evaluating it to a desired accuracy $\delta>0$ can be done with high probability in $O(\delta^{-4})$ time (independent of the length $N$ of the input) after sampling $n=O(\delta^{-2})$ columns from $x$ uniformly at random and using~\eqref{eq:transformer_sampling}. 

\subsection{Missing Proofs from Section~\ref{sec:sampling_wrep}}\label{sec:proofs_sampling}
We begin by proving Theorem~\ref{thm:W1_bound_sampling}. 
The following is the key proposition we shall need.
\begin{proposition}\label{prop:concentration_coFS}
    Fix a bounded $f\colon\Omega_k\to\RR$ and $x\in\bigsqcup_n\Omega_n$, and define the random variable $Z_f=\mbb E_{\msf R_k}f(\msf R_k\circ\msf R_n(x)) - \mbb Ef(\msf R_k(x))$ where expectations are only with respect to $\msf R_k$, which is independent of $\msf R_n$. 
    \begin{enumerate}[align=left, font=\emph, labelwidth=!, labelindent=0pt]
        \item[(Bias)] We have $|\mbb E_{\msf R_n}Z_f|\leq \frac{k(k-1)}{n}\|f\|_{\infty}$. 

        \item[(Concentration)] $Z_f-\mbb E_{\msf R_n}Z_f$ is $\frac{4k^2}{n}\|f\|_{\infty}^2$-subgaussian.
    \end{enumerate}
\end{proposition}
\begin{proof}
    If $k>n$ then $\frac{k(k-1)}{n}\geq 2$ and $|\mbb E_{\msf R_n}Z_f|\leq 2\|f\|_{\infty}\leq\frac{k(k-1)}{n}\|f\|_{\infty}$. Suppose therefore that $k\leq n$. By~\cite{sampling_tv_dist}, there is a coupling of a uniformly random map $F_{n,k}$ and of a uniformly random injection $\Phi_{n,k}$ such that $\mbb P[F_{n,k}\neq \Phi_{n,k}]\leq \frac{k(k-1)}{2n}$. Furthermore, since $F_{N,n}\circ\phi_{n,k}\overset{d}{=}F_{N,k}$ for any injection $\phi_{n,k}\colon[k]\to[n]$, we have $\rho(\Phi_{n,k})\circ \msf R_n\overset{d}{=}\msf R_k$ if $\msf R_n$ is independent of $\Phi_{n,k}$. In this case, we have
    \begin{equation*}
        |\mbb EZ_f| = |\mbb Ef(\rho(F_{n,k})\circ\msf R_n(x)) - \mbb Ef(\rho(\Phi_{n,k})\circ\msf R_n(x))| \leq 2\|f\|_{\infty}\mbb P\Big[F_{n,k}\neq \Phi_{n,k}\Big]\leq \|f\|_{\infty}\frac{k(k-1)}{n},
    \end{equation*}
    as claimed. 

    To see the subgaussianity of $Z_f$, represent a map $[n]\to[N]$ as a sequence in $[N]^n$, and observe that $\mathrm{Law}(F_{N,n})=\mathrm{Unif}[N]^{\otimes n}$. Writing $\msf R_n=\rho(\widehat{F}_{N,n})$, let $\mu_n=\frac{1}{n}\sum_{i=1}^n\delta_{\widehat{F}_{N,n}(i)}$ be the empirical measure obtained by sampling $n$ iid points from $\mu=\mathrm{Unif}([N])$. Now observe that $\mathrm{Law}(\widehat{F}_{N,n}\circ F_{n,k})=\mu_n^{\otimes k}$ conditioned on $\widehat{F}_{N,n}$ while $\mathrm{Law}(F_{N,k})=\mu^{\otimes k}$. Furthermore, if we define $\widetilde f\colon [N]^k\to \RR$ by $\widetilde f(F)=f(\rho(F)x)$, then 
    \begin{equation}\label{eq:V_stat_rep}
    \mbb E_{\msf R_k}f(\msf R_k\circ\msf R_n(x))=\mbb E_{\mu_n^{\otimes k}}\widetilde f=\frac{1}{n^k}\sum_{i_1,\ldots,i_k=1}^n\widetilde f\Big(\left[\widehat F_{N,n}(i_1),\ldots,\widehat F_{N,n}(i_k)\right]\Big),
    \end{equation}
    where $\|\widetilde f\|_{\infty}\leq\|f\|_{\infty}$. Changing the value $\widehat F_{N,n}(\ell)$ for any fixed $\ell\in[n]$ changes the value of $\mbb E_{\mu_n^{\otimes k}}\widetilde f$ by at most $\frac{2k}{n}\|f\|_{\infty}$, hence $Z_f-\mbb EZ_f$ is $\frac{4k^2}{n}\|f\|_{\infty}^2$-subgaussian by the bounded-difference inequality~\cite{wainwright2019high}.
\end{proof}

We turn to proving Theorem~\ref{thm:W1_bound_sampling} by combining Proposition~\ref{prop:concentration_coFS} with the following standard covering number bound. 
\begin{lemma}[Covering number bound]\label{lem:covering_nums}
    In the setting of Proposition~\ref{prop:concentration_coFS}, suppose $\mc F_k$ is a family of functions on $\Omega_k$ with $\sup_{f\in\mc F_k}\sup_{x\in\Omega_k}|f(x)|\leq r$. Then
    \begin{equation*}
        \mbb E\sup_{f\in\mc F_k}Z_f \leq r\frac{k(k-1)}{n} + \inf_{\delta>0}\left\{4\delta + \frac{16\sqrt{2}k}{\sqrt{n}}\int_{\delta}^{r}\sqrt{\log N(\mc F_k,\|\cdot\|_{\infty},\epsilon)}\, d\epsilon\right\},
    \end{equation*}
    where $N(\mc F_k,\|\cdot\|_{\infty},\epsilon)$ is the $\epsilon$-covering number of the function class $\mc F_k$ in $\|\cdot\|_{\infty}$.
    Moreover, the random variable $\sup_{f\in\mc F_k}Z_f$ is $\frac{4k^2r^2}{n}$-subgaussian.
\end{lemma}
\begin{proof}
    Since $Z_f-Z_g=Z_{f-g}$, we conclude that $\{Z_f-\mbb EZ_f\}_{f\in\mc F_k}$ is a centered $\frac{4k^2}{n}$-subgaussian process with respect to the uniform metric. Since the diameter of $\mc F_k$ under the uniform metric is at most $2r$ by assumption, Dudley's entropy integral bound~\cite{dudley_bound} gives
    \begin{equation*}
        \mbb E\sup_{f\in\mc F_k}(Z_f-\mbb EZ_f)\leq \inf_{\delta>0}\left\{4\delta + \frac{16\sqrt{2}k}{\sqrt{n}}\int_{\delta}^{r}\sqrt{\log N(\mc F_k,\|\cdot\|_{\infty},\epsilon)}\, d\epsilon\right\}.
    \end{equation*}
    Combining the above bound with our bias bound from Proposition~\ref{prop:concentration_coFS} yields the first claim.

    Finally, observe that $Z_f$ is a function of $n$ iid random indices $\widehat F_{N,n}(1),\ldots,\widehat{F}_{N,n}(n)$, and that changing the value of one of these indices affects at most $n^k-(n-1)^k\leq kn^{k-1}$ terms in~\eqref{eq:V_stat_rep} by at most $2\|f\|_{\infty}/n^k\leq 2r/n^k$. Therefore, each $Z_f$ is a $\frac{2kr}{n}$-Lipschitz function with respect to the Hamming distance on $[N]^n$, and hence the same is true of $\sup_{f\in\mc F_k}Z_f$. The bounded difference inequality~\cite[Cor.~2.21]{wainwright2019high} yields the second claim.
\end{proof}

Combining Lemma~\ref{lem:covering_nums} with covering number bounds for Lipschitz functions yields Theorem~\ref{thm:W1_bound_sampling}.
\begin{proof}[Proof (Theorem~\ref{thm:W1_bound_sampling})]
    Let $\mc F_k$ be the collection of all functions $f\colon[-r,r]^{\mc I_k}\to\RR$ with $f(0)=0$ that are 1-Lipschitz with respect to the $\ell_{\infty}$-norm, and note that $\sup_{f\in\mc F_k}\|f\|_{\infty}\leq r$. Combining the definition of the $W_1$ distance with Lemma~\ref{lem:covering_nums} gives
    \begin{equation*}
        \mbb E_{\msf R_n}W_1(\msf R_k(x),\msf R_k\circ\msf R_n(x)) \leq \mbb E\sup_{f\in\mc F_k}Z_f\leq r\frac{k(k-1)}{n} + \inf_{\delta>0}\left\{4\delta + \frac{16\sqrt{2}k}{\sqrt{n}}\int_{\delta}^{r}\sqrt{\log N(\mc F_k,\|\cdot\|_{\infty},\epsilon)}\, d\epsilon\right\}.
    \end{equation*}
    By~\cite[\S9]{kolmogorov1959varepsilon} (or by~\cite[Thm.~17]{luxburg2004distance}), we have 
    \begin{equation*}
        \log N(\mc F_k,\|\cdot\|_{\infty},\epsilon) \leq (\log 2)N([-r,r]^{\mc I_k},\|\cdot\|_{\infty},\epsilon/2)+\log\left(2\left\lceil\frac{2r}{\epsilon}\right\rceil +1\right).
    \end{equation*}
    We further have $N([-r,r]^{\mc I_k},\|\cdot\|_{\infty},\epsilon/2)\leq \left\lceil 4r/\epsilon\right\rceil^{|\mc I_k|}$ and hence $\log N(\mc F_k,\|\cdot\|_{\infty},\epsilon) \leq (1+\log 2)\left(\frac{5r}{\epsilon}\right)^{|\mc I_k|}$ for all $\epsilon\in(0,r)$. Since $|\mc I_k|>2$, we have
    \begin{equation*}
        4\delta + \frac{16\sqrt{2}k}{\sqrt{n}}\int_{\delta}^{r}\sqrt{\log N(\mc F_k,\|\cdot\|_{\infty},\epsilon)}\, d\epsilon\leq 4\delta+16\sqrt{2(1+\log 2)}\frac{k}{\sqrt{n}}(5r)^{|\mc I_k|/2}\frac{\delta^{-|\mc I_k|/2+1}}{|\mc I_k|/2-1}.
    \end{equation*}
    Choosing $\delta=5r\left(4\frac{k}{\sqrt{n}}\sqrt{2(1+\log 2)}\right)^{2/|\mc I_k|}$, we obtain the claimed estimate $228r\left(\frac{k^2}{n}\right)^{1/|\mc I_k|}$. If this chosen $\delta$ exceeds $r$, then $\left(\frac{k^2}{n}\right)^{1/|\mc I_k|}\geq 1/4$, and choosing $\delta=r$ instead gives the bound $4r\leq 16r\left(\frac{k^2}{n}\right)^{1/|\mc I_k|}\leq 228r\left(\frac{k^2}{n}\right)^{1/|\mc I_k|}$. 
%
    Lemma~\ref{lem:covering_nums} also shows that $\sup_{f\in\mc F_k}Z_f$ is $\frac{4k^2r^2}{n}$-subgaussian.

    Turning to the claimed bound on $d_{\mathrm{samp}}$, for any $K\leq n$ we have
    \begin{equation*}
        \mbb Ed_{\mathrm{samp}}(x,\msf R_n(x))\leq \sum_{k=1}^K2^{-k}\mbb E_{\msf R_n}W_1(\msf R_k(x),\msf R_k\circ\msf R_n(x)) + 2^{-K+1}r\leq Cr\left(\frac{K^2}{n}\right)^{\frac{1}{MK^D}} + 2^{-K+1}r
    \end{equation*}
    obtained by estimating $\mbb E_{\msf R_n}W_1(\msf R_k(x),\msf R_k\circ\msf R_n(x))\leq 2r$ for $k\geq K$ and applying the upper bounds from the first part of this corollary (if $|\mc I_k|\leq 2$ for some small $k$, the covering number argument yields bounds of $O(n^{-1/2})$ or $O(n^{-1/2}\log n)$, both of which are dominated by the slower $|\mc I_k|>2$ terms). Choosing $K=\left\lceil\left(\frac{\log_2 n}{M}\right)^{\frac{1}{D+1}}\right\rceil$ gives the desired bound. Finally, the same proof as in Proposition~\ref{prop:concentration_coFS} shows that $d_{\mathrm{samp}}(x,\msf R_n(x))$ is a function of $n$ independent random variables and is $\sum_{k\geq 1}2^{-k}\frac{2kr}{n}=\frac{4r}{n}$-Lipschitz in Hamming distance in each of them, hence the bounded-difference inequality yields its claimed $\frac{16r^2}{n}$-subgaussianity.
\end{proof}

We turn to proving Theorem~\ref{thm:proj_consistent_laws} by combining Proposition~\ref{prop:fd_measures_as_limits} and a ``finite'' de Finetti theorem from~\cite{levin2025deFin}. 
\begin{proof}[Proof (Theorem~\ref{thm:proj_consistent_laws}).]
    Propositions~\ref{prop:symmetry_coFS} and~\ref{prop:relations_coFS} show that if $\mu_n=\mathrm{Law}(\msf R_n(X))$ for random $X$ independent of $\msf R_n$, then $(\mu_n)$ is projection-consistent. Conversely, if $(\mu_n)$ is projection-consistent then by~\cite[Thm.~4.15]{levin2025deFin} there is a sequence $\nu_i\in\mc P(\Omega_{\infty})$ such that $\mathrm{Law}(\msf R_n(X_i))\xrightarrow{i\to\infty}\mu_n$ weakly for each $n$ where $X_i\sim\nu_i$ is independent of $\msf R_n$. This implies that the sequence $(\nu_i)$ converges weakly with respect to $d_{\mathrm{samp}}$, and $\mu_{\infty}=\lim_i\nu_i\in\mc P(\overline{\Omega}_{\infty})$ then satisfies the desired sampling representation. The last claim follows by Proposition~\ref{prop:fd_measures_as_limits}, which applies by Theorem~\ref{thm:W1_bound_sampling}. Specifically, Theorem~\ref{thm:W1_bound_sampling} shows that $\mbb Ed_{\mathrm{samp}}(x,\msf R_n(x))\to0$ as $n\to\infty$ at a universal rate for all $x\in\Omega_{\infty}$. If $x\in\overline{\Omega}_{\infty}$ then we can write $x=\lim_ix_i$ for $x_i\in\Omega_{\infty}$, so $\msf R_n(x)$ is the weak limit of $\msf R_n(x_i)$. A simple application of the triangle inequality shows that $\mbb Ed_{\mathrm{samp}}(x,\msf R_n(x))\to0$ at the same universal rate, so the condition~\eqref{eq:convergence_of_samples} in Proposition~\ref{prop:fd_measures_as_limits} is satisfied.
\end{proof}

Next, we prove the sketching and generalization bounds in Theorem~\ref{thm:sampling_rates_funcs}.
\begin{proof}[Proof (Theorem~\ref{thm:sampling_rates_funcs}).]
    If $f$ and $\widehat f$ are $L$-Lipschitz in $d_{\mathrm{samp}}$ and $x\in\bigsqcup_n\Omega_n$, then $|f(x)-f(\msf R_n(x))|\leq Ld_{\mathrm{samp}}(x,\msf R_n x)$ and
    \begin{equation}\label{eq:main_generalization_bound}
        |f(x)-\widehat f(x)|\leq |f(x)-\mbb Ef(\msf R_n(x))| + |\mbb E[f(\msf R_n(x))-\widehat f(\msf R_n(x))]| + |\mbb E\widehat f(\msf R_n(x))-\widehat f(x)|.
    \end{equation}
    Now observe that $|f(x)-\mbb Ef(\msf R_n(x))|\leq L\mbb Ed_{\mathrm{samp}}(x,\msf R_n(x))$ and similarly for $\widehat f$, and that $|\mbb E[f(\msf R_n(x))-\widehat f(\msf R_n(x))]|\leq \mathrm{e}_n(f,\widehat f)$. The bounds for general Lipschitz functions now follow from Theorem~\ref{thm:W1_bound_sampling}. 

    If $f(x)=F(\mathrm{Law}(\msf R_k(x)))$ for $L$-Lipschitz $F$, we have $|f(x)-f(\msf R_n(x))|\leq LW_1(\msf R_k(x),\msf R_k\circ\msf R_n(x))$ where we fix a realization of $\msf R_n$, and $|f(x)-\mbb Ef(\msf R_n(x))|\leq L\mbb EW_1(\msf R_k(x),\msf R_k\circ\msf R_n(x))$. Since $\widehat f$ satisfies the same bounds, we obtain the claimed rates from Theorem~\ref{thm:W1_bound_sampling} using~\eqref{eq:main_generalization_bound} again. 

    Finally, suppose $f(x)=\sigma(\mbb Eg_k(\msf R_k(x)))$ and note that 
    \begin{equation*}
        |f(x)-f(\msf R_n(x))|\leq L_{\sigma}\|\mbb E_{\msf R_k}g_k(\msf R_k(x))-\mbb E_{\msf R_k}g_k(\msf R_k\circ\msf R_n(x))\|_1=L_{\sigma}\sum_{j=1}^{\ell}\Big|\mbb E[g_k]_{j}(\msf R_k(x))-\mbb E_{\msf R_k}[g_k]_{j}(\msf R_k\circ\msf R_n(x))\Big|,
    \end{equation*} 
    where we fix a realization of $\msf R_n$ and denote by $[g_k]_{j}$ the $j$th coordinate function of $g_k$. Similarly, we have $|f(x)-\mbb Ef(\msf R_n(x))|\leq L_{\sigma}\mbb E_{\msf R_n}\sum_{j=1}^{\ell}|\mbb E_{\msf R_k}[g_k]_{j}(\msf R_k(x))-\mbb E_{\msf R_k}[g_k]_{j}(\msf R_k\circ\msf R_n(x))|$. Since $\widehat f$ satisfies the same bounds, we obtain the claimed rates from Proposition~\ref{prop:concentration_coFS} using~\eqref{eq:main_generalization_bound} again. If $\sigma=\mathrm{id}$, we can improve the generalization bound because $|f(x)-\mbb Ef(\msf R_n(x))|\leq \|\mbb Eg_k(\msf R_k(x))-\mbb Eg_k(\msf R_k\circ\msf R_n(x))\|_1$ where both expectations are inside the norm, and hence we only incur an error due to the bias in Proposition~\ref{prop:concentration_coFS}.
    
    All the above bounds hold for average error with respect to projection-consistent distributions $(\mu_n)$. This can be seen by writing $\mu_n=\mathrm{Law}(\msf R_n(X))$ for a random $X\in\overline{\Omega}_{\infty}$ using Theorem~\ref{thm:proj_consistent_laws} and applying the above bounds conditionally on $X$.
\end{proof}

Finally, we prove Corollary~\ref{cor:transformers} using the measure-theoretic expression for attention in~\cite{furuya2025transformers}.
\begin{proof}[Proof (Corollary~\ref{cor:transformers}).]
    Let $\mc P_c(\RR^d)$ be the space of compactly supported probability measures on $\RR^d$, and consider the in-context attention mapping $\bar A\colon\RR^d\times\mc P_c(\RR^d)\to\RR^d$ introduced in~\cite{furuya2025transformers}, given by
    \begin{equation}\label{eq:in_context_attn}
        \bar A(z,\mu) = z + \frac{\mbb E_{Y\sim\mu}\exp(\langle Qz, KY\rangle)VY}{\mbb E_{Y\sim\mu}\exp(\langle Qz,KY\rangle)},
    \end{equation}
    which satisfies
    \begin{equation*}
        A(x)_i = \bar A(x_i,\mathrm{Law}(\msf R_1(x))),\quad \textrm{ for each } x\in(\RR^d)^n \textrm{ and } i\in[n].
    \end{equation*}
    Define a radius bound on all the layers of the transformer inductively by setting $R_0=\sup_{z\in\Theta}\|z\|_2\leq r\sqrt{d}$ and $R_{j+1}=\|\phi(0)\|_2 + L_{\phi}(1+\|V\|_{\mathrm{op}})R_j$. Set $t=\max_{j\leq \ell}R_j$ and denote $B_t=\{z\in\RR^d:\|z\|_2\leq t\}$. We can find Lipschitz constants $L_1,L_2>0$ such that $\bar A(x,\cdot)\colon\mc P(B_t)\to \RR^d$ is $L_1$-Lipschitz while $\bar A(\cdot,\mu)\colon B_t\to\RR^d$ is $L_2$-Lipschitz for any $x\in B_t$ and $\mu\in\mc P(B_t)$, see~\cite[Lemma~2]{furuya2026approximation} for example. 
    Furthermore, if $\mathrm{supp}(\mu)\subseteq B_t$ and $x\in B_t$, let $\mu_n$ be a random empirical measure obtained from $n$ iid samples from $\mu$ and note that with probability at least $1-2e^{-2\epsilon^2}$ we have
    \begin{equation}\label{eq:barA_sampling}
        \sup_{z\in B_t}\|\bar A(z,\mu) - \bar A(z,\mu_n)\|_2 \leq \frac{C(1+2\epsilon)}{\sqrt{n}},
    \end{equation}
    for a constant $C>0$. Indeed, consider the Hilbert space $\mc H = \bigoplus_{m\geq 0} (\RR^h)^{\otimes m}$ with inner product $\langle (u_m),(v_m)\rangle_{\mc H}=\sum_m\frac{1}{m!}\langle u_m,v_m\rangle$ and consider the map $G\colon \RR^h\to\mc H$ given by $G(z)=(z^{\otimes m})_m$. Note that $\|\mbb E_{Y\sim \mu}G(KY) - \mbb E_{Y\sim\mu_n}G(KY)\|_{\mc H}\leq \frac{M(1+2\epsilon)}{\sqrt{n}}$ with probability at least $1-e^{-2\epsilon^2}$, where $M=\sup_{y\in B_t}\|G(Ky)\|_{\mc H}$. This follows by the bounded-difference inequality. Therefore, for any $z\in B_t$ we have
    \begin{equation*}
        |\mbb E_{Y\sim\mu}\exp(\langle Qz,KY\rangle) - \mbb E_{Y\sim\mu_n}\exp(\langle Qz,KY\rangle)| = |\langle G(Qz), \mbb E_{Y\sim \mu}G(KY) - \mbb E_{Y\sim\mu_n}G(KY)\rangle| \leq M'\frac{M(1+2\epsilon)}{\sqrt{n}},
    \end{equation*}
    where $M'=\sup_{z\in B_t}\|G(Qz)\|_{\mc H}$.
    Similarly, noting that $\exp(\langle Qz,Ky\rangle)Vy = (G(Qz)^*\otimes I)(G(Ky)\otimes Vy)$, we have
    \begin{equation*}
        \|\mbb E_{Y\sim\mu}\exp(\langle Qz,KY\rangle)VY - \mbb E_{Y\sim\mu_n}\exp(\langle Qz,KY\rangle)VY\|_2 \leq M'\frac{M''(1+2\epsilon)}{\sqrt{n}},
    \end{equation*}
    where $M''=\sup_{y\in B_t}\|G(Ky)\|_{\mc H}\cdot \|Vy\|$.
    Since the denominator in~\eqref{eq:in_context_attn} is uniformly bounded on $B_t\times\mc P(B_t)$ above and below, we obtain~\eqref{eq:barA_sampling}.

    Next, we fix a realization of $F_n\colon[n]\to[N]$, set $\msf R_n=\rho(F_n)$, and fix $x\in \Theta^N$. Define $$\Delta_\ell = \sup_{i\in[n]}\|[\msf R_nT_1^{\circ \ell}(x)]_i-[T_1^{\circ\ell}(\msf R_nx)]_i\|_2,$$ where $T_1=\phi\circ A$. In words, $\Delta_{\ell}$ measures the extend to which $\ell$ layers of $T_1$ commute with sampling. We claim that with probability at least $1-2\ell e^{-2\epsilon^2}$ we have
    \begin{equation}\label{eq:Delta_ell_bound}
        \Delta_{\ell}\leq \frac{C'(1+2\epsilon)}{\sqrt{n}},
    \end{equation}
    for some constant $C'>0$. To this end, note that $\msf R_n\circ \phi = \phi\circ \msf R_n$ as $\phi$ is applied column-wise. Furthermore, define the three measures
    \begin{equation*}
        \mu = \mathrm{Law}(\msf R_1T_1^{\circ \ell}(x)),\quad \mu_n = \mathrm{Law}(\msf R_1\msf R_nT_1^{\circ \ell}(x)),\quad \widehat \mu_n = \mathrm{Law}(\msf R_1T_1^{\circ \ell}(\msf R_nx)),
    \end{equation*}
    where we recall that the realization of $\msf R_n$ is fixed and that the randomness above only comes from $\msf R_1$. Note that $\mu_n$ is the empirical measure obtained from $n$ iid samples from $\mu$, while $W_1(\mu_n,\widehat\mu_n)\leq \Delta_{\ell}$.
    Therefore,
    \begin{equation*}\begin{aligned}
        \Delta_{\ell+1} &\leq L_{\phi}\sup_{i\in[n]}\|A(T_1^{\circ \ell}(x))_{F_n(i)} - A(T_1^{\circ \ell}(\msf R_nx))_i\|_2 = L_{\phi}\sup_{i\in[n]}\|\bar A(T_1^{\circ \ell}(x)_{F_n(i)}, \mu) - \bar A(T_1^{\circ \ell}(\msf R_n x)_{i}, \widehat\mu_n)\|_2\\
        &\leq L_{\phi}\left[\sup_{i\in[n]}\|\bar A(T_1^{\circ \ell}(x)_{F_n(i)}, \mu) - \bar A(T_1^{\circ \ell}(x)_{F_n(i)}, \mu_n)\|_2 + \sup_{i\in[n]}\|\bar A(T_1^{\circ \ell}(x)_{F_n(i)}, \mu_n) - \bar A(T_1^{\circ \ell}(\msf R_n x)_{i}, \widehat\mu_n)\|_2\right]\\
        &\leq L_{\phi}\left[\frac{C(1+2\epsilon)}{\sqrt{n}} + L_2\sup_i\|T_1^{\circ \ell}(x)_{F_n(i)} - T_1^{\circ\ell}(\msf R_nx)_i\|_2 + L_1W_1(\mu_n,\widehat\mu_n)\right]\\
        &\leq L_{\phi}(L_1+L_2)\Delta_{\ell} + \frac{L_{\phi}C(1+2\epsilon)}{\sqrt{n}},
    \end{aligned}\end{equation*}
    where the third inequality holds with probability at least $1-2e^{-2\epsilon^2}$.
    Inducting on $\ell$ and applying a union bound, we obtain the claim.

    Finally, note that
    \begin{equation*}\begin{aligned}
        \|T(x)-T(\msf R_nx)\|_2 &= \|\mbb E\msf R_1T_1^{\circ \ell}(x) - \mbb E\msf R_1T_1^{\circ\ell}(\msf R_nx)\|_2 \\&\leq \|\mbb E\msf R_1T_1^{\circ\ell}(x) - \mbb E\msf R_1\msf R_nT_1^{\circ\ell}(x)\|_2 + \|\mbb E\msf R_1\msf R_nT_1^{\circ\ell}(x) - \mbb E\msf R_1T_1^{\circ\ell}(\msf R_nx)\|_2\\ &\leq \frac{C''(1+2\epsilon)}{\sqrt{n}} + \Delta_{\ell},
    \end{aligned}\end{equation*}
    with probability at least $1-2e^{-2\epsilon^2}$, which, together with~\eqref{eq:Delta_ell_bound} concludes the claimed sketching bound.
    %
    %
    The first generalization bound now follows as in the proof of Theorem~\ref{thm:sampling_rates_funcs}. For the second generalization bound, defining $\mc A\colon\mc P_c(\RR^d)\to\mc P_c(\RR^d)$ by $\mc A(\mu)=\mathrm{Law}(\bar A(X,\mu))$ where $X\sim\mu$, and letting $\phi$ act on measures via pushforward, observe that $T(x)=\bar T(\mathrm{Law}(\msf R_1(x)))$ where $\bar T(\mu)=\mbb E_{X\sim (\phi\circ \mc A)^{\circ \ell}(\mu)}X$. Recalling that $\bar A(x,\cdot)\colon\mc P(B_t)\to \RR^d$ is $L_1$-Lipschitz while $\bar A(\cdot,\mu)\colon B_t\to\RR^d$ is $L_2$-Lipschitz for any $x\in B_t$ and $\mu\in\mc P(B_t)$, it is easy to see that $\bar T$ is $[L_{\phi}(L_1+L_2)]^{\ell}$-Lipschitz with respect to $W_1$. The last generalization bound now follows from Theorem~\ref{thm:W1_bound_sampling}.
\end{proof}

\section{Rates for Random Binning and Species Sampling}\label{sec:quotients}
In this section, we prove rates for sketching and generalization using the sampling metric defined with respect to random binning from Section~\ref{sec:sampling_and_binning}, and relate the topologies defined by random binning and species sampling via their corresponding sampling metrics.

\subsection{Main Results}
We fix a compatible sequence of index sets $(\mc I_n)$, and consider the associated vector spaces $(\vct V_n=\RR^{\mc I_n})$ and random binning maps $(\msf B_k\colon\bigsqcup_n\vct V_n\to\vct V_k)$ defined by~\eqref{eq:binning_general}. We also fix a sequence of compact sets $(\Omega_n)$ closed under random binning, such as the sequence of $\ell_1$ balls $\mc B_{\ell_1}^{(n)}(r)=\{x\in\vct V_n:\|x\|_1\leq r\}$.  Proceeding as in Section~\ref{sec:sampling_wrep}, it may be tempting to try to prove a result akin to Theorem~\ref{thm:W1_bound_sampling} in which we might show that $d_{\mathrm{samp}}(x,\msf B_n(x))\to0$ with high probability, and quantify the rate at which it converges to zero.  Unfortunately, while we will show below that $\mbb Ed_{\mathrm{samp}}(x,\msf B_n(x))\to0$, the following example illustrates that individual realizations $d_{\mathrm{samp}}(x,\msf B_n(x))$ do not concentrate as well around their means as in the case of sampling with replacement.

\begin{example}[Anti-concentration for random binning]\label{ex:anti_concentration}
    Suppose $\vct V_n=\RR^n$ and consider $x=(1/2,1/2)\in\vct V_2$. Observe that $\msf B_n(x)$ equals one of the standard basis vectors with probability $1/n$, since this is the probability that the two entries of $x$ to map to the same bin. Therefore, with probability $1/n$, we have that $d_{\mathrm{samp}}(x,\msf B_n(x)) = d_{\mathrm{samp}}(x,e_1) = \sum_{k\geq 2}2^{-k}(1-\frac{1}{k})\frac{1}{\sqrt{2}}=\frac{1-\log 2}{\sqrt{2}}$ where we define $d_{\mathrm{samp}}$ with respect to the $\ell_2$ norms on the $\vct V_n$.  
    Thus, for any rate $(\epsilon_n\geq0)$ with $\epsilon_n\downarrow 0$ we have $\sum_{n\geq 1}\mbb P[d_{\mathrm{samp}}(x,\msf B_n(x))\geq \epsilon_n]=\infty$.
\end{example}
In contrast, for sampling with replacement Theorem~\ref{thm:W1_bound_sampling} gives a rate $\epsilon_n$ such that $\mbb P[d_{\mathrm{samp}}(x,\msf R_n(x))\geq \epsilon_n]\leq Cn^{-2}$ for a constant $C>0$. 

To remedy this situation, we show in Theorem~\ref{thm:W1_rates_FS} that $d_{\mathrm{samp}}(x,\msf E_n(x))$ exhibits the requisite concentration, where $\msf E_n$ is the species sampling map from Section~\ref{sec:species}.  
In other words, species sampling furnishes the requisite sketching map for the sampling metric defined via random binning. 
Subsequently, we follow a similar agenda as in Section~\ref{sec:sampling_wrep}.  
We characterize the any-dimensional data distributions that are obtained via random binning in Theorem~\ref{thm:eqp_consistent_laws}.  
We then give explicit any-dimensional sketching and generalization rates in Theorem~\ref{thm:binning_rates_funcs}, followed by corollaries that provide concrete illustrations of our results in the context of applications in graph signal processing and neural networks. We conclude with Proposition~\ref{prop:relation_to_species} in which we relate convergence with respect to random binning and species sampling.

We begin by quantifying the convergence $d_{\mathrm{samp}}(x,\msf E_n(x))\to 0$ in expectation and with high probability.
\begin{theorem}\label{thm:W1_rates_FS}
    Suppose $(\mc I_n)$ is a compatible sequence of index sets of degree $D$, and consider the sampling metric defined using random binning $(\msf B_k)$ and the $W_1$-metric with respect to the $\ell_2$ norm. Fix $(\Omega_n\subseteq \RR^{\mc I_n})$ closed under both binning and species sampling such that $\sup_{n\in\NN}\sup_{x\in\Omega_n}\|x\|_1\leq r$.
    \begin{enumerate}[labelwidth=!, labelindent=0pt]
        \item For any $x\in\bigsqcup_n\Omega_n$ and $k\in\NN$, we have 
        \begin{equation*}
            \mbb E_{\msf E_n}W_1(\msf B_k(x),\msf B_k\circ\msf E_n(x))\leq \frac{r}{\sqrt{n}},
        \end{equation*}
        and similarly $\mbb Ed_{\mathrm{samp}}(x,\msf E_n(x))\leq r/\sqrt{n}$. Moreover, both $W_1(\msf B_k(x),\msf B_k\circ\msf E_n(x))$ and $d_{\mathrm{samp}}(x,\msf E_n(x))$ are $\frac{4r^2}{n}$-subgaussian with respect to the randomness in $\msf E_n$.

        \item For any $x\in\bigsqcup_n\Omega_n$ and any $k\in\NN$, we have 
        \begin{equation}\label{eq:expected_binning_rate}
            \mbb E_{\msf B_n}W_1(\msf B_k(x),\msf B_k\circ\msf B_n(x))\leq r\sqrt{\frac{2D(3D-1)}{n}},
        \end{equation}
        and similarly $\mbb Ed_{\mathrm{samp}}(x,\msf B_n(x))\leq r\sqrt{\frac{2D(3D-1)}{n}}$. We also have $\|\mbb E\msf B_k(x) - \mbb E\msf B_k(\msf B_n(x))\|_1\leq r\frac{D(D-1)}{n}$.
    \end{enumerate}
\end{theorem}
We emphasize again that although Theorem~\ref{thm:W1_rates_FS}(2) shows that $\mbb Ed_{\mathrm{samp}}(x,\msf B_n(x))\to0$, these distances do not concentrate well around their means as shown in Example~\ref{ex:anti_concentration}. The last $O(1/n)$ bound in Theorem~\ref{thm:W1_rates_FS}(2) gives us an improved generalization rate for polynomials unchanged by zero-padding.

The proof of Theorem~\ref{thm:W1_rates_FS}(1) is based on an extension of the following standard bound for the empirical sampling map~\eqref{eq:empirical_sampling},
\begin{equation}\label{eq:Binom_bound}
    \mbb E\|x-\mscr E_n(x)\|_2 \leq \|x\|_1\sqrt{\sum_{\alpha\in\mc I_N}\mbb E\Big(p_{\alpha} -\tfrac{1}{n}\mathrm{Binom}(n,p_{\alpha})\Big)^2} \leq \frac{\|x\|_1}{\sqrt{n}},\quad \textrm{for any } x\in\RR^{\mc I_N},\ N\in\NN,
\end{equation}
where $p_{\alpha}=|x_{\alpha}|/\|x\|_1$ are the probabilities proportional to the magnitudes of entries of $x$, together with the fact that $d_{\mathrm{samp}}(\mscr E_n(x),\msf E_n(x))=0$ almost surely.
The proof of Theorem~\ref{thm:W1_rates_FS}(2) uses an explicit coupling between $\msf B_k(x)$ and $\msf B_k\circ\msf B_n(x)$.
We give the full details in Section~\ref{sec:proofs_quotients}.
Examples of sets satisfying the hypotheses of Theorem~\ref{thm:W1_rates_FS} include the sequence of simplices $\Omega_n=\Delta^{\mc I_n}$ and $\ell_1$-balls. 

We proceed to describe some of the implications of Theorem~\ref{thm:W1_rates_FS}.  First, the fact that $\mbb Ed_{\mathrm{samp}}(x,\msf B_n(x))\to0$ at a universal rate as $n\to\infty$ allows us to apply Proposition~\ref{prop:fd_measures_as_limits} and leads to a characterization of sequences of any-dimensional distributions that are compatible with random binning.
\begin{theorem}[Equipartition-consistent distributions]\label{thm:eqp_consistent_laws}
    Suppose $(\mc I_n)$ has finite degree, let $\vct V_n=\RR^{\mc I_n}$, and suppose $(\Omega_n)$ is a sequence of compact sets closed under binning with $\sup_n\sup_{x\in\Omega_n}\|x\|_1<\infty$. Then the following are equivalent for a sequence of probability distributions $(\mu_n\in\mc P(\Omega_n))$: 
    \begin{enumerate}[align=left, font=\emph, labelwidth=!, labelindent=0pt]
        \item[(Equipartition consistency)] We have $\beta(\psi_{n,N})\mu_N=\mu_n$ whenever $ \psi_{n,N}\colon[N]\to[n]$ is an equipartition.
        
        \item[(Sampling representation)] There exists a distribution $\mu_{\infty}\in\mc P(\overline{\Omega}_{\infty})$ such that $\mu_n=\mathrm{Law}(\msf B_n(X))$ for all $n$, where $X\sim\mu_{\infty}$ is independent of $\msf B_n$.
    \end{enumerate}
    Moreover, the measure $\mu_{\infty}$ above is unique, and sequences of the form $(\mu_n=\mathrm{Law}(\msf B_n(x)))$ for deterministic $x\in\overline{\Omega}_{\infty}$ are extremal in the set of all such sequences.
\end{theorem}
Similarly to Theorem~\ref{thm:proj_consistent_laws}, the proof combines Proposition~\ref{prop:fd_measures_as_limits} with a ``dual'' finite de Finetti theorem comparing uniformly random binning and random binning into equally sized bins from~\cite{levin2025deFin}, see Section~\ref{sec:proofs_quotients}.
We give some examples of the any-dimensional data distributions obtained by randomly binning limit objects. Using Theorem~\ref{thm:W1_rates_FS}, we then give generalization rates for average error with respect to these distributions.
\begin{example}
    If $\mc I_n=[n]$ and $\Omega_n=\Delta^n$, then a sequence $(\mu_n\in\mc P(\Delta^n))$ is equipartition-consistent if and only if (i) each $\mu_n$ is exchangeable; (ii) whenever $n|N$ we have
    \begin{equation*}
        \left(\sum_{j=1}^{N/n}(X_N)_{j + (i-1)(N/n)}\right)_{i=1}^n \overset{d}{=} X_n.
    \end{equation*}
    For example, the sequence of Dirichlet distributions $\mu_n=\mathrm{Dir}(\alpha\mathbbm{1}_n/n)$ satisfies these conditions for any $\alpha\in[0,\infty)$, where $\alpha=0$ corresponds to a uniformly random coordinate vector.

    If $\mc I_n=[n]^2$, then equipartition-consistent sequences $(\mu_n\in\mc P(\Delta^{n\times n}))$ were called ``equipartition-consistent random graph models'' in~\cite{levin2025graphs}, where they were shown to correspond to certain limits of growing-sized graphs. In particular, the specialization of Theorem~\ref{thm:eqp_consistent_laws} to this case was proved in~\cite[Thm.~1.9]{levin2025graphs}.

\end{example}

The next consequences of Theorem~\ref{thm:W1_rates_FS} are the following sketching and generalization rates, proved in Section~\ref{sec:proofs_quotients}. In contrast to the analogous Theorem~\ref{thm:sampling_rates_funcs} for sampling with replacement, here we get a rate of $n^{-1/2}$ regardless of the latent low-dimensional structure of our functions.
\begin{theorem}[Sketching and generalization rates]\label{thm:binning_rates_funcs}
    In the setting of Theorem~\ref{thm:W1_rates_FS}, suppose $f\colon\bigsqcup_n\Omega_n\to\RR$ satisfies one of the following conditions.
    \begin{enumerate}[align=left, font=\emph,labelwidth=!, labelindent=0pt]
        \item[(General)] $f$ is $L$-Lipschitz in $d_{\mathrm{samp}}$;
        \item[(Fixed-dimensional laws)] $f$ is $L$-Lipschitz in $(x,y)\mapsto W_1(\msf B_k(x),\msf B_k(y))$ for some $k\in\NN$;
        \item[(2-norm continuity)] $f$ is unchanged by zero-padding so $f\circ\beta(\phi_{N,n})=f$ for any injection $\phi_{N,n}\colon[n]\to[N]$, and $f|_{\Omega_n}$ is $L$-Lipschitz in $\ell_2$ norm for all $n$.
    \end{enumerate}
    Then for any $x\in\bigsqcup_n\Omega_n$ with probability at least $1-e^{-2\epsilon^2}$ we have
    \begin{equation}\label{eq:FS_sketching}
        |f(x)-f(\msf E_n(x))|\leq \frac{Lr(1 + 2\epsilon)}{\sqrt{n}}.
    \end{equation}
    If $\widehat f$ is another such function and $n \geq D$, then
    \begin{equation}\label{eq:FS_generalization}
        \mathrm{e}_{\infty}(f,\widehat f)\leq \mathrm{e}_{n}(f,\widehat f) + 2Lr\sqrt{\frac{1}{\lfloor n/D\rfloor}}.
    \end{equation}
    If $(\mu_n)$ is equipartition-consistent and $\mu_{\infty}$ is the representing measure from Theorem~\ref{thm:eqp_consistent_laws}, we also have
    \begin{equation*}
        \mathrm{e}_{\mu_{\infty}}(f,\widehat f)\leq \mathrm{e}_{\mu_n}(f,\widehat f) + 2Lr\sqrt{\frac{2D(3D-1)}{n}}.
    \end{equation*}
\end{theorem}
The proof is a direct application of the above two theorems, and we defer it to Section~\ref{sec:proofs_quotients}. 
The three conditions on the function $f$ in Theorem~\ref{thm:binning_rates_funcs} are different sufficient conditions for continuity of $f$ with respect to random binning. Some of these conditions are more conveniently applicable to different function classes, as we proceed to illustrate.

As in Section~\ref{sec:sampling_wrep}, a prominent example of functions satisfying the above hypotheses is polynomials. 
\begin{corollary}\label{cor:FS_polynomials}
    In the setting of Theorem~\ref{thm:W1_rates_FS}, let $p\colon\bigsqcup_n\Omega_n\to\RR$ be a polynomial unchanged by zero-padding, so $p\circ\beta(\phi_{N,n})=p$ as polynomials for all injections $\phi_{N,n}\colon[n]\to[N]$ and all $n\leq N$. Then $p$ satisfies the second condition in Theorem~\ref{thm:binning_rates_funcs}.  Moreover, the following improved $O(1/n)$ generalization bounds hold. If $p$ and $\hat p$ are two such polynomials of degree $d$, there exists a constant $L'>0$ such that
    \begin{equation}
        \mathrm e_{\infty}(p,\widehat p) \leq \mathrm e_n(p, \widehat p) + L'\max\{r^d,1\}\frac{d(d+1)D(dD-1)}{n}.
    \end{equation}
    The same generalization bound applies for average error with respect to equipartition-consistent distributions.
\end{corollary}
The proof uses a result of~\cite{levin2025deFin} expressing such a polynomial as the moment of a fixed-sized random binning of its input together with Theorem~\ref{thm:W1_rates_FS}, see Section~\ref{sec:proofs_quotients}. We thus obtain $O(1/\sqrt{n})$ sketching rates and $O(1/n)$ generalization rates for polynomials unchanged by zero-padding, similarly to the rates for polynomials unchanged by duplication given in Corollary~\ref{cor:moment_polys}.
The following are a few examples of the polynomials to which Corollary~\ref{cor:FS_polynomials} applies, which constitute ``unnormalized'' analogs of the polynomials unchanged by duplication considered in Corollary~\ref{cor:moment_polys}.
\begin{enumerate}[align=left, font=\emph, labelwidth=!, labelindent=0pt, wide]
    \item[(Multisymmetric Functions)] When $\vct V_n=\RR^{d\times n}$, polynomials unchanged by zero-padding are precisely multisymmetric functions~\cite{multisym_funcs}, consisting of polynomials of the form $$p(x)=q\Big(p_{\alpha_1}(x),\ldots,p_{\alpha_k}(x)\Big),$$ where $q\in\RR[y_1,\ldots,y_k]$ is a fixed polynomial, and for $\alpha\in\NN_0^d\setminus\{0\}$ we have $p_{\alpha}(x)=\sum_{i=1}^n\prod_{j=1}^dx_{j,i}^{\alpha_j}$. When $d=1$ and $\alpha\in\NN$, the polynomials $p_{\alpha}(x)=\sum_{i=1}^nx_i^{\alpha}$ are known as power-sum polynomials, and polynomials unchanged by zero-padding are called symmetric functions. These are classic objects of study in combinatorics and representation theory, see~\cite[Chap.~7]{Stanley_Fomin_1999} and~\cite{macdonald1998symmetric}.

    \item[((Hyper)Graph Numbers)] When $\vct V_n=\RR^{n\times n}$, polynomials unchanged by zero-padding are precisely linear combinations of graph homomorphism numbers~\cite[\S6.4]{levin2025deFin}. Specifically, if $H\in \NN_0^{k\times k}$ is a multigraph with no isolated vertices, we define the homomorphism number of $H$ in $G\in\vct V_n$ by
    \begin{equation*}
        \mathrm{hom}(H;G)=\sum_{f\colon[k]\to[n]}\prod_{i,j=1}^kG_{f(i),f(j)}^{H_{i,j}}.
    \end{equation*}
    When $H,G$ are simple graphs, the value $\mathrm{hom}(H;G)$ is the number of graph homomorphisms from $H$ to $G$.
    Similarly, when $\vct V_n=(\RR^n)^{\otimes d}$ polynomials unchanged by zero-padding are linear combinations of hypergraph homomorphism numbers, analogously defined.

    \item[(Graph Signals)] When $\vct V_n=\RR^{n\times n}\oplus\RR^{n\times d}$, polynomials unchanged by zero-padding include polynomial (unnormalized) graph neural networks, which are compositions of polynomials of the form
    \begin{equation*}
        (G,X)\mapsto \left(G,\sigma\left(\sum_{\ell=0}^LG^{\ell}X\Theta_{\ell}\right)\right),\quad \textrm{for polynomial } \sigma \textrm{ with } \sigma(0)=0 \textrm{ and } \Theta_0,\ldots,\Theta_L\in\RR^{d\times d},
    \end{equation*}
    followed by a zero-padding invariant final layer such as $(G,X)\mapsto X^\top\mathbbm{1}$.
\end{enumerate}
For all of the above families of polynomials, we obtain $O(1/\sqrt{n})$ sketching and $O(1/n)$ generalization rates.
We now turn to proving sketching and generalization rates for some non-polynomial function classes.
\begin{corollary}[DeepSets]\label{cor:deepsets}
    Let $\vct V_n=\RR^{d\times n}$ and consider maps of the form $f(x_1,\ldots,x_n)=\sigma(\sum_i\rho(x_i))$ for $\rho\colon\RR^d\to\RR^{\ell}$ with $\rho(0)=0$ and $\sigma\colon\RR^{\ell}\to\RR$. Then $f$ is unchanged by zero-padding.
    \begin{enumerate}[labelwidth=!, labelindent=0pt]
        \item Suppose $\sigma$ is $L_{\sigma}$-Lipschitz, and that $\rho$ is $L_{\rho}\sqrt{r}$-Lipschitz with respect to the $\ell_2$ norm on $\{x\in\RR^d:\|x\|_1\leq r\}$ for each $r$. Then $f$ satisfies the third condition in Theorem~\ref{thm:binning_rates_funcs} with $L=L_{\rho}L_{\sigma}\sqrt{2r}$.
        \item The 1-norm $f(x)=\sum_i|x_i|$ is discontinuous in sampling metric. 
    \end{enumerate}
\end{corollary}


The proof is a direct computation, see Section~\ref{sec:proofs_quotients}.  Although the 1-norm is of the form $f(x_1,\ldots,x_n)=\sigma(\sum_i\rho(x_i))$ for $d=1$, $\sigma = \mathrm{id}_{\RR}$ and $\rho = |\cdot|$, both of which are 1-Lipschitz, the stronger condition on $\rho$ in part 1, requiring its Lipschitz constant to decay to zero around the origin, fails to hold for the absolute value function.  Therefore, mere Lipschitz continuity of $\rho$ (as well as $\sigma$) is not by itself sufficient to obtain continuity with respect to random binning.  The hypothesis on $\rho$ in part 1 is satisfied if, for example, $\rho$ is differentiable with a $1/2$-H\"older continuous derivative satisfying $\rho'(0)=0$.


When $\sigma$ and $\rho$ in Corollary~\ref{cor:deepsets} are neural networks, the resulting architecture is called DeepSets~\cite{deepsets}. When $d=1$ and $\Omega_n=\Delta^n$, different choices of $\sigma$ and $\rho$ yield various diversity indices used to quantify the concentration of a discrete probability distribution~\cite{Zhang02072016}. For example, setting $\rho(x)=x^q$ and $\sigma(t)=\frac{1-t}{q-1}$ gives Tsallis’ entropy, which for $q\geq 3/2$ satisfies the hypotheses of Corollary~\ref{cor:deepsets}(1). We note however that some diversity indices are not continuous in sampling metric. This is the case for Shannon and R\'enyi entropies (given by $\sigma=\mathrm{id}$, $\rho(t)=t\log \frac{1}{t}$ for the former and $\sigma(t)=\frac{1}{1-q}\log(t)$, $\rho(t)=t^q$ for the latter) for example, both of which diverge along the (convergent) sequence $x^{(n)}=\mathbbm{1}_n/n$.

\begin{corollary}[PointNet]\label{cor:pointnet}
    Let $\vct V_n=\RR^{d\times n}$ and consider maps of the form $f(x_1,\ldots,x_n)=\sigma(\sup_i\rho(x_i))$ for $\rho\colon\RR^d\to\RR^{\ell}$ and $\sigma\colon\RR^{\ell}\to\RR$, where the supremum is taken coordinate-wise. If $\rho(x)\geq0$ for all $x$, $\rho(0)=0$, and if $\sigma$ and $\rho$ are $L_{\sigma}$ and $L_{\rho}$-Lipschitz continuous in $\ell_2$-norm, then the function $f$ satisfies the third condition in Theorem~\ref{thm:binning_rates_funcs} with $L=L_{\rho}L_{\sigma}$.
\end{corollary}
\begin{proof}
    Direct computation.
\end{proof}
Finally, we consider non-polynomial graph neural networks.
\begin{corollary}[Graph Neural Networks]\label{cor:gnns}
    Let $\vct V_n=\RR^{n\times n}\oplus\RR^{n\times d}$ and consider a composition of maps of the form
    \begin{equation*}
        F(A,X)=\left(A, \sigma\left(\sum_{d=0}^DA^dX\Theta_d\right)\right),\quad \textrm{where } \Theta_0,\ldots,\Theta_D\in\RR^{d\times d},
    \end{equation*}
    and $\sigma\colon\RR^d\to\RR^d$ satisfies $\sigma(0)=0$, is $L_{\sigma}$-Lipschitz in $\ell_2$, and is applied row-wise. Also define $P\colon\vct V_n\to\RR^d$ by $P(A,X)=\rho(X)^\top\mathbbm{1}_n$ for $\rho\colon\RR^d\to\RR^d$ that is applied row-wise and satisfies the hypotheses of Corollary~\ref{cor:deepsets}(1). Consider the depth-$\ell$ graph neural network given by the composition
    \begin{equation*}
        f = P\circ F^{\circ \ell}\colon\bigsqcup_n\mc B_{\ell_1}^{(n)}(r)\to\RR^d,
    \end{equation*}
    where $\mc B_{\ell_1}^{(n)}(r)=\{(A,X)\in\vct V_n:\|A\|_1+\|X\|_1\leq r\}$.
    Then each coordinate function of $f$ satisfies the third condition in Theorem~\ref{thm:binning_rates_funcs}.
\end{corollary}
The proof is an elementary but long computation, and we defer it to Section~\ref{sec:proofs_quotients}. 
For all of the function classes in the above corollaries, Theorem~\ref{thm:binning_rates_funcs} yields $n^{-1/2}$ sketching and generalization rates.

So far in this section, we have worked with the sampling metric defined by random binning, and showed that species sampling is an appropriate sketching map with respect to this metric. We end this section by further considering the sampling metric defined by species sampling, and show that its induced topology is closely related to the one induced by random binning. 
\begin{proposition}\label{prop:relation_to_species}
    Let $(\mc I_n)$ be a compatible sequence of finite degree, and let $(x_i)\subseteq\bigsqcup_n\RR^{\mc I_n}$. Then $(\msf E_k(x_i))$ converges for all $k$ if and only if the sequence of random tuples $((\msf B_k(x_i^+),\msf B_k(x_i^-)))_i$ converges for all $k$, where $x_i^+=\max\{x_i,0\}$ and $x_i^-=\max\{-x_i,0\}$ are the entrywise positive and negative parts of $x_i$, and we apply the same random map $\msf B_k$ to both $x_i^+$ and $x_i^-$.
\end{proposition}
In words, convergence of the species samples is equivalent to joint convergence of random binning of the positive and negative parts. 
In particular, the two notions of convergence are equivalent for nonnegative vectors.
The special case of this equivalence for $\mc I_n=[n]^2$ and $\Omega_n=\Delta^{n\times n}$ was shown in~\cite[Prop.~4.1]{levin2025graphs}.

We remark that species sampling gives another sequence of any-dimensional data distributions, and our results yield generalization rates with respect to these distributions. For example, if $(\Omega_n=\Delta^{\mc I_n})$ then we get the same completion $\overline{\Omega}_{\infty}$ with respect to both $(\msf B_k)$ and $(\msf E_k)$ by Proposition~\ref{prop:relation_to_species}. 
Combining Theorem~\ref{thm:W1_rates_FS} and Proposition~\ref{prop:fd_measures_as_limits} shows that each probability distribution $\mu\in\mc P(\overline{\Omega}_{\infty})$ on limit objects yields two sequences of distributions on finite objects converging weakly to $\mu$, namely $(\mathrm{Law}(\msf B_k(X)))$ and $(\mathrm{Law}(\msf E_k(X)))$ for $X\sim\mu$ independent of $\msf B_k$ and $\msf E_k$.
Both sequences of measures can be used as any-dimensional distributions for which $n^{-1/2}$ generalization rates are available. 
It would be interesting to characterize sequences of the form $(\mathrm{Law}(\msf E_k(X)))$ analogously to Theorem~\ref{thm:eqp_consistent_laws}. For example, when $\mc I_n=[n]^2$ such sequences of distributions are precisely edge-exchangeable multigraph models normalized to have unit edge weights, as shown in~\cite[Rmk.~4.2]{levin2025graphs}. 
We leave such a characterization more generally for future work.

\subsection{Missing Proofs from Section~\ref{sec:quotients}}\label{sec:proofs_quotients}
We begin by proving Theorem~\ref{thm:W1_rates_FS}. We shall need the following several lemmas to do so. 
The first lemma shows that the species sampling $\msf E_n$ and empirical sampling $\mscr E_n$ maps in~\eqref{eq:species_sampling_general} and~\eqref{eq:empirical_sampling}, respectively, are equivalent in the sampling metric defined by random binning. 
\begin{lemma}[Equivalence of species and empirical sampling]\label{lem:empirical_species_equivalent}
    Suppose $(\mc I_n)$ has finite degree and let $\vct V_n=\RR^{\mc I_n}$. Then $d_{\mathrm{samp}}(\mscr E_n(x),\msf E_n(x))=0$ for any $x\in\vct V_N$ and any $n,N\in\NN$.
\end{lemma}
\begin{proof}
    After zero-padding $x$ if necessary, which does not change the distributions of $\mscr E_n(x)$ and $\msf E_n(x)$ by Proposition~\ref{prop:species_samp_zero_pad}, we may assume $N\geq nD$. We construct a random injection $\Phi_{N,nD}\colon[nD]\to[N]$ satisfying
    \begin{equation}\label{eq:random_inj_empirical_species}
        \mscr E_n(x)=\beta(\Phi_{N,nD})\msf E_n(x),
    \end{equation}
    which proves that $d_{\mathrm{samp}}(\mscr E_n(x),\msf E_n(x))=0$ by~\eqref{eq:injection_equivalence}.
    To construct this injection, recall the randomly-enumerated indices $\{t_1,\ldots,t_k\}$ from the construction of $\mscr E_n$ and $\msf E_n$ in Section~\ref{sec:species}, where $k\leq nD$. Set $\Phi_{N,nD}(i)=t_i$ for $i\leq k$ and set $\Phi_{N,nD}|_{[nD]\setminus[k]}$ to be a uniformly random injection into $[N]\setminus\{t_1,\ldots,t_k\}$. Note that~\eqref{eq:random_inj_empirical_species} is satisfied by construction of $\msf E_n(x)$ in~\eqref{eq:species_sampling_general}.
\end{proof}

As we shall see, empirical sampling satisfies several useful properties that we exploit to prove Theorem~\ref{thm:W1_rates_FS}. 
By exploiting these properties, we now prove the first bound from Theorem~\ref{thm:W1_rates_FS}(1).
\begin{lemma}\label{lem:W1_rate_finite_x}
    Suppose $(\mc I_n)$ has finite degree. For any $x\in\RR^{\mc I_N}$ and any $N,n,k\in\NN$, we have
    \begin{equation*}
        \mbb E_{\msf E_n}W_1(\msf B_k(x),\msf B_k\circ\msf E_n(x))\leq \frac{\|x\|_1}{\sqrt{n}}.
    \end{equation*}
\end{lemma}
\begin{proof}
    By homogeneity of $\msf E_n$ and linearity of $\msf B_k$, it suffices to prove the claim when $\|x\|_1=1$, which we assume for the remainder of the proof. We have
    \begin{equation*}
        \mbb EW_1(\msf B_k(x),\msf B_k\circ\msf E_n(x)) = \mbb EW_1(\msf B_k(x),\msf B_k\circ\mscr E_n(x))\leq \mbb E\|\beta(F_{k,N})(x-\mscr E_n(x))\|_2\leq \sqrt{\mbb E\|\beta(F_{k,N})(x-\mscr E_n(x))\|_2^2},
    \end{equation*}
    where the equality follows from Lemma~\ref{lem:empirical_species_equivalent} and the first inequality follows by using the same random map $F_{k,N}$, chosen independently of $\mscr E_n$, to define the coupling. Write $\beta(F_{k,N})\mscr E_n(x)=\frac{1}{n}\sum_{i=1}^ns_ie_{\theta(F_{k,N})(\alpha_i)}$ where $\alpha_1,\ldots,\alpha_n\in\mc I_N$ are sampled iid from $|x|$ and $s_i=\mathrm{sign}(x_{\alpha_i})$. We can then write $\beta(F_{k,N})(x-\mscr E_n(x)) = \frac{1}{n}\sum_{i=1}^n\delta_i$ where 
    \begin{equation*}
        \delta_i = s_ie_{\theta(F_{k,N})(\alpha_i)} - \beta(F_{k,N})x = s_ie_{\theta(F_{k,N})(\alpha_i)} - \mbb Es_ie_{\theta(F_{k,N})(\alpha_i)}.
    \end{equation*}
    where the last equality is a direct computation. Note that conditioned on $F_{k,N}$, the different $\delta_i$ are iid, and each one satisfies $\mbb E\delta_i=0$ and $\mbb E\|\delta_i\|^2 \leq \mbb E\|s_ie_{\theta(F_{k,N})(\alpha_i)}\|_2^2=1$. We conclude that $\mbb E\|\beta(F_{k,N})(x-\mscr E_n(x))\|_2^2\leq \frac{1}{n}$ by first conditioning on $F_{k,N}$, as desired.
\end{proof}
Next, we prove the bound~\eqref{eq:expected_binning_rate} from Theorem~\ref{thm:W1_rates_FS}(2) using an explicit coupling between $\msf B_k(x)$ and $\msf B_k\circ\msf B_n(x)$. Along the way, we prove a bound between $x$ and an appropriate zero-padding of the binned $\msf B_n(x)$, to be used in the proof of Theorem~\ref{thm:binning_rates_funcs}.
\begin{lemma}\label{lem:binning_rate}
    Suppose $(\mc I_n)$ has degree $D$. For any $x\in\RR^{\mc I_N}$ and any $N\in\NN$, we have
    \begin{equation*}
        \mbb E_{\msf B_n}W_1(\msf B_k(x), \msf B_k\circ\msf B_n(x))\leq \|x\|_1\sqrt{\frac{2D(3D-1)}{n}},
    \end{equation*}
    and $\|\mbb E\msf B_k(x) - \mbb E\msf B_k(\msf B_n(x))\|_1\leq \|x\|_1\frac{D(D-1)}{n}$ for independent $\msf B_k$ and $\msf B_n$. Also, there is a random injection $\Phi\colon[n]\to[N]$ coupled to $\msf B_n$ satisfying $\mbb E\|x-\beta(\Phi)\msf B_n(x)\|_2\leq \|x\|_1\sqrt{\frac{2D(3D-1)}{n}}$.
\end{lemma}
\begin{proof}

    Since $(\mc I_n)$ has degree $D$, write $\mc I_n=\bigsqcup_{m=1}^M[n]^{d_m}/H_m$ with $D=\max_md_m$, and write $x=\sum_{\alpha\in\mc I_N}x_{\alpha}e_{\alpha}$ where $(e_{\alpha})$ is the standard basis for $\RR^{\mc I_N}$. After zero-padding $x$, we may assume that $N\geq n$. Let $F_k\colon[N]\to[k]$ and $F_n\colon[N]\to[n]$ be uniformly random maps, so that
    \begin{equation*}
        \msf B_n(x)=\sum_{\alpha\in\mc I_N}x_{\alpha}e_{\theta(F_n)(\alpha)},\quad \msf B_k(x)=\sum_{\alpha\in\mc I_N}x_{\alpha}e_{\theta(F_k)(\alpha)}.
    \end{equation*}
    We proceed to construct a coupling of $\msf B_k\circ\msf B_n(x)$ and of $\msf B_k(x)$ for a fixed realization of $F_n$. We do so by constructing a map $G_{k,n}\colon[n]\to[k]$ depending on $F_k$ and $F_n$ such that the distribution of $G_{k,n}$ given $F_n$ is uniform on the set of all maps from $[n]$ to $[k]$, and such that $G_{k,n}\circ F_n$ is equal to $F_k$ on the most ``important'' indices, which we proceed to formalize. Once we do so, we get $\beta(F_k)x\overset{d}{=}\msf B_k(x)$ and $\beta(G_{k,n}\circ F_n)x\overset{d}{=}\msf B_k\circ\msf B_n(x)$ for any realization of $F_n$, yielding the desired coupling.
    
    If $\alpha=H(i_1,\ldots,i_d)$ for $H\subseteq\mfk S_d$ we write $\mathrm{supp}(\alpha)=\{i_1,\ldots,i_d\}\subseteq[N]$ for the set of distinct indices appearing in $\alpha$. Note that this is well-defined as $H$ acts by permuting coordinates in a tuple, and that each $\alpha\in \mc I_N$ has this form for some $d$ and $H$. We now define the weight associated to $i\in[N]$ by
    \begin{equation}\label{eq:importance_weights}
        w_i = \sum_{\substack{\alpha\in\mc I_N\\ i\in\mathrm{supp}(\alpha)}}|x_{\alpha}|,
    \end{equation}
    so that $\sum_{i=1}^Nw_i\leq D\|x\|_1$ as each $\alpha\in\mc I_N$ contains at most $D$ indices in its support. Using these weights, for each $\ell\in[n]$ we choose a representative $\rho(\ell)\in F_n^{-1}(\ell)$ with the largest weight, breaking ties by choosing the smallest such index. Formally, we set
    \begin{equation*}
        \rho(\ell) = \min\left\{i\in F_n^{-1}(\ell): w_i=\max_{i'\in F_n^{-1}(\ell)}w_{i'}\right\},
    \end{equation*}
    if $F_n^{-1}(\ell)\neq\emptyset$ and set $\rho|_{[n]\setminus F_n([N])}$ to be an arbitrary injective map $[n]\setminus F_n([N])\to [N]\setminus \rho(F_n([N]))$.
    We thus get a random (depending only on $F_n$) injective map $\rho\colon[n]\to[N]$.
    We then define $G_{k,n}\colon[n]\to[k]$ by setting
    \begin{equation}
        G_{k,n}(\ell) = F_k(\rho(\ell)),
    \end{equation}
    so we map $\ell\in[n]$ to the image under $F_k$ of a largest-weight element in the fiber $F_n^{-1}(\ell)$. Note that $G_{k,n}$ is indeed a uniformly random map from $[n]$ to $[k]$ conditioned on $F_n$, since $\rho$ is always injective and since $F_k$ maps distinct inputs to independent and uniform elements in $[k]$. Thus, we have
    \begin{equation}\label{eq:coupling_squared}
        \mbb E_{\msf B_n}W_1(\msf B_k(x),\msf B_k\circ\msf B_n(x))\leq \mbb E\|\beta(F_k)x - \beta(G_{k,n}\circ F_n)x\|_2\leq \sqrt{\mbb E\|\beta(F_k)x-\beta(G_{k,n}\circ F_n)x\|_2^2}.
    \end{equation}

    We turn to analyzing the squared 2-norm of the difference
    \begin{equation*}
        \beta(F_k)x - \beta(G_{k,n}\circ F_n)x = \sum_{\alpha\in\mc I_N}x_{\alpha}\delta_{\alpha},\quad \textrm{where } \delta_{\alpha}=e_{\theta(F_k)(\alpha)} - e_{\theta(G_{k,n}\circ F_n)(\alpha)}.
    \end{equation*}
    Explicitly, if $\alpha=H(i_1,\ldots,i_d)$ then
    \begin{equation*}
        \delta_{\alpha} = e_{H(F_k(i_1),\ldots,F_k(i_d))} - e_{H(F_k\circ\rho\circ F_n(i_1),\ldots, F_k\circ\rho\circ F_n(i_d))}.
    \end{equation*}
    Consider $\langle \delta_{\alpha},\delta_{\beta}\rangle$ for $\alpha,\beta\in\mc I_N$. Cauchy--Schwarz gives $|\langle \delta_{\alpha},\delta_{\beta}\rangle|\leq 2$. We further claim that 
    \begin{equation*}
        \mbb E[\langle \delta_{\alpha},\delta_{\beta}\rangle|F_n]=0\quad \textrm{under the following two conditions on $\alpha,\beta$,}
    \end{equation*} 
    namely,
    \begin{equation}\label{eq:event_A}\tag{A}
        \rho\circ F_n(i)=i \textrm{ for all } i\in\mathrm{supp}(\alpha)\cap\mathrm{supp}(\beta),
    \end{equation}
    and
    \begin{equation}\label{eq:event_B}\tag{B}
        F_n(i)\neq F_n(i') \textrm{ for all distinct } i,i'\in\mathrm{supp}(\alpha)\cup\mathrm{supp}(\beta).
    \end{equation}
    Indeed, if~\eqref{eq:event_A} and~\eqref{eq:event_B} hold and we further condition on the variables $$V=\{F_k(i),F_k(\rho\circ F_n(i)): i\in\mathrm{supp}(\beta)\},$$
    we have
    \begin{equation*}
        \mbb E\Big[\langle \delta_{\alpha},\delta_{\beta}\rangle \Big| F_n, V\Big] = \left\langle \mbb E\big[\delta_{\alpha}|F_n,V\big],\delta_{\beta}\right\rangle,
    \end{equation*}
    because $V$ includes all random indices $F_k(i)$ appearing in $\delta_{\beta}$.
    Furthermore, if $\alpha = H(i_1,\ldots,i_d)$ then 
    \begin{equation*}
        (F_k(i_1),\ldots,F_k(i_d)) \overset{d}{=} (F_k\circ\rho\circ F_n(i_1),\ldots, F_k\circ\rho\circ F_n(i_d))\quad \textrm{conditioned on } V.
    \end{equation*}
    That is because either $\rho\circ F_n(i_j)=i_j$, or $\rho\circ F_n(i_j)\neq i_j$ in which case both $F_k(i_j)$ and $F_k\circ\rho\circ F_n(i_j)$ are not in $V$ (since $i_j\notin\mathrm{supp}(\beta)$ by~\eqref{eq:event_A}), so both are uniformly distributed over $[k]$ and independent for different $i_j$ (because $F_n(i_j)$ are all distinct by~\eqref{eq:event_B}).
    Thus, we conclude that $\mbb E[\delta_{\alpha}|F_n,V]=0$ and hence $\mbb E[\langle \delta_{\alpha},\delta_{\beta}\rangle|F_n]=0$ if the events~\eqref{eq:event_A} and~\eqref{eq:event_B} hold. 

    The above argument shows that
    \begin{equation}\label{eq:intermediate_bd}\begin{aligned}
        |\mbb E\langle \delta_{\alpha},\delta_{\beta}\rangle|&\leq 2\mbb P[\textrm{not~\eqref{eq:event_A} or not~\eqref{eq:event_B}}] \leq 2\Big(\mbb P[\rho\circ F_n(i)\neq i \textrm{ for some } i\in\mathrm{supp}(\alpha)\cap\mathrm{supp}(\beta)]\\ &\quad+ \mbb P[F_n(i) = F_n(i') \textrm{ for some } i\neq i'\in\mathrm{supp}(\alpha)\cup\mathrm{supp}(\beta)]\Big).
    \end{aligned}\end{equation}
    We proceed to bound each probability separately.
    First, we have
    \begin{equation*}\begin{aligned}
        &\mbb P[\rho\circ F_n(i)\neq i \textrm{ for some } i\in\mathrm{supp}(\alpha)\cap\mathrm{supp}(\beta)] \leq \sum_{i\in\mathrm{supp}(\alpha)\cap\mathrm{supp}(\beta)}\mbb P[\rho\circ F_n(i)\neq i]\\
        &\leq \sum_{i\in\mathrm{supp}(\alpha)\cap\mathrm{supp}(\beta)}\mbb P[\textrm{there is } j\in [N]\setminus\{i\} \textrm{ s.t.\ } w_j\geq w_i \textrm{ and } F_n(j)=F_n(i)].
    \end{aligned}\end{equation*}
    Let $S_i=\{j\in[N]\setminus\{i\}: w_j\geq w_i\}$ and note that $|S_i|\leq \frac{D\|x\|_1}{w_i}$. Indeed, note that
    \begin{equation*}
        w_i|S_i|\leq \sum_{j\in S_i}w_j \leq \sum_{j\in[N]}w_j\leq D\|x\|_1.
    \end{equation*}
    Therefore,
    \begin{equation*}
        \mbb P[\textrm{there is } j\in [N]\setminus\{i\} \textrm{ s.t.\ } w_j\geq w_i \textrm{ and } F_n(j)=F_n(i)] \leq \sum_{j\in S_i}\mbb P[F_n(j)=F_n(i)] \leq \frac{|S_i|}{n} \leq \frac{D\|x\|_1}{w_in},
    \end{equation*}
    since $F_n$ is uniformly random. Thus, we have the following bound on the first probability in~\eqref{eq:intermediate_bd}
    \begin{equation*}
        \mbb P[\rho\circ F_n(i)\neq i \textrm{ for some } i\in\mathrm{supp}(\alpha)\cap\mathrm{supp}(\beta)] \leq \frac{D\|x\|_1}{n}\sum_{i\in\mathrm{supp}(\alpha)\cap\mathrm{supp}(\beta)}\frac{1}{w_i}.
    \end{equation*}
    For the second probability in~\eqref{eq:intermediate_bd}, note that $|\mathrm{supp}(\alpha)\cup\mathrm{supp}(\beta)|\leq 2D$, so
    \begin{equation*}
        \mbb P[F_n(i) = F_n(i') \textrm{ for some } i\neq i'\in\mathrm{supp}(\alpha)\cup\mathrm{supp}(\beta)]\leq \frac{\binom{2D}{2}}{n} = \frac{D(2D-1)}{n}.
    \end{equation*}
    
    Using the above bounds, we have
    \begin{equation}\label{eq:almost_final_bd}\begin{aligned}
        &\mbb E\|\beta(F_k)x-\beta(G_{k,n}\circ F_n)x\|_2^2 \leq \sum_{\alpha,\beta\in\mc I_N}|x_{\alpha}x_{\beta}|\cdot|\mbb E\langle\delta_{\alpha},\delta_{\beta}\rangle|\\&\leq \frac{2D\|x\|_1}{n}\left(\sum_{\alpha,\beta\in\mc I_N}|x_{\alpha}x_{\beta}|\sum_{i\in\mathrm{supp}(\alpha)\cap\mathrm{supp}(\beta)}\frac{1}{w_i}\right) + \frac{2D(2D-1)}{n}\sum_{\alpha,\beta\in\mc I_N}|x_{\alpha}x_{\beta}|.
    \end{aligned}\end{equation}
    The second term in~\eqref{eq:almost_final_bd} is simply $\frac{2D(2D-1)}{n}\|x\|_1^2$. For the first term, we interchange the sums to obtain
    \begin{equation*}\begin{aligned}
        &\sum_{\alpha,\beta\in\mc I_N}|x_{\alpha}x_{\beta}|\sum_{i\in\mathrm{supp}(\alpha)\cap\mathrm{supp}(\beta)}\frac{1}{w_i} = \sum_{i=1}^N\frac{1}{w_i}\sum_{\substack{\alpha,\beta\in\mc I_N\\ i\in\mathrm{supp}(\alpha)\cap\mathrm{supp}(\beta)}}|x_{\alpha}x_{\beta}| = \sum_{i=1}^N\frac{1}{w_i}\left(\sum_{\substack{\alpha\in\mc I_N\\ i\in\mathrm{supp}(\alpha)}}|x_{\alpha}|\right)^2\\ &= \sum_{i=1}^Nw_i\leq D\|x\|_1,
    \end{aligned}\end{equation*}
    by definition of $w_i$ in~\eqref{eq:importance_weights}. Putting everything together, we get
    \begin{equation*}
        \mbb E\|\beta(F_k)x-\beta(G_{k,n}\circ F_n)x\|_2^2\leq \frac{2D^2\|x\|_1^2}{n}+\frac{2D(2D-1)\|x\|_1^2}{n} = \frac{2D(3D-1)\|x\|_1^2}{n},
    \end{equation*}
    giving the first claimed bound by~\eqref{eq:coupling_squared}.

    For the second claimed bound, write $$\|\mbb E\msf B_k(x) - \mbb E\msf B_k(\msf B_n(x))\|_1\leq \sum_{\alpha\in\mc I_N}|x_{\alpha}|\|\mbb Ee_{\theta(F_k)(\alpha)} - \mbb Ee_{\theta(F_{k,n}\circ F_n)(\alpha)}\|_1,$$ where $F_{k,n}\colon[n]\to[k]$ is independent of $F_n$. Note that $F_n$ is injective on $\mathrm{supp}(\alpha)$ with probability at least $1-\frac{D(D-1)}{2n}$, an event we shall call $\mathrm{C}$, and conditioning on this event we have $\theta(F_k)(\alpha)\overset{d}{=}\theta(F_{k,n}\circ F_n)(\alpha)$ so $\mbb Ee_{\theta(F_k)(\alpha)} = \mbb E[e_{\theta(F_{k,n}\circ F_n)(\alpha)} | \mathrm{C}]$. Thus, we have $$\|\mbb Ee_{\theta(F_k)(\alpha)} - \mbb Ee_{\theta(F_{k,n}\circ F_n)(\alpha)}\|_1 = \left\| (\mbb P[\mathrm{C}]-1)\mbb Ee_{\theta(F_k)(\alpha)} + \mbb P[\textrm{not } \mathrm{C}]\mbb E[e_{\theta(F_{k,n}\circ F_n)(\alpha)} | \textrm{not } \mathrm{C}]\right\|_1\leq \frac{D(D-1)}{n}.$$
    Combining the two displayed bounds above, we obtain the second claim.

    For the third claimed bound, we consider a large $k$ in the above coupling. Specifically, for $k\geq N$ the map $F_k\colon[N]\to[k]$ is injective with probability at least $1-\frac{N(N-1)}{2k}$, in which case 
    \begin{equation*}
        \|\beta(F_k)x-\beta(G_{k,n}\circ F_n)x\|_2 = \|\beta(F_k)(x-\beta(\rho\circ F_n)x)\|_2 = \|x-\beta(\rho\circ F_n)x\|_2.
    \end{equation*}
    Thus, we have
    \begin{equation*}\begin{aligned}
        &\left(1-\frac{N(N-1)}{2k}\right)\mbb E\|x-\beta(\rho\circ F_n)x\|_2 \leq \mbb E\Big[\|\beta(F_k)x-\beta(G_{k,n}\circ F_n)x\|_2\Big| F_k \textrm{ injective}\Big]\mbb P[F_k \textrm{ injective}]\\&\leq \mbb E\|\beta(F_k)x-\beta(G_{k,n}\circ F_n)x\|_2 \leq \|x\|_1\sqrt{\frac{2D(3D-1)}{n}}.
    \end{aligned}\end{equation*}
    Taking $k\to\infty$ and recalling that $\rho$ is injective, we obtain the third claimed bound.
\end{proof}
We are ready to prove Theorem~\ref{thm:W1_rates_FS}.
\begin{proof}[Proof (Theorem~\ref{thm:W1_rates_FS}).]
    For the first part, the expectation bounds follow from Lemma~\ref{lem:W1_rate_finite_x} and the fact that $\sup_n\sup_{x\in\Omega_n}\|x\|_1\leq r$. The claimed subgaussianity follows from the bounded-difference inequality. Indeed, by Lemma~\ref{lem:empirical_species_equivalent} we can replace $\msf E_n(x)$ by $\mscr E_n(x)$ without changing the distributions in question, and observe that $\mscr E_n(x)$ is a function of $n$ iid indices sampled from $x$ by construction in Section~\ref{sec:species}, and changing any one of them affects two entries in $\mscr E_n(x)$ by at most $r/n$. 

    For the second part, the bound~\eqref{eq:expected_binning_rate} follows from Lemma~\ref{lem:binning_rate}, and the bound on $\mbb Ed_{\mathrm{samp}}(x,\msf B_n(x))$ then follows by the definition~\eqref{eq:sampling_metric} of the sampling metric.
\end{proof}

We turn to proving Theorem~\ref{thm:eqp_consistent_laws} by combining Proposition~\ref{prop:fd_measures_as_limits} and a dual de Finetti theorem from~\cite{levin2025deFin}. 
\begin{proof}[Proof (Theorem~\ref{thm:eqp_consistent_laws}).]
    Propositions~\ref{prop:symmetry_coFS} and~\ref{prop:relations_coFS} show that if $\mu_n=\mathrm{Law}(\msf B_n(X))$ for random $X\in\overline{\Omega}_{\infty}$ independent of $\msf B_n$, then $(\mu_n)$ is equipartition-consistent. Conversely, if $(\mu_n)$ is equipartition-consistent then by~\cite[Thm.~4.20]{levin2025deFin} there is a sequence $(\nu_i\in\mc P(\Omega_{\infty}))$ such that $\mathrm{Law}(\msf B_n(X_i))\xrightarrow{i\to\infty} \mu_n$ weakly for each $n$, where $X_i\sim\nu_i$ is independent of $\msf B_n$. This implies that $(\nu_i)$ converges weakly with respect to $d_{\mathrm{samp}}$ to some $\mu_{\infty}\in\mc P(\overline{\Omega}_{\infty})$ satisfying the claimed sampling representation. The uniqueness of such $\mu_{\infty}$ and the extremality of sequences of the form $(\mathrm{Law}(\msf B_n(x)))$ for $x\in\overline{\Omega}_{\infty}$ both follow from Proposition~\ref{prop:fd_measures_as_limits}, which applies by Theorem~\ref{thm:W1_rates_FS}. 
    Specifically, Theorem~\ref{thm:W1_rates_FS} shows that $\mbb Ed_{\mathrm{samp}}(x,\msf B_n(x))\to0$ as $n\to\infty$ at a universal rate for all $x\in\Omega_{\infty}$. If $x\in\overline{\Omega}_{\infty}$ then we can write $x=\lim_ix_i$ for $x_i\in\Omega_{\infty}$, so $\msf B_n(x)$ is the weak limit of $\msf B_n(x_i)$. The triangle inequality shows that $\mbb Ed_{\mathrm{samp}}(x,\msf B_n(x))\to0$ at the same universal rate, so the condition~\eqref{eq:convergence_of_samples} in Proposition~\ref{prop:fd_measures_as_limits} is satisfied.
\end{proof}

Next, we use Theorem~\ref{thm:W1_rates_FS} to prove the sketching and generalization rates in Theorem~\ref{thm:binning_rates_funcs}. The following proof is similar to the proof of Theorem~\ref{thm:sampling_rates_funcs}.
\begin{proof}[Proof (Theorem~\ref{thm:binning_rates_funcs}).]
    First suppose $f$ is $L$-Lipschitz with respect to $d_{\mathrm{samp}}$. Then 
    \begin{equation}\label{eq:comp1}
        |f(x)-f(\msf E_n(x))|\leq Ld_{\mathrm{samp}}(x,\msf E_n(x))\leq \frac{Lr(1+2\epsilon)}{\sqrt{n}},    
    \end{equation}
    with probability at least $1-e^{-2\epsilon^2}$ by Theorem~\ref{thm:W1_rates_FS}(1). The same theorem gives 
    \begin{equation}\label{eq:comp2}
        |f(x)-\mbb Ef(\msf E_n(x))|\leq L\mbb Ed_{\mathrm{samp}}(x,\msf E_n(x))\leq \frac{Lr}{\sqrt{n}}.
    \end{equation}
    Second, if $f$ is $L$-Lipschitz with respect to $(x,y)\mapsto W_1(\msf B_k(x),\msf B_k(y))$, then Theorem~\ref{thm:W1_rates_FS}(1) gives the bounds~\eqref{eq:comp1} and~\eqref{eq:comp2} using the same arguments.
    
    Now suppose $f$ is unchanged by zero-padding and that each restriction $f|_{\vct V_n}$ is $L$-Lipschitz in 2-norm. Then by Lemma~\ref{lem:empirical_species_equivalent} we have
    \begin{equation*}
        |f(x)-f(\msf E_n(x))| = |f(x)-f(\mscr E_n(x))|\leq L\|x-\mscr E_n(x)\|_2\leq \frac{Lr(1+2\epsilon)}{\sqrt{n}},
    \end{equation*}
    with probability at least $1-e^{-2\epsilon^2}$, where the last inequality follows from~\eqref{eq:Binom_bound} and the fact that $\|x-\mscr E_n(x)\|_2$ is $\frac{4r^2}{n}$-subgaussian. Likewise, it follows from~\eqref{eq:Binom_bound} that $|f(x)-\mbb Ef(\msf E_n(x))|\leq \frac{Lr}{\sqrt{n}}$. 

    We have thus proved the claimed sketching rates under all the three conditions in Theorem~\ref{thm:binning_rates_funcs}. To obtain the generalization rates, let $f,\widehat f$ satisfy either of the three conditions in the theorem (possibly different ones) and fix any $x\in\bigsqcup_n\Omega_n$. If $m=\lfloor n/D\rfloor$, then $\msf E_m(x)\in\Omega_{mD}\subseteq\Omega_{\leq n}$, and hence
    \begin{equation*}
        |f(x)-\widehat f(x)|\leq |f(x)-\mbb Ef(\msf E_m(x))| + |\mbb Ef(\msf E_m(x))-\mbb E\widehat f(\msf E_m(x))| + |\mbb E\widehat f(\msf E_m(x))-\widehat f(x)|\leq \mathrm{e}_{n} + 2Lr\sqrt{\frac{1}{\lfloor n/D\rfloor}}.
    \end{equation*}
    Finally, suppose $(\mu_n)$ is equipartition-consistent. By Theorem~\ref{thm:eqp_consistent_laws}, there is a measure $\mu\in\mc P(\overline{\Omega}_{\infty})$ such that $\mu_n=\mathrm{Law}(\msf B_n(X))$ for $X\sim\mu$ independent of $\msf B_n$. Arguing as above, we have 
    \begin{equation*}\begin{aligned}
        \mathrm{e}_{\mu_{\infty}} &= \mbb E|f(X)-\widehat f(X)|\leq \mbb E_X|f(X)-\mbb E_{\msf B_n}f(\msf B_n(X))| + \mbb E|f(\msf B_n(X))-\widehat f(\msf B_n(X))|\\ &\quad + \mbb E_X|\mbb E_{\msf B_n}\widehat f(\msf B_n(X))-\widehat f(X)| \leq \mathrm{e}_{\mu_n} + 2Lr\sqrt{\frac{2D(3D-1)}{n}},
    \end{aligned}\end{equation*}
    using Lemma~\ref{lem:binning_rate} (together with a limiting argument to handle limit objects, like in the proof of Theorem~\ref{thm:eqp_consistent_laws} above). This is the last claimed bound.
\end{proof}
We proceed to prove Corollary~\ref{cor:FS_polynomials} concerning polynomials unchanged by zero-padding.
\begin{proof}[Proof (Corollary~\ref{cor:FS_polynomials}).]
    It was shown in~\cite[Sec.~5]{levin2025deFin} that there exists $g_k\in\RR[\vct V_k]$ of degree $d$ satisfying $p(x)=\mbb Eg_k(\msf B_k(x))$ for all $x$ with $k=dD$. If $L$ is the Lipschitz constant of $g_k$ over $\Omega_k$, then $p$ is $L$-Lipschitz in $(x,y)\mapsto W_1(\msf B_k(x),\msf B_k(y))$, as claimed. For the strengthened generalization bound, fix $x\in \Omega_N\subseteq \RR^{\mc I_N}$ for some $N\in\NN$, and form the vector $\mathrm{Sym}^{\leq d}x$ containing all monomials of degree at most $d$ in the entries of $x$. If we define $\mc J_n=\bigsqcup_{i=0}^d\mc I_n^i$ to be the collection of tuples of elements of $\mc I_n$ of size at most $d$, note that $(\mc J_n)$ is a compatible sequence of degree $dD$ and $(x^{\otimes i})_{i=0}^d\in \RR^{\mc J_N}$. Writing $g_k(x)=\langle c_{g_k}, (x^{\otimes i})_{i=0}^d\rangle$ and noting that $\msf B_n(x^{\otimes i})\overset{d}{=}(\msf B_n(x))^{\otimes i}$, we have
    \begin{equation}\begin{aligned}
        |p(x)-\mbb Ep(\msf B_n(x))| &= |\mbb Eg_k(\msf B_k(x))-\mbb Eg_k(\msf B_k(\msf B_n(x)))| \leq \|c_{g_k}\|_{\infty}\|\mbb E \msf B_k((x^{\otimes i})_{i=0}^d) - \mbb E\msf B_k(\msf B_n((x^{\otimes i})_{i=0}^d))\|_1\\&\leq \|c_{g_k}\|_{\infty}\frac{dD(dD-1)}{n}\|(x^{\otimes i})_{i=0}^d\|_1\leq \|c_{g_k}\|_{\infty}\frac{dD(dD-1)}{n}(d+1)\max\{r^d,1\},
    \end{aligned}\end{equation}
    where the second inequality follows by the last bound in Theorem~\ref{thm:W1_rates_FS}(2) and the last inequality follows from the bound $\|(x^{\otimes i})_{i=0}^d\|_1 = \sum_{i=0}^d\|x\|_1^i \leq (d+1)\max\{1,r^d\}$. Similarly writing $\widehat p(x)=\mbb E\widehat g_k(\msf B_k(x))$ and letting $L'=\|c_{g_k}\|_{\infty} + \|c_{\widehat g_k}\|_{\infty}$, the result follows as in the proof of Theorem~\ref{thm:binning_rates_funcs}.
\end{proof}
We now apply Theorem~\ref{thm:binning_rates_funcs} to analyze the DeepSets-like functions from Corollary~\ref{cor:deepsets}.
\begin{proof}[Proof (Corollary~\ref{cor:deepsets}).]
    The first claim follows from the observation that if $x,y\in\vct V_n$ then
    \begin{equation*}
        |f(x)-f(y)|\leq L_{\sigma}\sum_i\|\rho(x_i)-\rho(y_i)\|\leq L_{\sigma}L_{\rho}\sum_i\|x_i-y_i\|_2\sqrt{\|x_i\|_1+\|y_i\|_1}\leq L_{\sigma}L_{\rho}\sqrt{2r}\|x-y\|_2.
    \end{equation*}
    For the second claim, consider the sequence $x^{(n)}=(\frac{1}{2n}e_1,\ldots,\frac{1}{2n}e_1,-\frac{1}{2n}e_1,\ldots,-\frac{1}{2n}e_1)\in\vct V_{2n}$, where both entries are repeated $n$ times. Note that $x^{(n)}\to0$ in sampling metric. Indeed, we have $\msf B_k(x^{(n)})\overset{d}{=}\frac{1}{2n}\msf B_k(e_1\mathbbm{1}_n^\top) - \frac{1}{2n}\msf B_k'(e_1\mathbbm{1}_n^\top)$ where $\msf B_k,\msf B_k'$ are independent. Since $\msf B_k(\frac{1}{n}e_1\mathbbm{1}_n^\top)=(\frac{N_1}{n}e_1,\ldots,\frac{N_k}{n}e_1)$ where $(N_1,\ldots,N_k)\sim\mathrm{Multinom}(n,k,\mathbbm{1}_k/k)$, we get $\mbb E\|\msf B_k(x^{(n)})\|_2^2 = \frac{1}{2}\sum_{i=1}^k\mathrm{Var}(N_i/n) = \frac{1-1/k}{2n}$, proving that $\msf B_k(x^{(n)})\to0$ weakly for each $k$.
    However, we have $f(x^{(n)})\equiv 1$ for all $n$, proving that $f$ is discontinuous in sampling metric.
\end{proof}

We turn to proving Corollary~\ref{cor:gnns} analyzing certain graph neural networks.
\begin{proof}[Proof (Corollary~\ref{cor:gnns}).]
    It is easy to verify that $f$ is unchanged by zero-padding. If $\|A\|_1+\|X\|_1\leq r$ and similarly for $(B,Y)$, we have
    \begin{equation*}\begin{aligned}
        &\|F(A,X)-F(B,Y)\|_2\leq \|A-B\|_2 + L_{\sigma}\left\|\sum_{d=0}^DA^dX\Theta_d - \sum_{d=0}^DB^dY\Theta_d\right\|_2\\&\leq \|A-B\|_2 + L_{\sigma}\sum_{d=0}^D\Big(\|A^dX\Theta_d-A^dY\Theta_d\|_2 + \|A^dY\Theta_d-B^dY\Theta_d\|_2\Big)\\
        &\leq \|A-B\|_2 + L_{\sigma}\sum_{d=0}^D\|A\|_2^d\|\Theta_d\|_{\mathrm{op}}\|X-Y\|_2 + \sum_{d=0}^D\|A^d-B^d\|_{\mathrm{op}}\|Y\|_2\|\Theta_d\|_{\mathrm{op}}\\
        &\leq \left(L_{\sigma}\sum_{d=0}^Dr^d\|\Theta_d\|_{\mathrm{op}}\right)\|X-Y\|_2 + \left(1+L_{\sigma}\sum_{d=0}^Ddr^d\|\Theta_d\|_{\mathrm{op}}\right)\|A-B\|_2,\\
        &\leq \underbrace{\left(1+L_{\sigma}\sum_{d=0}^Dr^d\|\Theta_d\|_{\mathrm{op}} + L_{\sigma}\sum_{d=0}^Ddr^d\|\Theta_d\|_{\mathrm{op}}\right)}_{=L(r)}\|(A,X)-(B,Y)\|_2.
    \end{aligned}\end{equation*}
    Finally, note that if $\|A\|_1+\|X\|_1\leq r$, then 
    \begin{equation*}\begin{aligned}
        \|F(A,X)\|_1\leq \|A\|_1+L_{\sigma}\sqrt{d}\left(\|\Theta_0\|_1\|X\|_1+\sum_{d=1}^D\|\Theta_d\|_1\|A^d\|_1\|X\|_1\right)&\leq r+L_{\sigma}\sqrt{d}\sum_{d=0}^Dr^{d+1}\|\Theta_d\|_1.
    \end{aligned}\end{equation*}
    Denoting the last bound above by $R(r)$, after $j$ applications of $F(A,X)$ we have $\|F^{\circ j}(A,X)\|_1\leq R^{\circ j}(r)$, and hence is Lipschitz with constant $\prod_{j=1}^{\ell}L(R^{\circ (j-1)}(r))$.
    Since $f=P\circ F^{\circ\ell}$ and $P$ is $L_{\rho}\sqrt{2R^{\circ \ell}(r)}$-Lipschitz in 2-norm on the 1-norm ball of radius $R^{\circ \ell}(r)$ by Corollary~\ref{cor:deepsets}(1), we obtain the claim.
\end{proof}

Finally, we prove Proposition~\ref{prop:relation_to_species} showing that convergence of species samples is equivalent to joint convergence of the random binnings of positive and negative parts. To this end, we shall need the following two lemmas. The first lemma is an analog of the first bound in Theorem~\ref{thm:W1_rates_FS}(1) with the roles of random binning and species sampling reversed.
\begin{lemma}\label{lem:species_instead_of_binning}
    Let $(\mc I_n)$ be a compatible sequence of degree $D$ and let $x\in\RR^{\mc I_N}$ with $x\geq0$. Then
    \begin{equation*}
        \mbb E_{\msf B_n}W_1(\msf E_k(x),\msf E_k\circ\msf B_n(x))\leq \|x\|_1\frac{kD(kD-1)}{n}.
    \end{equation*}
\end{lemma}
\begin{proof}
    By homogeneity of $\msf E_k$ and $\msf E_k\circ\msf B_n$, it suffices to prove the claim for $x\in\Delta^{\mc I_N}$. After zero-padding $x$ if needed, we may assume that $N\geq kD$. 
    
    Following the construction in Section~\ref{sec:species}, sample $\alpha_1,\ldots,\alpha_k\in\mc I_N$ from $x$, define $\mscr E_k(x)=\frac{1}{k}\sum_{i=1}^ke_{\alpha_i}$, and let $t_1,\ldots,t_{\ell}$ be a random enumeration of $S=\bigcup_{i=1}^k\mathrm{supp}(\alpha_i)\subseteq[N]$, where $\ell\leq kD$ and if $\alpha=H(j_1,\ldots,j_d)$ for some $H\subseteq\mfk S_d$ we denote $\mathrm{supp}(\alpha)=\{j_1,\ldots,j_d\}$. 
    Define
    \begin{equation*}
        (X_k)_{H(i_1,\ldots,i_d)} = \mscr E_k(x)_{H(t_{i_1},\ldots,t_{i_d})} = \frac{1}{k}\sum_{i=1}^k\mathbbm{1}[\alpha_i=H(t_{i_1},\ldots,t_{i_d})],    
    \end{equation*}
    for each $H(i_1,\ldots,i_d)\in \mc I_{kD}$, which satisfies $X_k\overset{d}{=}\msf E_k(x)$.

    We now construct a random $Y_k$ coupled to $X_k$ with $Y_k\overset{d}{=}\msf E_k\circ\msf B_n(x)$. Fix a realization $F_{n,N}\colon[N]\to[n]$ of a uniformly random map, so $\msf B_n(x)\overset{d}{=}\beta(F_{n,N})x$.
    Observe that $\theta(F_{n,N})(\alpha_1),\ldots,\theta(F_{n,N})(\alpha_k)$ are iid samples from $\beta(F_{n,N})x$, so $\beta(F_{n,N})\mscr E_k(x)\overset{d}{=}\mscr E_k(\beta(F_{n,N})x)$. Furthermore, we have $\bigcup_{i=1}^k\mathrm{supp}(\theta(F_{n,N})(\alpha_i))=F_{n,N}(S)$. Therefore, if $F_{n,N}$ is injective on the random and independent set $S$, then a uniformly random enumeration of $\bigcup_{i=1}^k\mathrm{supp}(\theta(F_{n,N})(\alpha_i))=F_{n,N}(S)$ is given by $F_{n,N}(t_1),\ldots,F_{n,N}(t_{\ell})$ and we set
    \begin{equation}\label{eq:Yk_F_injective}
        (Y_k)_{H(i_1,\ldots,i_d)} = [\beta(F_{n,N})\mscr E_k(x)]_{H(F_{n,N}(t_{i_1}),\ldots,F_{n,N}(t_{i_d}))} = \frac{1}{k}\sum_{i=1}^k\mathbbm{1}[\theta(F_{n,N})(\alpha_i)=H(F_{n,N}(t_{i_1}),\ldots,F_{n,N}(t_{i_d}))],
    \end{equation}
    for each $H(i_1,\ldots,i_d)\in \mc I_{kD}$. If $F_{n,N}$ is not injective on $S$, then define $Y_k$ similarly with an arbitrary random enumeration of $F_{n,N}(S)$ that is independent of $F_{n,N}$. We therefore have $Y_k\overset{d}{=}\msf E_k(\beta(F_{n,N})x)$.

    If $F_{n,N}$ is injective on $S$, we have $X_k=Y_k$ by~\eqref{eq:Yk_F_injective} because $\alpha_i = H(t_{i_1},\ldots,t_{i_d})$ if and only if 
    $$\theta(F_{n,N})(\alpha_i)=\theta(F_{n,N})H(t_{i_1},\ldots,t_{i_d}) = H(F_{n,N}(t_{i_1}),\ldots,F_{n,N}(t_{i_d})).$$ 
    If $F_{n,N}$ is not injective on $S$, we have $\|X_k-Y_k\|_2\leq \|X_k\|_1+\|Y_k\|_1 = 2$. 
    We conclude that
    \begin{equation*}
        \mbb E_{\msf B_n}W_1(\msf E_k(x),\msf E_k\circ\msf B_n(x))\leq 2\mbb P[F_{n,N} \textrm{ not injective on } S] \leq 2\frac{\binom{kD}{2}}{n} = \frac{kD(kD-1)}{n},
    \end{equation*}
    as claimed.
\end{proof}
We proceed to relate convergence with respect to species sampling and random binning.
\begin{lemma}\label{lem:relation_to_species_part}
    Let $(\mc I_n)$ be a compatible sequence of finite degree, set $\vct V_n=\RR^{\mc I_n}$, and let $(x_i)\subseteq \bigsqcup_n\vct V_n$. 
    \begin{enumerate}[labelwidth=!, labelindent=0pt]
        \item If $(\msf E_k(x_i))$ converges weakly for each $k$, then $(\msf B_k(x_i))$ converges weakly for each $k$. 
        \item If $(\msf B_k(x_i))$ converges weakly for each $k$ and $x_i\geq0$ for all $i$, then $(\msf E_k(x_i))$ converges weakly for each $k$.
    \end{enumerate}
\end{lemma}
\begin{proof}
    For the first claim, suppose $(\msf E_n(x_i))_i$ converges weakly for each $n$. In particular, the sequence of norms $\|x_i\|_1=\|\msf E_1(x_i)\|$ converges, so we can find $r>0$ such that $\sup_i\|x_i\|_1\leq r$. Choose $\epsilon>0$ and $n\geq 2(r/\epsilon)^2$. By Skorokhod's representation theorem, we can find a coupling $(X_{i,n})_i$ of $(\msf E_n(x_i))_i$ converging almost surely. For any $k,i,j\in\NN$, we then have
    \begin{equation*}\begin{aligned}
        &W_1(\msf B_k(x_i),\msf B_k(x_j))\leq \mbb E_{\msf E_n}W_1(\msf B_k(x_i), \msf B_k\circ\msf E_n(x_i)) + \mbb E_{X_{i,n},X_{j,n}}W_1(\msf B_k(X_{i,n}),\msf B_k(X_{j,n}))\\ &\quad + \mbb E_{\msf E_n}W_1(\msf B_k\circ\msf E_n(x_j),\msf B_k(x_j)) \leq 2\epsilon + \mbb E\|\msf B_k(X_{i,n}-X_{j,n})\|_2,
    \end{aligned}\end{equation*}
    by Theorem~\ref{thm:W1_rates_FS}(1) and our choice of $n$, where in the last line $\msf B_k$ is independent of $X_{i,n},X_{j,n}$. Noting that $\|\msf B_k(X_{i,n}-X_{j,n})\|_2\leq \|X_{i,n}-X_{j,n}\|_1\to0$ almost surely, we have $\limsup_{i,j\to\infty}W_1(\msf B_k(x_i),\msf B_k(x_j))\leq 2\epsilon$ for any $\epsilon>0$, and hence that $(\msf B_k(x_i))$ converges for each $k$. This proves the first claim.
    
    Now interchange the roles of species sampling and random binning, using Lemma~\ref{lem:species_instead_of_binning} instead of Theorem~\ref{thm:W1_rates_FS}(1). We conclude that if $(\msf B_n(x_i))_i$ converges for each $n$, then $$W_1(\msf E_k(x_i),\msf E_k(x_j))\leq 2\epsilon + \mbb E_{X_{i,n},X_{j,n}}W_1(\msf E_k(X_{i,n}),\msf E_k(X_{j,n})),$$ 
    where $(X_{i,n})_i$ is an almost-surely convergent coupling of $(\msf B_n(x_i))$. Since $X_{i,n}\geq0$, either $X_{i,n}\to0$ in which case $W_1(\msf E_k(X_{i,n}),\msf E_k(X_{j,n}))\to0$, or $\lim_i\|X_{i,n}\|_1>0$, in which case
    \begin{equation*}\begin{aligned}
        W_1(\msf E_k(X_{i,n}),\msf E_k(X_{j,n})) \leq &\|X_{i,n}\|_1W_1(\msf E_k(X_{i,n}/\|X_{i,n}\|_1),\msf E_k(X_{j,n}/\|X_{j,n}\|_1)) + \Big| \|X_{i,n}\|_1 - \|X_{j,n}\|_1\Big|.
    \end{aligned}\end{equation*}
    Both terms converge to zero almost surely, the second because $(\|X_{i,n}\|_1=\|\msf B_1(x_i)\|_1)_i$ converges, and the first because $W_1(\msf E_k(x),\msf E_k(y))\leq k\|x-y\|_1$ for $x,y\in\Delta^{\mc I_n}$. This can be seen by noting that $\frac{1}{2}\|x-y\|_1$ is the total variation distance between the distributions defined by $x$ and $y$ on $\mc I_n$, so there is a coupling between $k$ iid samples from these distributions that are equal with probability at least $1-\frac{k}{2}\|x-y\|_1$. Applying this fact to the construction of $\msf E_k(x)$ and $\msf E_k(y)$ from Section~\ref{sec:species} yields the claimed bound.
\end{proof}
We are ready to prove Proposition~\ref{prop:relation_to_species}.
\begin{proof}[Proof (Proposition~\ref{prop:relation_to_species}).]
    Define a new compatible sequence $(\mc J_n=\mc I_n\sqcup\mc I_n)$, consisting of two disjoint copies of the original one. Note that its degree is also $D$. 
    We have maps $S\colon\RR^{\mc I_n}\to \RR^{\mc J_n}_+$ into nonnegative vectors sending $x\mapsto (x^+,x^-)$, and $Q\colon \RR^{\mc J_n}_+\to \RR^{\mc I_n}$ sending $(y,z)\mapsto y-z$. Observe that $\msf E_k(x)\overset{d}{=}Q\circ\msf E_k\circ S(x)$ for any $x\in\RR^{\mc I_N}$ and any $N$ by construction in Section~\ref{sec:species}. Indeed, the empirical sample $\mscr E_k(x)=\|x\|_1\mathrm{sign}(x)\odot\mscr E_k(|x|/\|x\|_1)$ can equivalently be formed by sampling $k$ iid elements $(\alpha_1,s_1),\ldots,(\alpha_k,s_k)$ from the distribution $(x^+,x^-)/\|x\|_1$ on $\mc J_N$, where $s_i=1$ if $\alpha_i$ is sampled from $x^+$ and $s_i=-1$ if it is sampled from $x^-$. We then have $Q\circ\mscr E_k\circ S(x)\overset{d}{=}\frac{\|x\|_1}{k}\sum_{i=1}^ks_ie_{\alpha_i} \overset{d}{=} \mscr E_k(x)$. Relabelling the sampled $\alpha_i$ randomly as in~\eqref{eq:species_sampling_general} shows that the same identity holds for $\msf E_k$ instead of $\mscr E_k$.
    
    We conclude that $(\msf E_k(x_i))_i$ converge weakly if and only if $(Q\circ\msf E_k\circ S(x_i))_i$ converges weakly. Since $\msf E_k\circ S(x_i) = (Y_i,Z_i)$ and the supports of $Y_i$ and $Z_i$ are disjoint almost surely, we have $\msf E_k\circ S(x_i)=S\circ Q\circ \msf E_k\circ S(x_i)$ for all $i$, so $(Q\circ\msf E_k\circ S(x_i))_i$ converges weakly if and only if $(\msf E_k\circ S(x_i))_i$ converges weakly. In turn, since $S(x_i)\geq 0$, we conclude by Lemma~\ref{lem:relation_to_species_part} that $(\msf E_k\circ S(x_i))_i$ converges weakly for all $k$ if and only if $(\msf B_k\circ S(x_i))_i$ converges weakly for all $k$. Finally, observe that $\msf B_k(x_i^+,x_i^-)\overset{d}{=}(\msf B_k(x_i^+),\msf B_k(x_i^-))$ where we apply the same random map $\msf B_k$ to both $x_i^+$ and $x_i^-$ by construction of random binning in~\eqref{eq:binning_general} (see also Section~\ref{sec:sampling_and_binning}). 
\end{proof}

\section{Conclusions}\label{sec:conclusion}
We have considered the closely related problems of generalization and sketching of any-dimensional functions.
To tackle these problems, we compare objects of different sizes by comparing distributions of their random samples. By using the right sampling maps, depending on the application domain and the relations between inputs of different sizes there, we obtain rich families of compact sets on which we get uniform rates for generalization and sketching, and a correspondence between limit objects and any-dimensional data distributions. 
Focusing on specific generalizations of sampling with replacement, random binning, and species sampling, we then obtain precise quantitative rates for approximation, sketching, and generalization. Some of the function classes to which our framework applies include polynomials, permutation-invariant transformers, and several neural network architectures defined for sets, point clouds, and graphs of all sizes.
We end with a few directions for future work.
\begin{enumerate}[align=left, font=\textbf]

    \item[(Distributions of species samples)] Can we characterize the collection of sequences $(\mathrm{Law}(\msf E_k(X)))_k$ for random $X$ in terms of the relations of these distributions across dimensions, analogously to Theorems~\ref{thm:proj_consistent_laws} and~\ref{thm:eqp_consistent_laws}? These sequences of distributions include random partitions of integers~\cite{pitman1995exchangeable,kingman1,kingman2,rodriguez2013nonparametric} and edge-exchangeable random graph models~\cite{janson2018edge,levin2025graphs}.

    \item[(Other sampling maps)] While we focused on three particular notions of sampling in this paper, there are other notions that might be appropriate for different applications. For example, in the context of natural language sentences, which are not permutation-invariant, is there a different notion of sampling that can summarize long sentences by short ones?
    
    \item[(Set-based summaries)] In this paper, we consider summarizing objects using random sampling, and comparing these random summaries in Wasserstein distance. Another type of summary studied in the context of certain graph limits involves summarizing an object by forming sets consisting of all possible projections of it, suitably defined, and comparing these sets in Hausdorff distance, see~\cite[Chap.~12]{lovasz2012large} and~\cite{graphop_transfer} for example. Can we develop a general framework and rates for such summaries?  
    
\end{enumerate}

\section*{Acknowledgements}
The authors were supported in part by AFOSR grant FA9550-23-1-0070 and by NSF grant DMS-2502377. Some ideas in the proofs of Proposition~\ref{prop:equivalence_converse}(1), Corollary~\ref{cor:transformers}, Theorem~\ref{thm:W1_rates_FS}, and Proposition~\ref{prop:relation_to_species} were suggested by ChatGPT, which was also used to proofread the manuscript.
This work was conducted while EL was at the department of Computing and Mathematical Sciences at Caltech.
\bibliographystyle{unsrt}
\bibliography{free_cvx_refs}

@article{convergent_seqs1,
title = {Convergent sequences of dense graphs {I}: Subgraph frequencies, metric properties and testing},
journal = {Advances in Mathematics},
volume = {219},
number = {6},
pages = {1801-1851},
year = {2008},
issn = {0001-8708},
doi = {https://doi.org/10.1016/j.aim.2008.07.008},
url = {https://www.sciencedirect.com/science/article/pii/S0001870808002053},
author = {C. Borgs and J.T. Chayes and L. Lovász and V.T. Sós and K. Vesztergombi}
}

@book{lovasz2012large,
  title={Large networks and graph limits},
  author={Lov{\'a}sz, L{\'a}szl{\'o}},
  volume={60},
  year={2012},
  publisher={American Mathematical Soc.}
}

@inproceedings{
maron2018invariant,
title={Invariant and Equivariant Graph Networks},
author={Haggai Maron and Heli Ben-Hamu and Nadav Shamir and Yaron Lipman},
booktitle={International Conference on Learning Representations},
year={2019},
url={https://openreview.net/forum?id=Syx72jC9tm},
}

@article{levin2025deFin,
  title={Any-Dimensional Polynomial Optimization via de {F}inetti Theorems},
  author={Levin, Eitan and Chandrasekaran, Venkat},
  journal={arXiv preprint arXiv:2507.15632},
  year={2025}
}

@article{sampling_tv_dist,
author = {Stam, A. J.},
title = {Distance between sampling with and without replacement},
journal = {Statistica Neerlandica},
volume = {32},
number = {2},
pages = {81-91},
doi = {https://doi.org/10.1111/j.1467-9574.1978.tb01387.x},
url = {https://onlinelibrary.wiley.com/doi/abs/10.1111/j.1467-9574.1978.tb01387.x},
eprint = {https://onlinelibrary.wiley.com/doi/pdf/10.1111/j.1467-9574.1978.tb01387.x},
year = {1978}
}

@techreport{cardaliaguet2010notes,
  title={Notes on mean field games},
  author={Cardaliaguet, Pierre},
  year={2010}
}

@book{Stanley_Fomin_1999, place={Cambridge}, series={Cambridge Studies in Advanced Mathematics}, title={Enumerative Combinatorics vol. 2}, publisher={Cambridge University Press}, author={Stanley, Richard P.}, year={1999}, collection={Cambridge Studies in Advanced Mathematics}}

@article{diaconis_freedman,
author = {P. Diaconis and D. Freedman},
title = {{Finite Exchangeable Sequences}},
volume = {8},
journal = {The Annals of Probability},
number = {4},
publisher = {Institute of Mathematical Statistics},
pages = {745 -- 764},
year = {1980},
doi = {10.1214/aop/1176994663},
URL = {https://doi.org/10.1214/aop/1176994663}
}

@book{macdonald1998symmetric,
  title={Symmetric functions and {H}all polynomials},
  author={Macdonald, Ian Grant},
  year={1998},
  publisher={Oxford university press}
}

@article{hewitt_savage,
 ISSN = {00029947, 10886850},
 URL = {http://www.jstor.org/stable/1992999},
 author = {Edwin Hewitt and Leonard J. Savage},
 journal = {Transactions of the American Mathematical Society},
 number = {2},
 pages = {470--501},
 publisher = {American Mathematical Society},
 title = {Symmetric Measures on {C}artesian Products},
 urldate = {2025-05-29},
 volume = {80},
 year = {1955}
}

@article{dynkin1953classes,
  title={Classes of equivalent random quantities},
  author={Dynkin, Evgenii Borisovich},
  journal={Uspekhi Matematicheskikh Nauk},
  volume={8},
  number={2},
  pages={125--130},
  year={1953},
  publisher={Russian Academy of Sciences, Steklov Mathematical Institute of Russian~…}
}

@article{diaconis2007graph,
  title={Graph limits and exchangeable random graphs},
  author={Diaconis, Persi and Janson, Svante},
  journal={arXiv preprint arXiv:0712.2749},
  year={2007}
}

@article{levin2025transferring,
  title={On Transferring Transferability: Towards a Theory for Size Generalization},
  author={Levin, Eitan and Ma, Yuxin and D{\'\i}az, Mateo and Villar, Soledad},
  journal={arXiv preprint arXiv:2505.23599},
  year={2025}
}

@article{levin2025graphs,
  title={Limits of Weighted Graphs via Random Quotients},
  author={Levin, Eitan and Chandrasekaran, Venkat},
  journal={arXiv preprint arXiv:2512.23149},
  year={2025}
}

@article{lovasz2007szemeredi,
  title={Szemer{\'e}di’s lemma for the analyst},
  author={Lov{\'a}sz, L{\'a}szl{\'o} and Szegedy, Bal{\'a}zs},
  journal={GAFA Geometric And Functional Analysis},
  volume={17},
  number={1},
  pages={252--270},
  year={2007},
  publisher={Springer}
}

@article{fournier2023convergence,
  title={Convergence of the empirical measure in expected {Wasserstein} distance: non-asymptotic explicit bounds in {$\mathbb{R}^d$}},
  author={Fournier, Nicolas},
  journal={ESAIM: Probability and Statistics},
  volume={27},
  pages={749--775},
  year={2023},
  publisher={EDP Sciences}
}

@book{wainwright2019high,
  title={High-dimensional statistics: A non-asymptotic viewpoint},
  author={Wainwright, Martin J},
  volume={48},
  year={2019},
  publisher={Cambridge university press}
}

@book{kallenberg1997foundations,
  title={Foundations of modern probability},
  author={Kallenberg, Olav},
  year={1997},
  publisher={Springer}
}

@article{janson2018edge,
  title={On edge exchangeable random graphs},
  author={Janson, Svante},
  journal={Journal of Statistical Physics},
  volume={173},
  number={3},
  pages={448--484},
  year={2018},
  publisher={Springer}
}

@article{LOVASZ2006933,
title = {Limits of dense graph sequences},
journal = {Journal of Combinatorial Theory, Series B},
volume = {96},
number = {6},
pages = {933-957},
year = {2006},
issn = {0095-8956},
doi = {https://doi.org/10.1016/j.jctb.2006.05.002},
url = {https://www.sciencedirect.com/science/article/pii/S0095895606000517},
author = {László Lovász and Balázs Szegedy},
}

@article{luxburg2004distance,
  title={Distance-based classification with {L}ipschitz functions},
  author={Luxburg, Ulrike von and Bousquet, Olivier},
  journal={Journal of Machine Learning Research},
  volume={5},
  number={Jun},
  pages={669--695},
  year={2004}
}

@article{shi2009hash,
  title={Hash kernels for structured data},
  author={Shi, Qinfeng and Petterson, James and Dror, Gideon and Langford, John and Smola, Alex and Vishwanathan, S. V. N.},
  journal={Journal of Machine Learning Research},
  volume={10},
  number={11},
  year={2009}
}

@article{CORMODE200558,
title = {An improved data stream summary: the count-min sketch and its applications},
journal = {Journal of Algorithms},
volume = {55},
number = {1},
pages = {58-75},
year = {2005},
issn = {0196-6774},
doi = {https://doi.org/10.1016/j.jalgor.2003.12.001},
url = {https://www.sciencedirect.com/science/article/pii/S0196677403001913},
author = {Graham Cormode and S. Muthukrishnan}
}

@article{yassaee2014achievability,
  title={Achievability proof via output statistics of random binning},
  author={Yassaee, Mohammad Hossein and Aref, Mohammad Reza and Gohari, Amin},
  journal={IEEE Transactions on Information Theory},
  volume={60},
  number={11},
  pages={6760--6786},
  year={2014},
  publisher={IEEE}
}

@article{kingman1,
    author = {Kingman, J. F. C.},
    title = {The Representation of Partition Structures},
    journal = {Journal of the London Mathematical Society},
    volume = {s2-18},
    number = {2},
    pages = {374-380},
    year = {1978},
    month = {10},
    issn = {0024-6107},
    doi = {10.1112/jlms/s2-18.2.374},
    url = {https://doi.org/10.1112/jlms/s2-18.2.374},
    eprint = {https://academic.oup.com/jlms/article-pdf/s2-18/2/374/2788610/s2-18-2-374.pdf},
}

@article{kingman2,
    author = {Kingman, John Frank Charles},
    title = {Random partitions in population genetics},
    journal = {Proceedings of the Royal Society of London. A. Mathematical and Physical Sciences},
    volume = {361},
    number = {1704},
    pages = {1-20},
    year = {1978},
    month = {05},
    issn = {0080-4630},
    doi = {10.1098/rspa.1978.0089},
    url = {https://doi.org/10.1098/rspa.1978.0089},
    eprint = {https://royalsocietypublishing.org/rspa/article-pdf/361/1704/1/62915/rspa.1978.0089.pdf},
}

@article{pitman1995exchangeable,
  title={Exchangeable and partially exchangeable random partitions},
  author={Pitman, Jim},
  journal={Probability Theory and Related Fields},
  volume={102},
  number={2},
  pages={145--158},
  year={1995},
  publisher={Springer}
}

@article{probability_graphons,
  TITLE = {{Probability-graphons: Limits of large dense weighted graphs}},
  AUTHOR = {Abraham, Romain and Delmas, Jean-Fran{\c c}ois and Weibel, Julien},
  URL = {https://hal.science/hal-04361443},
  JOURNAL = {{Innovations in Graph Theory}},
  PUBLISHER = {{Dutch Science Council (NWO)}},
  VOLUME = {2},
  PAGES = {25--117},
  YEAR = {2025},
  MONTH = Mar,
  DOI = {10.5802/igt.7},
  HAL_ID = {hal-04361443},
  HAL_VERSION = {v2},
}

@article{furuya2026approximation,
  title={Approximation Theory for {L}ipschitz Continuous Transformers},
  author={Furuya, Takashi and Murari, Davide and Sch{\"o}nlieb, Carola-Bibiane},
  journal={arXiv preprint arXiv:2602.15503},
  year={2026}
}

@misc{dudley_bound,
  title        = {Theoretical Statistics, Lecture 14},
  author       = {Peter Bartlett},
  howpublished = {\url{https://www.stat.berkeley.edu/~bartlett/courses/2013spring-stat210b/notes/14notes.pdf}},
  year         = {2013}
}

@article{kolmogorov1959varepsilon,
  title={$\varepsilon$-entropy and $\varepsilon$-capacity of sets in function spaces},
  author={Kolmogorov, Andrei Nikolaevich and Tikhomirov, Vladimir Mikhailovich},
  journal={Uspekhi Matematicheskikh Nauk},
  volume={14},
  number={2},
  pages={3--86},
  year={1959},
  publisher={Russian Academy of Sciences, Steklov Mathematical Institute of Russian~…}
}

@inproceedings{deepsets,
 author = {Zaheer, Manzil and Kottur, Satwik and Ravanbakhsh, Siamak and Poczos, Barnabas and Salakhutdinov, Russ R and Smola, Alexander J},
 booktitle = {Advances in Neural Information Processing Systems},
 editor = {I. Guyon and U. Von Luxburg and S. Bengio and H. Wallach and R. Fergus and S. Vishwanathan and R. Garnett},
 pages = {},
 publisher = {Curran Associates, Inc.},
 title = {Deep Sets},
 url = {https://proceedings.neurips.cc/paper_files/paper/2017/file/f22e4747da1aa27e363d86d40ff442fe-Paper.pdf},
 volume = {30},
 year = {2017}
}

@book{Niu_Spivak_2025, 
place={Cambridge}, 
series={London Mathematical Society Lecture Note Series}, 
title={Polynomial Functors: A Mathematical Theory of Interaction}, 
publisher={Cambridge University Press}, 
author={Niu, Nelson and Spivak, David I.}, 
year={2025}, 
collection={London Mathematical Society Lecture Note Series}
}

@article{bueno2021representation,
  title={On the representation power of set pooling networks},
  author={Bueno, Christian and Hylton, Alan},
  journal={Advances in Neural Information Processing Systems},
  volume={34},
  pages={17170--17182},
  year={2021}
}

@inproceedings{
furuya2025transformers,
title={Transformers are Universal In-context Learners},
author={Takashi Furuya and Maarten V. de Hoop and Gabriel Peyr{\'e}},
booktitle={The Thirteenth International Conference on Learning Representations},
year={2025},
url={https://openreview.net/forum?id=6S4WQD1LZR}
}

@article{multisym_funcs,
     author = {Vaccarino, Francesco},
     title = {The ring of multisymmetric functions},
     journal = {Annales de l'Institut Fourier},
     pages = {717--731},
     year = {2005},
     publisher = {Association des Annales de l'Institut Fourier},
     volume = {55},
     number = {3},
     doi = {10.5802/aif.2111},
     mrnumber = {2149400},
     zbl = {1062.05143},
     language = {en},
     url = {https://www.numdam.org/articles/10.5802/aif.2111/}
}

@article{Zhang02072016,
author = {Zhiyi Zhang and Michael Grabchak},
title = {Entropic representation and estimation of diversity indices},
journal = {Journal of Nonparametric Statistics},
volume = {28},
number = {3},
pages = {563--575},
year = {2016},
publisher = {Taylor \& Francis},
doi = {10.1080/10485252.2016.1190357},
URL = {https://doi.org/10.1080/10485252.2016.1190357},
eprint = {https://doi.org/10.1080/10485252.2016.1190357}
}

@inproceedings{rodriguez2013nonparametric,
  title={Nonparametric bayesian inference},
  author={Rodriguez, Abel and M{\"u}ller, Peter},
  booktitle={NSF-CBMS Regional Conference Series in Probability and Statistics},
  volume={9},
  pages={i--110},
  year={2013},
  organization={JSTOR}
}

@book{vapnik1999nature,
  title={The Nature of Statistical Learning Theory},
  author={Vapnik, Vladimir},
  year={1999},
  publisher={Springer},
  doi={10.1007/978-1-4757-3264-1}
}

@InProceedings{pointnet,
author = {Qi, Charles R. and Su, Hao and Mo, Kaichun and Guibas, Leonidas J.},
title = {{PointNet}: Deep Learning on Point Sets for {3D} Classification and Segmentation},
booktitle = {Proceedings of the IEEE Conference on Computer Vision and Pattern Recognition (CVPR)},
month = {July},
year = {2017}
}

@ARTICLE{GNNs,
  author={Scarselli, Franco and Gori, Marco and Tsoi, Ah Chung and Hagenbuchner, Markus and Monfardini, Gabriele},
  journal={IEEE Transactions on Neural Networks}, 
  title={The Graph Neural Network Model}, 
  year={2009},
  volume={20},
  number={1},
  pages={61-80},
  doi={10.1109/TNN.2008.2005605}}

@inproceedings{transferab1,
 author = {Ruiz, Luana and Chamon, Luiz and Ribeiro, Alejandro},
 booktitle = {Advances in Neural Information Processing Systems},
 editor = {H. Larochelle and M. Ranzato and R. Hadsell and M.F. Balcan and H. Lin},
 pages = {1702--1712},
 publisher = {Curran Associates, Inc.},
 title = {Graphon Neural Networks and the Transferability of Graph Neural Networks},
 url = {https://proceedings.neurips.cc/paper_files/paper/2020/file/12bcd658ef0a540cabc36cdf2b1046fd-Paper.pdf},
 volume = {33},
 year = {2020}
}

@article{maskey2023transferability,
  title={Transferability of graph neural networks: an extended graphon approach},
  author={Maskey, Sohir and Levie, Ron and Kutyniok, Gitta},
  journal={Applied and Computational Harmonic Analysis},
  volume={63},
  pages={48--83},
  year={2023},
  publisher={Elsevier}
}

@inproceedings{transferb1,
 author = {Maskey, Sohir and Levie, Ron and Lee, Yunseok and Kutyniok, Gitta},
 booktitle = {Advances in Neural Information Processing Systems},
 editor = {S. Koyejo and S. Mohamed and A. Agarwal and D. Belgrave and K. Cho and A. Oh},
 pages = {4805--4817},
 publisher = {Curran Associates, Inc.},
 title = {Generalization Analysis of Message Passing Neural Networks on Large Random Graphs},
 url = {https://proceedings.neurips.cc/paper_files/paper/2022/file/1eeaae7c89d9484926db6974b6ece564-Paper-Conference.pdf},
 volume = {35},
 year = {2022}
}

@inproceedings{graphop_transfer,
 author = {Le, Thien and Jegelka, Stefanie},
 booktitle = {Advances in Neural Information Processing Systems},
 editor = {A. Oh and T. Naumann and A. Globerson and K. Saenko and M. Hardt and S. Levine},
 pages = {41305--41342},
 publisher = {Curran Associates, Inc.},
 title = {Limits, approximation and size transferability for {GNNs} on sparse graphs via graphops},
 url = {https://proceedings.neurips.cc/paper_files/paper/2023/file/8154c89c8d3612d39fd1ed6a20f4bab1-Paper-Conference.pdf},
 volume = {36},
 year = {2023}
}

@inproceedings{transformers_time_series,
author = {Wen, Qingsong and Zhou, Tian and Zhang, Chaoli and Chen, Weiqi and Ma, Ziqing and Yan, Junchi and Sun, Liang},
title = {Transformers in time series: a survey},
year = {2023},
isbn = {978-1-956792-03-4},
url = {https://doi.org/10.24963/ijcai.2023/759},
doi = {10.24963/ijcai.2023/759},
booktitle = {Proceedings of the Thirty-Second International Joint Conference on Artificial Intelligence},
articleno = {759},
numpages = {9},
location = {Macao, P.R.China},
series = {IJCAI '23}
}

@article{LIN2022111,
title = {A survey of transformers},
journal = {AI Open},
volume = {3},
pages = {111-132},
year = {2022},
issn = {2666-6510},
doi = {https://doi.org/10.1016/j.aiopen.2022.10.001},
url = {https://www.sciencedirect.com/science/article/pii/S2666651022000146},
author = {Tianyang Lin and Yuxin Wang and Xiangyang Liu and Xipeng Qiu}
}

@article{diaz2025invariant,
  title={Invariant kernels: Rank stabilization and generalization across dimensions},
  author={D{\'i}az, Mateo and Drusvyatskiy, Dmitriy and Kendrick, Jack and Thomas, Rekha R},
  journal={arXiv preprint arXiv:2502.01886},
  year={2025}
}

@inproceedings{
zhou2024transformers,
title={Transformers Can Achieve Length Generalization But Not Robustly},
author={Yongchao Zhou and Uri Alon and Xinyun Chen and Xuezhi Wang and Rishabh Agarwal and Denny Zhou},
booktitle={ICLR 2024 Workshop on Mathematical and Empirical Understanding of Foundation Models},
year={2024},
url={https://openreview.net/forum?id=DWkWIh3vFJ}
}

@inproceedings{limit_transformers,
 author = {Huang, Xinting and Yang, Andy and Bhattamishra, Satwik and Sarrof, Yash and Krebs, Andreas and Zhou, Hattie and Nakkiran, Preetum and Hahn, Michael},
 booktitle = {International Conference on Learning Representations},
 editor = {Y. Yue and A. Garg and N. Peng and F. Sha and R. Yu},
 pages = {58095--58179},
 title = {A Formal Framework for Understanding Length Generalization in Transformers},
 url = {https://proceedings.iclr.cc/paper_files/paper/2025/file/928170bcb050fe64a63fe781b82265aa-Paper-Conference.pdf},
 volume = {2025},
 year = {2025}
}

@inproceedings{PE_transformers,
 author = {Kazemnejad, Amirhossein and Padhi, Inkit and Natesan Ramamurthy, Karthikeyan and Das, Payel and Reddy, Siva},
 booktitle = {Advances in Neural Information Processing Systems},
 editor = {A. Oh and T. Naumann and A. Globerson and K. Saenko and M. Hardt and S. Levine},
 pages = {24892--24928},
 publisher = {Curran Associates, Inc.},
 title = {The Impact of Positional Encoding on Length Generalization in Transformers},
 url = {https://proceedings.neurips.cc/paper_files/paper/2023/file/4e85362c02172c0c6567ce593122d31c-Paper-Conference.pdf},
 volume = {36},
 year = {2023}
}

@article{yang2026length,
  title={Length Generalization Bounds for Transformers},
  author={Yang, Andy and Bergstr{\"a}{\ss}er, Pascal and Zetzsche, Georg and Chiang, David and Lin, Anthony W},
  journal={arXiv preprint arXiv:2603.02238},
  year={2026}
}

@article{furuya2026function,
  title={Function graph transformers universally approximate operators between function spaces},
  author={Furuya, Takashi and Mis, David and Dokmani{\'c}, Ivan and de Hoop, Maarten V. and Lassas, Matti},
  journal={arXiv preprint arXiv:2605.17968},
  year={2026}
}

@inproceedings{
Yun2020Are,
title={Are Transformers universal approximators of sequence-to-sequence functions?},
author={Chulhee Yun and Srinadh Bhojanapalli and Ankit Singh Rawat and Sashank Reddi and Sanjiv Kumar},
booktitle={International Conference on Learning Representations},
year={2020},
url={https://openreview.net/forum?id=ByxRM0Ntvr}
}

@ARTICLE{graph_signal_processing,
  author={Ortega, Antonio and Frossard, Pascal and Kovačević, Jelena and Moura, José M. F. and Vandergheynst, Pierre},
  journal={Proceedings of the IEEE}, 
  title={Graph Signal Processing: Overview, Challenges, and Applications}, 
  year={2018},
  volume={106},
  number={5},
  pages={808-828},
  doi={10.1109/JPROC.2018.2820126}
  }

@inproceedings{miniLM,
 author = {Wang, Wenhui and Wei, Furu and Dong, Li and Bao, Hangbo and Yang, Nan and Zhou, Ming},
 booktitle = {Advances in Neural Information Processing Systems},
 editor = {H. Larochelle and M. Ranzato and R. Hadsell and M.F. Balcan and H. Lin},
 pages = {5776--5788},
 publisher = {Curran Associates, Inc.},
 title = {{MiniLM}: Deep Self-Attention Distillation for Task-Agnostic Compression of Pre-Trained Transformers},
 url = {https://proceedings.neurips.cc/paper_files/paper/2020/file/3f5ee243547dee91fbd053c1c4a845aa-Paper.pdf},
 volume = {33},
 year = {2020}
}

@inproceedings{tinybert,
    title = "{T}iny{BERT}: Distilling {BERT} for Natural Language Understanding",
    author = "Jiao, Xiaoqi  and
      Yin, Yichun  and
      Shang, Lifeng  and
      Jiang, Xin  and
      Chen, Xiao  and
      Li, Linlin  and
      Wang, Fang  and
      Liu, Qun",
    editor = "Cohn, Trevor  and
      He, Yulan  and
      Liu, Yang",
    booktitle = "Findings of the Association for Computational Linguistics: EMNLP 2020",
    month = nov,
    year = "2020",
    address = "Online",
    publisher = "Association for Computational Linguistics",
    url = "https://aclanthology.org/2020.findings-emnlp.372/",
    doi = "10.18653/v1/2020.findings-emnlp.372",
    pages = "4163--4174"
}

@article{lavenant2026continuous,
  title={Continuous transformations of probability measures and their transport representations},
  author={Lavenant, Hugo and Savar{\'e}, Giuseppe},
  journal={arXiv preprint arXiv:2604.16653},
  year={2026}
}

@article{bohbot2025token,
  title={Token sample complexity of attention},
  author={Bohbot, L{\'e}a and Letrouit, Cyril and Peyr{\'e}, Gabriel and Vialard, Fran{\c{c}}ois-Xavier},
  journal={arXiv preprint arXiv:2512.10656},
  year={2025}
}



\end{document}